\documentclass[11pt]{amsart}
\usepackage{amscd,amssymb}
\usepackage[arrow,matrix]{xy}
\usepackage[colorlinks,plainpages,backref,urlcolor=blue]{hyperref}

\newtheorem{theorem}[subsection]{Theorem}
\newtheorem{lemma}[subsection]{Lemma}
\newtheorem{prop}[subsection]{Proposition}
\newtheorem{corollary}[subsection]{Corollary}

\newtheorem{thm}{Theorem}

\theoremstyle{definition}
\newtheorem{remark}[subsection]{Remark}
\newtheorem{definition}[subsection]{Definition}
\newtheorem{example}[subsection]{Example}
\newtheorem*{ack}{Acknowledgment}

\numberwithin{equation}{section}
\newcommand{\cP}{{\mathcal P}}

\newcommand{\M}{{\mathcal M}}
\newcommand{\F}{{\mathcal F}}
\newcommand{\A}{{\mathcal A}}
\newcommand{\OO}{{\mathcal O}}
\newcommand{\B}{{\mathcal B}}

\newcommand{\CC}{{\mathcal C}}
\newcommand{\LL}{{\mathcal L}}
\newcommand{\E}{{\mathcal E}}
\newcommand{\V}{{\mathcal V}}
\newcommand{\R}{{\mathcal R}}

\newcommand{\cH}{{\mathcal H}}
\newcommand{\cN}{{\mathcal N}}

\newcommand{\cI}{{\mathcal I}}

\newcommand{\m}{\mathfrak {m}}
  
\newcommand{\bb}{\mathfrak {b}}
\newcommand{\gl}{\mathfrak {gl}}

\newcommand{\fs}{\mathfrak {s}}
\newcommand{\fg}{\mathfrak {g}}
\newcommand{\fp}{\mathfrak {p}}
\newcommand{\fq}{\mathfrak {q}}

\newcommand{\wR}{\widehat{R}}
\newcommand{\woR}{\widehat{\overline{R}}}

\newcommand{\oR}{\overline{R}}
\newcommand{\oP}{\overline{P}}
\newcommand{\ow}{\overline{\omega}}

\newcommand{\sCDGA}{{\sf CDGA}}
\newcommand{\sCGA}{{\sf CGA}}
\newcommand{\sDGL}{{\sf DGL}}
\newcommand{\sSS}{{\sf SS}}
\newcommand{\sArt}{{\sf Art}}
\newcommand{\sAlg}{{\sf Alg}}
\newcommand{\sLie}{{\sf Lie}}
\newcommand{\sMin}{{\sf Min}}
\newcommand{\sMod}{{\sf Mod}}
\newcommand{\sTop}{{\sf Top}}
\newcommand{\sComm}{{\sf Comm}}
\newcommand{\sACDGA}{{\sf ACDGA}}
\newcommand{\kk}{\kappa}
\newcommand{\BB}{\mathbb{B}}
\newcommand{\T}{\mathbb{T}}
\newcommand{\Z}{\mathbb{Z}}
\newcommand{\N}{\mathbb{N}}
\newcommand{\Q}{\mathbb{Q}}
\newcommand{\RR}{\mathbb{R}}
\newcommand{\C}{\mathbb{C}}
\newcommand{\HH}{\mathbb{H}}
\newcommand{\K}{\Bbbk}
\newcommand{\PP}{\mathbb{P}}

\newcommand{\bK}{\mathbb{K}}

\newcommand{\bS}{\mathbb{S}}
\newcommand{\bG}{\mathbb{G}}
\newcommand{\bA}{\mathbb{A}}
\newcommand{\eexp}{\mathsf{exp}}

\newcommand{\tA}{\widetilde{{\mathcal A}}}

\DeclareMathOperator{\Hom}{Hom}

\DeclareMathOperator{\id}{id}
\DeclareMathOperator{\pr}{pr}

\DeclareMathOperator{\gr}{gr}
\DeclareMathOperator{\GL}{GL}
\DeclareMathOperator{\SL}{SL}
\DeclareMathOperator{\ab}{ab}

\DeclareMathOperator{\ad}{ad}
\DeclareMathOperator{\Ad}{Ad}

\DeclareMathOperator{\Sym}{Sym}

\DeclareMathOperator{\Sing}{Sing}
\DeclareMathOperator{\loc}{loc}

\DeclareMathOperator{\Aut}{Aut}
\DeclareMathOperator{\End}{End}
\DeclareMathOperator{\Ob}{Ob}
\DeclareMathOperator{\alggr}{alggr}
\DeclareMathOperator{\mon}{mon}
\DeclareMathOperator{\alg}{alg}
\DeclareMathOperator{\diff}{diff}
\DeclareMathOperator{\pt}{pt}
\DeclareMathOperator{\Mal}{Mal}

\DeclareMathOperator{\red}{red}
\DeclareMathOperator{\aom}{Aom}

\newcommand{\surj}{\twoheadrightarrow}
\newcommand{\isom}{\xrightarrow{\,\simeq\,}}
\newcommand{\eqv}{{\Longleftrightarrow}}

\begin{document}

\title[Cohomology jump loci from an analytic viewpoint]
{Nonabelian cohomology jump loci from an analytic viewpoint}

\author[A.~Dimca]{Alexandru Dimca$^{1}$}
\address{Institut Universitaire de France et Laboratoire J.A. Dieudonn\'e, UMR du CNRS 7351,
                 Universit\'e de Nice Sophia-Antipolis,
                 Parc Valrose,
                 06108 Nice Cedex 02,
                 France}

\email{dimca@unice.fr}

\author[\c{S}.~Papadima ]{\c{S}tefan Papadima$^{2}$}
\address{Simion Stoilow Institute of Mathematics,
P.O. Box 1-764,
RO-014700 Bucharest, Romania}
\email{Stefan.Papadima@imar.ro}

\thanks{$^1$Partially supported by the  ANR-08-BLAN-0317-02 (SEDIGA)}
\thanks{$^2$Partially supported by PN-II-ID-PCE-2011-3-0288, grant 
132/05.10.2011}



\begin{abstract}
For a space, we investigate its CJL (cohomology jump loci), sitting inside varieties of 
representations of the fundamental group. To do this, for a CDG (commutative differential graded) algebra, 
we define its CJL,
sitting inside varieties of  flat connections. The analytic germs at the origin
$1$ of representation varieties are shown to be determined by the Sullivan $1$-minimal model of the space.
Up to a degree $q$, the two types of CJL have 
the same analytic germs at the origins, when the space and the algebra have the same $q$-minimal model.
We apply this general approach to formal spaces (obtaining the degeneration of the
Farber-Novikov spectral sequence), quasi-projective manifolds, and finitely generated nilpotent groups.
When the CDG algebra has positive weights, we elucidate some of the structure of (rank one complex) 
topological and algebraic CJL:
all their irreducible components passing through the origin are connected affine subtori,
respectively rational linear subspaces. Furthermore, the global exponential map sends all algebraic
CJL into their topological counterpart.
\end{abstract}

\maketitle

\tableofcontents

\section{Introduction and statement of main results} 
\label{sec:intro}

Representation varieties of the fundamental group into linear algebraic groups parametrize (types of) finite rank
local systems on a space. We tackle the basic question of computing the corresponding twisted Betti numbers, i.e.,
describing the so-called jump subvarieties. Our main results give both local answers, concerning homomorphisms 
near the trivial representation, and global information on these jump loci. We approach the problem via flat connections 
and algebraic jump loci coming from the associated  covariant derivative.

\subsection{Relative jump loci}
\label{ss01}

We work over the field $\K =\RR$ or $\C$. We consider, up to homotopy,   a connected CW complex $X$,
and we want to analyze twisted Betti numbers of $X$ up to a fixed degree $q$ ($1\le q \le \infty$). We also fix a
homomorphism of $\K$--linear algebraic groups, $\iota : \BB \to \GL (V)$, with derivative denoted
$\theta : \bb \to \gl (V)$. Setting $G= \pi_1 (X)$, let $\cH (G, \BB)$ be the set of group homomorphisms,
$\Hom_{\gr} (G, \BB)$. We view $\rho \in \cH (G, \BB)$ as a rank $\ell$ local system on $X$ ($\ell = \dim_{\K} V$)
factoring through $\iota$, denoted ${}_{\iota \rho}V$. When $\iota$ is an inclusion, this allows us to treat simultaneously
various types of local systems (general, volume--preserving, unitary, unipotent...). Following \cite{DPS}, we
are interested in the {\em relative characteristic varieties} $\V^i_r (X, \iota)$, defined for $i,r \ge 0$ by
\begin{equation}
\label{eq=char0}
\V^i_r (X, \iota)=\{\rho \in \cH (G, \BB) \mid 
\dim_{\K} H^i(X, {}_{\iota\rho}V)\ge r\} \, .
\end{equation}

When $X$ has finite $q$--skeleton $X^{(q)}$ (where $X^{(\infty)}=X$), it turns out that both 
$\V^i_0 (X, \iota)= \cH (G, \BB)$ and $\V^i_r (X, \iota)$ are affine varieties, for all $i\le q$ and $r$. 
Their analytic germs at the trivial representation $1\in \cH (G, \BB)$ will be denoted by the subscript
$(\cdot)_{(1)}$. 

Our goal is to replace $X$ by a commutative differential graded algebra ($\sCDGA$) $\A^{\bullet}$. The
algebraic analog of $\cH (G, \BB)$ is the set of {\em flat connections}, 
\[
\F (\A, \bb)= \{ \omega \in \A^1 \otimes \bb \mid d\omega + \frac{1}{2} [\omega, \omega]=0 \} \, ,
\]
whose construction is given in detail in Section \ref{sec:prel}. When $\omega \in \A^1 \otimes \bb$ is flat,
the {\em covariant derivative} $d_{\omega}= d+ \ad_{\omega}$ (depending on $\theta$) satisfies $d_{\omega}^2=0$, giving
thus rise to the {\em Aomoto cochain complex}, 
\[
(\A^{\bullet} \otimes V, d_{\omega})\, ;
\] 
we refer to Section \ref{sec:prel} for full details. The analog of \eqref{eq=char0} is provided by the
{\em relative resonance varieties} $\R^i_r (\A, \theta)$, defined for $i,r \ge 0$ by
\begin{equation}
\label{eq=res0}
\R^i_r (\A, \theta)=\{\omega \in \F (\A, \bb) \mid 
\dim_{\K} H^i(\A^{\bullet} \otimes V, d_{\omega}) \ge r\} \, .
\end{equation}

This definition extends the setup from \cite{DPS}, where only $\sCDGA$'s with trivial differential ($d=0$)
were considered. When $\A^{\bullet}$ is connected (i.e., $\A^0= \K \cdot 1$) and
$\dim_{\K} \A^{\le q}< \infty$ (where $\A^{\le q}$ denotes $\oplus_{i\le q} \A^i$),  it turns out that both 
$\R^i_0 (\A, \theta)= \F (\A, \bb)$ and $\R^i_r (\A, \theta)$ are affine varieties, for all $i\le q$ and $r$. 
Their analytic germs at the trivial connection $0\in  \F (\A, \bb)$ will be denoted by the subscript
$(\cdot)_{(0)}$. 

\subsection{Rational homotopy theory}
\label{ss02}

The space $X$ and the $\sCDGA$  $\A^{\bullet}$ may be related via {\em homotopical algebra}, in the sense of
D. Sullivan \cite{S}. The starting point is the notion of {\em $q$--equivalence}: a $\sCDGA$ map inducing in
cohomology an isomorphism up to degree $q$, and a monomorphism in degree $q+1$. We say that two $\sCDGA$'s
have the {\em same $q$--type} (notation: $\A \simeq_q \B$) if they can be connected by a zigzag of  $q$--equivalences.
In particular, $\A^{\bullet}$ is {\em $q$--formal} (in the sense from \cite{Mac} and \cite{DP}) if
$(\A^{\bullet}, d) \simeq_q (H^{\bullet} \A, d=0)$. For $q=\infty$, we recover Sullivan's celebrated notion of {\em formality}.

For a space $X$, we denote by $\Omega^{\bullet}(X, \K)$ Sullivan's $\sCDGA$ of piecewise $C^{\infty}$ $\K$--forms on $X$,
with cohomology algebra $H^{\bullet}(X, \K)$, the untwisted singular cohomology ring. When $X$ is path--connected, it is known that
$\Omega^{\bullet}(X, \K) \simeq_q \A^{\bullet}$ if and only if $X$ and $\A^{\bullet}$ have the same {\em $q$--minimal model}.
In particular, we may speak about $q$--formal spaces, and groups (by considering $X=K(G, 1)$).

\subsection{Local results}
\label{ss03}

We aim at a simultaneous extension of \cite{GM} and \cite{DPS}. In \cite{GM}, Goldman and Millson analyze analytic germs 
of representation varieties, via the deformation theory of flat connections on differentiable manifolds. We will extend this (at $1$) to
topological spaces, taking also into account the characteristic varieties \eqref{eq=char0}. In \cite{DPS}, analytic isomorphisms
$\V^1_r (X, \iota)_{(1)} \cong \R^1_r (\A, \theta)_{(0)}$ are obtained for all $r\ge 0$, assuming $X$ to be $1$--formal with
$X^{(1)}$ finite and taking $\A^{\bullet}= (H^{\bullet}(X, \C), d=0)$. We will extend this beyond formality, and in arbitrary degree.

The result below is proved in Theorem \ref{thm:gralkm}. In the particular case when the group $G$ is $1$--formal, it was obtained 
by Kapovich and Millson in \cite{KM}, based on \cite{GM}. 

\begin{thm}
\label{thm:km0}
Let $G$ be a finitely generated group and $\BB$ be a $\K$--linear algebraic group. Then 
the analytic germ $\cH (G, \BB)_{(1)}$ depends only on the $1$--minimal model of $G$ and the Lie algebra of $\BB$.
\end{thm}

Our second main local result offers a concrete description of $\cH (G, \BB)_{(1)}$, compatible with the jump loci
\eqref{eq=char0} and \eqref{eq=res0}.

\begin{thm}
\label{thm:8main}
Let  $\iota : \BB \to \GL (V)$ be a rational representation of $\K$--linear algebraic groups ($\K =\RR$ or $\C$), with
tangent Lie representation $\theta : \bb \to \gl (V)$. Let $X$ be a connected CW complex with finite $q$--skeleton 
($1\le q\le \infty$), up to homotopy, and let $\A^{\bullet}$ be a connected $\K$--$\sCDGA$ with $\dim_{\K} \A^{\le q} <\infty$.
If $\Omega^{\bullet}(X, \K) \simeq_q \A^{\bullet}$, the following hold.
\begin{enumerate}

\item \label{8m1}
There is an isomorphism of reduced analytic germs, 
$e: \F (\A, \bb)_{(0)} \stackrel{\simeq}{\rightarrow} \cH (\pi_1(X), \BB)_{(1)} $,
inducing a local analytic isomorphism, $e: \R^i_r (\A, \theta)_{(0)} \stackrel{\simeq}{\rightarrow} \V^i_r (X, \iota)_{(1)}$,
for all $i\le q$ and $r\ge 0$. 

\item \label{8m2}
When $\BB$ is abelian, there is a natural identification, $\F (\A, \bb) \cong \Hom_{\gr} (\pi_1(X), \bb)$, and $e$ is induced by
the analytic map $\eexp_{\BB} : \Hom_{\gr} (\pi_1(X), \bb) \rightarrow \Hom_{\gr} (\pi_1(X), \BB)$, where
$\eexp_{\BB} : \bb \to \BB$ is the exponential map.

\end{enumerate}
\end{thm}

\subsection{Examples}
\label{ss04}
(cf. Section \ref{sec:ex} and \S \ref{ss84})

Theorem \ref{thm:8main} applies in particular to a connected CW complex $X$ with $X^{(q)}$ finite (up to homotopy)
that is $q$--formal (over $\K$), by taking $\A^{\bullet}= (H^{\bullet}(X, \K), d=0)$; the case $q=1$ was treated in
\cite[Theorem A]{DPS}.

Another particular case of considerable interest covered by Theorem \ref{thm:8main} arises when $X$ is a 
connected, smooth, quasi--projective complex variety (for short, a {\em quasi--projective manifold}). Here, $q=\infty$
and we may take $\A^{\bullet}$ to be the {\em Gysin model} constructed by Morgan in \cite{M}, by using a compactification of $X$.
In the {\em rank one} case, when $\BB= \GL_1 (\C)$ and $\iota =\id_{\C^{\times}}$, the characteristic varieties
(or, Green--Lazarsfeld sets of a space $X$) are usually denoted $\V^i_r (X)$; they were introduced for projective manifolds by 
Green--Lazarsfeld in \cite{GL}, and reinterpreted topologically by Beauville in \cite{Be}. Green--Lazarsfeld sets of
quasi--projective manifolds received a lot of attention, over the years; see e.g. \cite{DPS} and the references therefrom.
Our result from Theorem \ref{thm:8main}, in the quasi--projective case, may be viewed as a substantial nonabelian extension
of previously known facts about Green--Lazarsfeld sets.

Nonabelian cohomology jump loci of an arbitrary finitely generated nilpotent group $G$ may also be handled by Theorem \ref{thm:8main}.
Here, $q=\infty$ and $\A^{\bullet}$ is the $1$--minimal model of $G$; see Corollary \ref{cor:nilgerms}. Another interesting class
of polycyclic groups to which the approach from Theorem \ref{thm:8main} applies consists of the fundamental groups of
Hattori's solvmanifolds from \cite{Hat}, where $q=\infty$ and $\A^{\bullet}$ is the $\sCDGA$ of cochains on a certain solvable 
Lie algebra; see Corollary \ref{cor:solvgerms}.

To the best of our knowledge, only sporadic computations of twisted cohomology for polycyclic groups exist in the literature.
For example, unipotent representations of finitely generated, torsion--free nilpotent groups were studied in \cite{P};
a similar analysis was recently carried out in \cite{Kas}, for lattices in $1$--connected solvable Lie groups;
$\SL_2$--jump loci for a family of nilmanifolds were described in \cite{Al}; the finiteness of characteristic varieties, for
the solvmanifolds we consider in Corollary \ref{cor:solvgerms}, was established in the rank one case in \cite{Mil}; for
arbitrary finitely generated nilpotent groups, the triviality of rank one characteristic varieties was obtained in \cite{MaP}.

\subsection{Positive weights}
\label{ss05}

In our next three results, we illustrate the fact that Theorem \ref{thm:8main} sheds a new light even in abelian cases.
The presence of positive weights (in the sense from Definition \ref{def:posw}) will be a unifying theme.

The notion of {\em positive--weight decomposition} for a $\Q$--$\sCDGA$ $(\A^{\bullet}, d)$ goes back 
to Body--Sullivan \cite{BS} and Morgan \cite{M}. This restrictive property means the existence of a 
grading (required to be positive in degree $1$), called weight, preserved by both multiplication and differential. 

When $\Omega^{\bullet}(X, \K) \simeq_1 \A^{\bullet}$, the weight decomposition induces a grading, 
$H^1(X,\K)=\oplus_{j \in J}H^1_j$, on the first cohomology group. 
We say that a subspace $E \subseteq H^1(X,\K)$ is {\it weighted homogeneous}
with respect to this grading if $E=\oplus_{j \in J}(E\cap H^1_j).$

A rich source of examples is provided by $\sCDGA$'s of the form $\A^{\bullet}= (H^{\bullet}(X, \K), d=0)$, with
weight equal to degree. Gysin models of quasi--projective manifolds also have positive weights, as we recall in
Example \ref{ex:qproj}. The corresponding grading, $H^1(X,\Q)=H_1^1 \oplus H_2^1$, has the property that 
$H_1^1=W_1H^1(X,\Q)$, i.e., it is a rational splitting of the weight filtration of the Deligne MHS on $H^1(X,\Q)$, 
see $(4)$ in \cite{Dcompo}.

Our viewpoint in the next three (global) results is that the existence of positive weights 
is a good substitute, in a purely topological context, for the Deligne mixed Hodge structure on the cohomology of algebraic varieties.
(Needless to say, the positive--weight decomposition contains weaker information than the MHS.) 
See also Hain and Matsumoto \cite{HM}.

\subsection{Global results}
\label{ss06}

Our first global result clarifies the qualitative structure of rank one {\em non--translated} components (i.e., irreducible
components passing through the origin), for both characteristic and resonance varieties, in the presence of positive weights.

\begin{thm}
\label{thm:tori}
In Theorem \ref{thm:8main}, suppose moreover that $\iota =\id_{\C^{\times}}$, the $\sCDGA$ $\A$ is defined over $\Q$
and has positive weights, and the isomorphism induced by the zigzag of $q$--equivalences, 
$H^1(X, \C)  \stackrel{\sim}{\leftrightarrow} H^1 \A$, preserves $\Q$--structures. Then:
\begin{enumerate}

\item \label{8t1}
For all $i\le q$ and $r\ge 0$, $\R^i_r (\A, \theta)$ is a finite union of linear subspaces of $\F (\A, \C)= H^1 \A$, that are
defined over $\Q$ and weighted homogeneous.

\item \label{8t2}
For all $i\le q$ and $r\ge 0$, every irreducible component of $\V^i_r (X)$ passing through $1$ is an affine subtorus
of the character torus $\T (\pi_1(X))=\cH  (\pi_1(X), \C^{\times})$.

\end{enumerate}
\end{thm}

Theorem \ref{thm:tori} works in particular for a connected CW complex $X$ with finite $q$--skeleton (up to homotopy)
that is $q$--formal (over $\Q$). As explained in \S \ref{ss86}, we recover for $q=1$ Theorem B from \cite{DPS}. At the
same time, Theorem \ref{thm:tori} \eqref{8t2} provides a substantial generalization (for non--translated components) of
a basic result due to Arapura, who proved in \cite{A} that $\V^1_1 (X)$ is a finite union of (possibly translated) subtori of
$\T (\pi_1(X))$, when $X$ is a quasi--projective manifold. He has obtained a similar result for 
$\V^i_r (X)$ with $i>1$ under the additional assumption that $H^1(X,\Q)$ is a pure Hodge structure.
In the recent preprint \cite{BW}, Budur and Wang use Theorem \ref{thm:tori} \eqref{8t2} as a key ingredient to prove that
all rank one characteristic varieties of a quasi--projective manifold are finite unions of torsion--translated subtori.
Note that the strong conclusions of Theorem \ref{thm:tori} may not hold without our positive weight hypothesis;
see Example \ref{ex:counterc}.

In the case when $X$ is a quasi--projective manifold, we get the following consequence of Theorem \ref{thm:tori}. 
The proof of this corollary follows the ideas in Deligne's Lemma, stated as  Lemma 2 in \S 3 from Voisin's paper \cite{V}.

\begin{corollary}
\label{cor:subMHS}
Let $X$ be a connected smooth quasi--projective variety. Then, for any $i \geq 1$ and  $r\geq 1$, every irreducible component $E$ 
of the resonance variety $\R^i_r (\A, \theta)$, where $\A$ is a Gysin model of $X$, is a (rational) mixed Hodge substructure in $H^1(X)=H^1(\A)$.

\end{corollary}

In particular, the associated irreducible component $W=\exp(E)$ of the characteristic variety  $\V^i_r (X)$ is
of exponential Hodge type in the terminology of Arapura \cite{A}.  It follows that for any such irreducible component $W$ there is 
a holomorphic map $f_W:X \to T_W$, where $T_W$ is a complex torus, such that $W=f_W^*(H^1(T_W,\C^{\times}))$. More precisely, 
if $h^{1,0}=h^{0,1}$ and $h^{1,1}$ are the mixed Hodge numbers of the MHS $E$, then one has an extension
$$ 0 \to T_a \to T_W \to T_c \to 0 ,$$
where $T_a$ is an affine torus of dimension $h^{1,1}$ and $T_c$ is a complex compact torus of (complex) dimension $h^{1,0}$. 
In particular, $ \dim T_W =h^{1,0}+h^{1,1} \leq \dim W=b_1(T_W)$.
For details, see the end of \S II.1, and Corollary 1.2 and its proof in \S V of Arapura's paper \cite{A}.

Our approach to non--translated components of rank one
characteristic varieties for quasi--projective manifolds, via positive--weight decompositions of Gysin models, is in the same
spirit as a key result due to Morgan \cite{M}, where mixed Hodge structures are used to show that the Malcev Lie algebra
of the fundamental group is obtained by completion from a finitely presented Lie algebra having positive weights, on both
generators and relations.

In the rank one case, the resonance varieties $\R^i_r (\A, \theta)$ are usually denoted $\R^i_r (X)$, when 
$\A^{\bullet}= (H^{\bullet}(X, \K), d=0)$. They first appeared in work by Falk \cite{Fa} on complements of complex 
hyperplane arrangements, and have been intensively studied, by using a variety of techniques. As we recall in Example \ref{ex:arr},
an arrangement complement $X= X_{\bA}$ is formal over $\Q$, by results due to Brieskorn \cite{B}, Orlik and Solomon \cite{OS},
Orlik and Terao \cite{OT}. In particular, Theorem \ref{thm:tori} may be applied to $X_{\bA}$, with $q=\infty$, by replacing the 
cohomology ring of $X_{\bA}$ with the Orlik--Solomon algebra of the arrangement $\bA$, that depends only on the underlying 
combinatorics of $\bA$. 

Adopting this viewpoint, we may infer from Theorem \ref{thm:8main} that all analytic germs of characteristic varieties,
$\V^i_r (X_{\bA}, \iota)_{(1)} \cong \R^i_r ((H^{\bullet} (X_{\bA}, \K), d=0), \theta)_{(0)}$, are
combinatorially determined. This represents an extension from a broad nonabelian perspective of the main (local) result on  $X_{\bA}$
due to Esnault, Schechtman and Viehweg \cite {ESV} (corresponding to our Theorem \ref{thm:8main} \eqref{8m2}), where only 
the rank one case was treated and those due to Schechtman, Terao and Varchenko \cite {STV} where the abelian case is considered.

For an arrangement complement, Theorem \ref{thm:tori} \eqref{8t2} follows from \cite{A}, and then Theorem \ref{thm:tori} \eqref{8t1}
on $\R^i_r (X_{\bA})$ becomes a consequence of \cite{ESV}.

Our next result gives a global version of Theorem \ref{thm:8main} \eqref{8m2}: 
$\eexp_{\BB} (\R^i_r (\A, \theta)) \subseteq \V^i_r (X, \iota)$, for all $i\le q$ and $r\ge 0$, under a positive--weight assumption.
To formulate this more precisely, we define, for each $i\ge 0$, $b_i (X, \rho)= \dim_{\K} H^i(X, {}_{\iota\rho}V)$ for
$\rho \in \cH (\pi_1(X), \BB)$, and $\beta_i (\A, \omega)= \dim_{\K} H^i(\A^{\bullet} \otimes V, d_{\omega})$ for
$\omega \in \F (\A, \bb)$. 

\begin{thm}
\label{thm:8gl}
In Theorem \ref{thm:8main} \eqref{8m2}, identify $\F (\A, \bb)$ with $\Hom_{\gr} (\pi_1(X), \bb)$. Then the following hold,
for all $i\le q$ and $r\ge 0$.

\begin{enumerate}

\item \label{gl1} 
For $\K=\C$, $\eexp_{\BB} (\R^i_r (\A, \theta)_0) \subseteq \V^i_r (X, \iota)$, where $\R^i_r (\A, \theta)_0$ denotes the union 
of the irreducible components of $\R^i_r (\A, \theta)$ passing through the origin.

\item \label{gl2} 
If in Theorem \ref{thm:8main} \eqref{8m2} we also assume $\A^{\bullet}$ has positive weights, then 
$\eexp_{\BB} (\R^i_r (\A, \theta)) \subseteq \V^i_r (X, \iota)$. In other words
\[
\beta_i (\A, \omega) \le b_i (X, \eexp_{\BB} (\omega)) \, ,
\]
for every $\omega \in \F (\A, \bb)$ and all $i\le q$, over $\K=\RR$ or $\C$.

\end{enumerate}
\end{thm}

Note that Theorem \ref{thm:8gl} holds in particular when $X$
is a compact Kahler manifold and $\A^{\bullet}= (H^{\bullet}(X, \K), d=0)$, since $X$ is formal
(cf. Deligne, Griffiths, Morgan and Sullivan \cite{DGMS}). Another particular case occurs when $X$ is an arrangement complement, 
$\A^{\bullet}$ is the associated
Orlik--Solomon algebra, $q=\infty$ and $\iota =\id_{\C^{\times}}$. In this case, the global inequalities were obtained by
Libgober and Yuzvinsky in \cite{LY}. They may be extended to the case of an arbitrary quasi--projective manifold $X$, for
arbitrary abelian coefficients $\BB$ and $q=\infty$, replacing the Orlik--Solomon algebra by a Gysin model. We also have the following.

\begin{corollary}
\label{cor:compactK}
Suppose that $\iota =\id_{\C^{\times}}$, $X$ is a compact Kahler manifold and we take  $\A^{\bullet}=(H^{\bullet}(X), d=0)$. Then
$$
\beta_i (\A, \omega) = b_i (X, \eexp_{\BB} (\omega)) \, ,
$$
for every $\omega \in H^{1,0}(X) \cup  H^{0,1}(X)$ and all $i$.
Moreover, the spectral sequence with $E_1^{p,q}=H^q(X,\Omega^p_X)$ and differential $d_1$ induced by 
the cup product with $\omega \in H^{1,0}(X)$, converging to the hypercohomology groups 
$\HH^{p+q}(X; (\Omega^{\bullet}_X,\omega \wedge ))$, degenerates at $E_2$.

\end{corollary}

Our last main result uses the canonical positive--weight decomposition of cohomology rings with trivial differential. It concerns the
{\em Farber--Novikov spectral sequence} of a connected, finite CW complex $X$, with respect to a group epimorphism,
$\nu : \pi_1(X) \surj \Z$, reviewed in \S \ref{ss88}. 

Denote by $\nu_\C \in H^1(X,\C)$ the associated cohomology class, and let $(H^{\bullet}(X,\C), \nu_\C \cdot)$ be the
corresponding Aomoto complex. Let $\nu^*: \C^{\times} \hookrightarrow \T (\pi_1(X))$ be the algebraic map 
induced on character tori. 
The Farber--Novikov spectral sequence starts at $E^i_2= H^i (H^{\bullet}(X,\C), \nu_\C \cdot)$ and converges to 
$ H^i (X, {}_{\nu^*(t)}\C)$, where $t\in \C^{\times} \setminus F$, with $F$ a finite set.

\begin{thm}
\label{thm:8deg}
If $X$ is a formal space, then $E^{\bullet}_2= E^{\bullet}_{\infty}$, in the Farber--Novikov spectral sequence.
\end{thm}

Theorem \ref{thm:8deg} extends considerably Farber's degeneracy result from \cite{F1}, where $X$ is assumed to be a connected,
compact complex manifold satisfying the $d d^c$--Lemma (known then to be formal, according to \cite{DGMS}). 
Theorem \ref{thm:8deg} was also obtained by Kohno and Pajitnov, in the recent preprint \cite{KP}, with a different technique,
based on twisted versions of Aomoto complexes.

In another recent preprint \cite{Wa}, Wang proposes a more general construction of twisted  Aomoto complexes,
$\aom^{\bullet}(X, \rho)$, associated to a connected, finite CW complex $X$, and an arbitrary representation
$\rho \in \cH (\pi_1(X), \GL (\C^n))$. When $X$ is compact Kahler and $\rho$ is semi-simple, the author succeeds to adapt
our proof of Theorem \ref{thm:8main} \eqref{8m1} in the case $\iota =\id_{\GL_n}$ and to obtain an interesting extension, 
with $\rho$ in place of $1$, $\F (\A, \bb)$ changed to the Goldman--Millson quadratic cone \cite{GM}, and 
$\R^i_r (\A, \theta)$ replaced by the corresponding jump loci of $\aom^{\bullet}(X, \rho)$. Note that Wang's result
does not hold on arbitrary quasi--projective manifolds, even for $n=1$ and $\rho=1$; see Example \ref{ex:counterc}.

\section{A brief guide to the proofs}
\label{sec:out}

Roughly speaking, we use the $\K^{\times}$--action associated to the positive--weight decomposition of $\A^{\bullet}$
(described in \S \ref{ss85}) to globalize local results; we obtain the local analytic isomorphisms at the level of completions, and then
invoke either M.~Artin's approximation theorem or faithful flatness (here, Tougeron's book \cite{T} is a good reference).

The first task is to show that the local analytic rings of $\cH (G, \BB)_{(1)} $, denoted by $R$, and of $\F (\A, \bb)_{(0)}$, denoted by $\oR$,
have isomorphic completions, i.e., in the language from \cite{GM}, they have the same functor of Artin rings. This requires two ingredients.
The passage from (algebraic) flat connections on $\A^{\bullet}$ to (topological) flat connections on $X$ is made possible by our assumption 
$\Omega^{\bullet}(X, \K) \simeq_1 \A^{\bullet}$. This allows the use of the `equivalence theorem' in the deformation theory of 
flat connections from \cite{GM}, recorded here as Theorem \ref{thm:aflat}. To relate in a convenient way the deformation theory of flat connections
on $X$ and the  functor of Artin rings of $\cH (G, \BB)_{(1)} $, we need to reformulate Sullivan's result \cite{S}, on the relationship between 
the $1$--minimal model of $X$ and Quillen's \cite{Q} $\K$--unipotent completion of $G=\pi_1(X)$. This second tool is provided by our 
Corollary \ref{cor:kunip}, where we use monodromy representations of certain flat connections on $X$, associated to its $1$--minimal model. 
In this way, we obtain an identification $\wR \stackrel{\sim}{\leftrightarrow} \woR$, as explained in Proposition \ref{prop:fart} and 
Remark \ref{rem:compliso}. At the same time, Corollary \ref{cor:kunip} leads to the proof of Theorem \ref{thm:km0}.

In the proof of Theorem \ref{thm:8main}, the isomorphism between Aomoto complex cohomology of $\A$ and twisted cohomology of $X$
also needs two steps. The first is achieved in Theorem \ref{thm:adr}. Here, we fully use our assumption 
$\Omega^{\bullet}(X, \K) \simeq_q \A^{\bullet}$ to identify naturally, up to degree $q$,  the cohomology of Aomoto complexes coming from 
Artin rings, associated to $\A^{\bullet}$ and $\Omega^{\bullet}(X, \K)$ respectively. The second step is a natural identification,  between
the cohomology of Aomoto complexes coming from flat connections on $X$, and the  twisted cohomology of $X$ with coefficients in 
the associated monodromy representations. This is provided by the Sullivan--Gomez Tato twisted De Rham theorem from \cite{GT}, 
recorded here in a convenient form in Theorem \ref{thm:twdr}. 

Finally, we have to show in Theorem \ref{thm:8main} that the picture for cohomology jump loci is compatible with the isomorphism
$\wR \stackrel{\sim}{\leftrightarrow} \woR$. This is done in Section \ref{sec:tech}. We denote by $P$ 
(respectively $\oP$) the coordinate ring of $\cH (G, \BB)$ (respectively $\F (\A, \bb)$), described in the book of Lubotzky and Magid \cite{LM} 
(respectively in \S \ref{ss12}). Note that $\wR = \widehat{P}$ and $\woR = \widehat{\oP}$, where the completions are taken 
with respect to the maximal ideals of $P$ and $\oP$ corresponding to the chosen basepoints.

In Section \ref{sec:tech}, we 
construct two {\em universal} cochain complexes, $C^{\bullet}(X, \iota)$ (over $P$) and  $C^{\bullet}(\A, \theta)$ (over $\oP$).
They have the property of computing, via specialization, the twisted cohomology of $X$, respectively the Aomoto complex cohomology of $\A$.
We conclude in Proposition \ref{prop:comp}, where we consider two ring epimorphisms, $\wR \surj S$ and $\woR \surj S$,
compatible with the identification $\wR \stackrel{\sim}{\leftrightarrow} \woR$, and we show that the cohomologies of the corresponding specializations,
$C^{\bullet} (X, \iota) \otimes_P S$ and $C^{\bullet}(\A, \theta) \otimes_{\oP} S$, are $S$--isomorphic up to degree $q$. This is used in 
\S\S \ref{ss100}--\ref{ss82} to finish the proof of Theorem \ref{thm:8main} \eqref{8m1}.

In the abelian case from Theorem \ref{thm:8main} \eqref{8m2}, the local analytic isomorphism 
$$\eexp_{\BB} : \F (\A, \bb)_{(0)} \isom \cH (\pi_1(X), \BB)_{(1)}$$ turns out to induce on completions the canonical identification
$\wR \stackrel{\sim}{\leftrightarrow} \woR$. Therefore, there is no need for Artin approximation. This may be replaced by a faithful flatness
argument, to infer that the corresponding cohomology jump loci are identified by $\eexp_{\BB}$.

\section{Algebraic preliminaries}
\label{sec:prel}

We draw up an inventory of algebraic tools, needed for establishing the `equivalence theorems' for deformations of 
flat connections and of cohomology with respect to the covariant derivative, following \cite{GM}, \cite{S}, \cite{HS}.

We consider non-negatively graded vector spaces over a characteristic $0$ ground field $\K$; $\otimes$ means $\otimes_{\K}$,
$\Hom$ stands for $\Hom_{\K}$, and $V^*$ denotes the $\K$--dual of $V$. 

\subsection{Commutative algebras and Lie algebras}
\label{ss11}

We denote by $\sCDGA$ the category of commutative differential graded algebras over $\K$. The category of differential graded Lie algebras 
is denoted $\sDGL$. (In both settings, we follow the standard sign conventions.) In particular, we consider $\A^{\bullet} \in \Ob (\sCGA)$ 
(the category of commutative graded algebras) as a $\sCDGA$ with differential $d=0$. We may also view $E\in \Ob (\sLie)$ 
(the category of Lie algebras) as a $\sDGL$, with differential $d=0$ and grading $E^{\bullet}=E^0=E$.

The {\em cochain functor}, $\CC : \sLie \to \sCDGA$, associates to a finite-dimensional Lie algebra $E$ with bracket $\beta$ the $\sCDGA$
$\CC^{\bullet}(E)= \wedge^{\bullet} E^*$, with differential $d=-\beta^* : E^* \to \wedge^2 E^*$ (extended to $ \wedge^{\bullet} E^*$
by the derivation property). In this way, the category of finite-dimensional Lie algebras is identified with the category of $\sCDGA$'s that are
free (as $\sCGA$'s) and finitely generated in degree $1$. 

The bifunctor $\LL : \sCDGA \times \sLie \to \sDGL$ associates to $\A^{\bullet} \in \Ob (\sCDGA)$  and $E\in \Ob (\sLie)$ the graded
vector space $\LL (\A, E)= \A^{\bullet} \otimes E$, with differential $d\otimes E$ (here and in the sequel, we denote the identity map of
a vector space $E$ by $E$); the Lie bracket is defined by $[a\otimes e, a' \otimes e']= a a'\otimes [e,e']$, for $a,a'\in \A$ and $e,e' \in E$.

If $L^{\bullet}\in \Ob (\sDGL)$ and $\Lambda \in \Ob (\sComm)$ (the category of ungraded, commutative $\K$--algebras), we have a 
natural extension of scalars, $L^{\bullet} \otimes \Lambda$,  with differential $d\otimes \Lambda$ and bracket 
$[x\otimes \lambda, x' \otimes \lambda']= [x,x']\otimes \lambda \lambda'$. Note that $L^{\bullet} \otimes \Lambda \in \Ob (\Lambda-\sDGL)$,
that is, it is a differential graded Lie algebra over $\Lambda$.

\subsection{Flat connections}
\label{ss12}

Given $L^{\bullet}\in \Ob (\sDGL)$, set 
\begin{equation}
\label{eq=mc}
\F (L)= \{ \omega \in L^1 \, \mid \, d\omega + \frac{1}{2} [\omega, \omega]=0 \} \, .
\end{equation} 
The above set of {\em flat connections} has a distinguished element, $0\in \F (L)$. 

For $\Lambda \in \Ob (\sComm)$, we will consider maximal ideals $\m \subseteq \Lambda$ with residue field $\Lambda/\m =\K$. 
We denote $\K$--algebra homomorphisms by $\Hom_{\alg}(\Lambda, \Lambda')$. Whenever $\Lambda$ and $\Lambda'$ have distinguished 
maximal ideals, $\m$ and $\m'$, $\Hom_{\loc}(\Lambda, \Lambda')$ stands for $\K$--algebra homomorphisms sending $\m$ into $\m'$.
This applies in particular to the category of Artinian local $\K$--algebras, $\sArt$. For $A \in \Ob (\sArt)$, we denote by $\m_A$ the
maximal ideal. 

If $L\in \Ob (\sDGL)$ and $\m \subseteq \Lambda$ is a maximal ideal, note that $L^{\bullet}\otimes \m$ is a sub--$\sDGL$ of 
$L^{\bullet} \otimes \Lambda$. We will use the notation 
\begin{equation}
\label{eq=locmc}
\F_{\loc} (L\otimes \Lambda)= \F(L \otimes \m) \, .
\end{equation}

For $\A\in \Ob (\sCDGA)$ and $E\in \Ob (\sLie)$, $\F (\A, E) := \F (\A\otimes E)$ denotes the set of flat connections \eqref{eq=mc} on the $\sDGL$
$\LL (\A, E)$ constructed in \S \ref{ss11}. When both $\A^1$ and $E$ are finite--dimensional, note that $\F (\A\otimes E)$ is an (usually nonreduced) affine
variety, with coordinate ring denoted $\oP$, and maximal ideal $\overline{\m}$ corresponding to the basepoint $0\in \F (\A\otimes E)$.

Moreover, there is an {\em universal flat connection}, $\ow \in \F_{\loc} (\A \otimes E \otimes \oP)$. Picking $\K$--bases for $\A^1$ and $E$,
$\{ a_i \}$ and $\{ e_{\alpha} \}$, with dual bases $\{ a_i^* \}$ and $\{ e_{\alpha}^* \}$, we may write
\begin{equation}
\label{eq=funiv}
\ow =\sum_{i, \alpha} a_i\otimes e_{\alpha} \otimes a_i^* \otimes e_{\alpha}^* \, ,
\end{equation}
modulo the relations defining $\oP$.

Clearly, there is a natural bijective correspondence, $\Hom_{\loc} (\oP , A) \isom \F_{\loc} (\A \otimes E \otimes A)$, for $A\in \Ob (\sArt)$,
sending $\overline{\varphi} \in \Hom_{\loc}$ to $(\A \otimes E \otimes \overline{\varphi}) (\ow) \in \F_{\loc}$. 

\subsection{Flat connections and Lie cochains}
\label{ss13}

It will be very useful to interpret flat connections in $\sCDGA$ terms, as follows.

\begin{lemma}
\label{lem:cflat}
Assuming $\dim_{\K} E< \infty$, the natural $\K$--linear isomorphism $\A^1 \otimes E \isom \Hom (E^*, \A^1)$ identifies 
$\F (\A \otimes E)$ with $\Hom_{\sCDGA}( \CC^{\bullet}(E), \A^{\bullet})$. 
\end{lemma}

\begin{proof}
Plainly, $\Hom (E^*, \A^1)= \Hom_{\sCGA}( \wedge^{\bullet} E^*, \A^{\bullet})$. If 
$\omega =\sum_{\alpha} \omega_{\alpha}\otimes e_{\alpha} \in \A^1 \otimes E$ corresponds to 
$f\in \Hom_{\sCGA}( \CC^{\bullet}(E), \A^{\bullet})$, then $f(e_{\alpha}^*)= \omega_{\alpha}$, by construction. We have to
check that $\omega$ satisfies the Maurer--Cartan equation \eqref{eq=mc} if and only if $f$ commutes with differentials.

To this end, write $[e_{\alpha}, e_{\beta}]= \sum_{\gamma} c^{\gamma}_{\alpha \beta}\,  e_{\gamma}$. Hence,
$d e_{\gamma}^* = - \sum_{\alpha <\beta}  c^{\gamma}_{\alpha \beta} \, e_{\alpha}^* \wedge e_{\beta}^*$, according to
\S \ref{ss11}. Therefore,  
$(df-fd) (e_{\gamma}^* )= d\omega_{\gamma}+  \sum_{\alpha <\beta}  c^{\gamma}_{\alpha \beta} \, \omega_{\alpha} \cdot \omega_{\beta}$.
On the other hand, $d\omega + \frac{1}{2} [\omega, \omega]= \sum_{\gamma} (d\omega_{\gamma}
+  \sum_{\alpha <\beta}  c^{\gamma}_{\alpha \beta}\, \omega_{\alpha} \cdot \omega_{\beta}) \otimes e_{\gamma}$.
The claim follows.
\end{proof}

In particular, $\omega_E := \sum_{\alpha} e_{\alpha}^* \otimes e_{\alpha} \in \F (\CC^{\bullet} (E)\otimes E)$, and the map
$f \mapsto (f\otimes E) (\omega_E)$ gives the inverse bijection,
\begin{equation}
\label{eq=finv}
\Hom_{\sCDGA}( \CC^{\bullet}(E), \A^{\bullet}) \isom \F (\A \otimes E) \, . 
\end{equation} 

\subsection{Covariant derivative}
\label{ss14}

For $L^{\bullet}\in \Ob (\sDGL)$ and $\omega \in L^1$, consider the {\em covariant derivative} $d_{\omega}:= d+ \ad_{\omega}$. 
This is a degree $1$ self-map of $L^{\bullet}$. 

\begin{lemma}
\label{lem:covder}
Assume $\omega\in \F(L)$. Then
\begin{enumerate}

\item \label{cd1}
The covariant derivative is a differential, i.e., $d_{\omega}^2 =0$.

\item \label{cd2}
When $L^{\bullet}$ is a $\Lambda-\sDGL$, $(L, d_{\omega})$ is a $\Lambda$--cochain complex.

\item \label{cd3}
If $f\in \Hom_{\sDGL}(L, L')$, then $f: (L, d_{\omega}) \to (L', d_{\omega'})$ is a cochain map, where
$\omega'= f(\omega)$. 
\end{enumerate}
\end{lemma}

\begin{proof}
\eqref{cd1} For $\lambda \in L^p$, $d_{\omega}^2 (\lambda)= d_{\omega}(d\lambda + [\omega, \lambda])=
d^2 \lambda + [d\omega, \lambda]- [\omega, d\lambda]+ [\omega, d\lambda]+ [\omega, [\omega, \lambda]]=
-\frac{1}{2} [[\omega, \omega], \lambda] +  [\omega, [\omega, \lambda]]$, by flatness. This in turn equals zero,
by the graded Lie identities.

Properties \eqref{cd2} and \eqref{cd3} are immediate consequences of the definition of $d_{\omega}$.
\end{proof}

Fix a Lie module $V$ over $\bb$, with structure map $\theta: \bb \to \gl (V)$. Consider the semi-direct product
Lie algebra, $V\rtimes_{\theta} \bb$. For $\A^{\bullet} \in \Ob (\sCDGA)$ and $\Lambda \in \Ob (\sComm)$, there is a 
natural, split exact sequence in $\Lambda-\sDGL$, 
\begin{equation}
\label{eq=splitdgl}
0\to \LL^{\bullet}(\A, V)\otimes \Lambda \longrightarrow  \LL^{\bullet}(\A, V \rtimes_{\theta} \bb)\otimes \Lambda \longrightarrow
\LL^{\bullet}(\A, \bb)\otimes \Lambda \to 0\, ,
\end{equation}
(cf. \S \ref{ss11}), where $V$ is considered as an abelian Lie ideal of $V \rtimes_{\theta} \bb$. 

For $\omega \in  \F (\A \otimes \bb \otimes \Lambda)$, consider the {\em Aomoto complex}
\begin{equation}
\label{eq=defaom}
(\A^{\bullet} \otimes V \otimes \Lambda, d_{\omega})\, .
\end{equation}
This is a sub--$\Lambda$--cochain complex of $(\LL^{\bullet}(\A, V \rtimes_{\theta} \bb)\otimes \Lambda, d_{\omega})$,
see Lemma \ref{lem:covder}. The above construction extends the one from \cite{DPS}, where only $\sCDGA$'s with 
zero differential were considered.

The Aomoto complex is natural, in the following sense. For $f\in \Hom_{\sCDGA}(\A, \A')$ and 
$\varphi \in \Hom_{\alg}(\Lambda, \Lambda')$, there is a $\Lambda$--linear cochain map, 
\[
f\otimes V\otimes \varphi : (\A \otimes V \otimes \Lambda, d_{\omega}) \longrightarrow
(\A' \otimes V \otimes \Lambda', d_{\omega'}) \, ,
\]
where $\omega'= (f\otimes \bb \otimes \varphi)(\omega) \in  \F (\A' \otimes \bb \otimes \Lambda')$.

Assuming $\A^1$, $\bb$ and $V$ to be finite--dimensional over $\K$, we have the {\em universal Aomoto complex},
\begin{equation}
\label{eq=univaom}
C^{\bullet}(\A, \theta):= (\A^{\bullet} \otimes V \otimes \oP, d_{\ow}) \, ,
\end{equation}
see \S \ref{ss12}. This is a free $\oP$--cochain complex (finitely generated in degrees up to $q$, when $\dim_{\K} \A^{\le q} <\infty$).
Moreover, for any $\K$--point of $\F (\A \otimes \bb)$, $\omega \in \Hom_{\alg}(\oP, \K)$, the $\K$--cochain complexes 
$C^{\bullet}(\A, \theta) \otimes_{\oP} \K$ and $(\A^{\bullet} \otimes V , d_{\omega})$ are isomorphic.

We will say that a cochain map $f: C^{\bullet} \to C'^{\bullet}$ is a  {\em $q$--equivalence} ($1\le q\le \infty$) if the induced map 
in cohomology, $H^{\bullet}f$, is an isomorphism in degrees $\bullet \le q$ and a monomorphism for  $\bullet = q+1$.
(In the literature, $\infty$--equivalences are also called quasi--isomorphisms.)

The first `equivalence theorem' says that the construction \eqref{eq=defaom} preserves $q$--equivalences, in the following sense.

\begin{theorem}
\label{thm:adr}
Let $f: \A \to \A'$ be a  $q$--equivalence in $\sCDGA$. Then 
$f\otimes V \otimes A : (\A \otimes V \otimes A, d_{\omega}) \rightarrow (\A' \otimes V \otimes A, d_{\omega'}) $ is a natural
$A$--linear $q$--equivalence, for  $A \in \Ob (\sArt)$ and $\omega \in  \F_{\loc} (\A \otimes \bb \otimes A)$, where
$\omega' = (f\otimes \bb \otimes A) (\omega)\in  \F_{\loc} (\A' \otimes \bb \otimes A)$.
\end{theorem}

\begin{proof}
Consider the finite, decreasing filtration $\{ \A^{\bullet} \otimes V \otimes \m_A^s \}_{s\ge 0}$. Since
$\omega \in  \F (\A \otimes \bb \otimes \m_A)$, the filtration is compatible with the differential $d_{\omega}$ and 
$E_1^{s,t}= H^{s+t} \A \otimes V \otimes \gr_s (A)$, where $\gr_{\bullet} (A)$ stands for the $\m_A$--adic associated graded,
in the spectral sequence converging to $H^{s+t} (\A \otimes V \otimes A, d_{\omega})$. Similar considerations apply to
$(\A' \otimes V \otimes A, d_{\omega'})$. Our hypothesis on $H^{\bullet}f$ implies that $E_1^{\bullet}(f\otimes V\otimes A)$
is an isomorphism for $\bullet \le q$ and a monomorphism for $\bullet =q+1$. The conclusion 
follows by a standard spectral sequence argument.
\end{proof}

Note that the assumption $\omega \in  \F_{\loc} (\A \otimes \bb \otimes A)$ is essential: when $\omega \in  \F (\A \otimes \bb \otimes A)$
is an arbitrary flat connection, Theorem \ref{thm:adr} may fail.

\subsection{The Deligne--Schlessinger--Stasheff theorem}
\label{ss15}

We will need to consider the category $\sACDGA$ of augmented $\sCDGA$'s $(\A, \varepsilon)$, where 
$\varepsilon \in \Hom_{\sCDGA}( \A, \K)$ is the {\em augmentation}; $f\in  \Hom_{\sACDGA}( (\A, \varepsilon), (\A' , \varepsilon') )$ 
if $f\in  \Hom_{\sCDGA}(\A , \A' )$ and $\varepsilon' f=\varepsilon$. The {\em augmentation ideal}, $\tA^{\bullet}= \ker (\varepsilon)$
is a differential graded ideal of $\A^{\bullet}$. Note that the inclusion identifies $H^+ \tA$ with $H^+ \A$.

We say that $\A^{\bullet} \in \Ob (\sCDGA)$ is {\em connected} if $\A^0= \K \cdot 1$, and {\em $h$--connected} if 
$(H^{\bullet} \A, d=0)$ is connected, i.e., $H^0 \A =\K \cdot 1$. Clearly, $H^0 \tA=0$ when $\A$ is $h$--connected, and
$\tA^0 =0$ if $\A$ is connected. Similarly, when $f\in  \Hom_{\sACDGA}(\A, \A' )$ is a $q$--equivalence and both $\A$ and $\A'$ are
$h$--connected, $\widetilde{f}: \tA \to \tA'$ is again a $q$--equivalence.

Note also that a connected $\sCDGA$ $\A^{\bullet}$ has a unique augmentation $\varepsilon$, sending $\A^+$ to $0$ and $1$ to $1$. 
Moreover, $\Hom_{\sACDGA}(\A, \A' )= \Hom_{\sCDGA}(\A, \A' )$, when both $\A$ and $\A'$ are connected. 

If $(\A^{\bullet}, \varepsilon) \in \Ob (\sACDGA)$, $E\in \Ob (\sLie)$ and $\Lambda \in \Ob (\sComm)$, the sequence
$$0\to \tA^{\bullet} \otimes E\otimes \Lambda \longrightarrow \A^{\bullet} \otimes E\otimes \Lambda 
\stackrel{\varepsilon \otimes E \otimes \Lambda}{\longrightarrow} E \otimes \Lambda \to 0$$ is exact in $\Lambda-\sDGL$. 
Clearly, the inclusion identifies $\F (\tA^{\bullet} \otimes E\otimes \Lambda)$ with $\F (\A^{\bullet} \otimes E\otimes \Lambda)$,
as well as $\F_{\loc} (\tA^{\bullet} \otimes E\otimes \Lambda)$ with $\F_{\loc} (\A^{\bullet} \otimes E\otimes \Lambda)$, for
any maximal ideal of $\Lambda$.

Plainly, $f\in \Hom_{\sACDGA}(\A, \A' )$ induces  $\widetilde{f} \otimes E \in \Hom_{\sDGL}(\tA \otimes E, \tA' \otimes E )$.
Furthermore, if $f$ is a $q$--equivalence and both $\A$ and $\A'$ are $h$--connected, then $\widetilde{f} \otimes E$ is a 
$q$--equivalence.

The second `equivalence theorem' is an immediate consequence of Theorem 2.4  from Goldman and Millson \cite{GM}.

\begin{theorem}
\label{thm:aflat}
Let $f\in \Hom_{\sCDGA} (\A, \A')$ be a $1$--equivalence between connected algebras. Then
$f\otimes E\otimes \m_A : \F_{\loc} (\A \otimes E\otimes A) \isom \F_{\loc} (\A' \otimes E\otimes A)$
is a natural bijection, for $E\in \Ob (\sLie)$ and $A \in \Ob (\sArt)$.
\end{theorem}

\begin{proof}
By the above remarks, we may replace $\A$ by $\tA$ and $\A'$ by $\tA'$ in the conclusion; at the same time, we know that
$\widetilde{f} \otimes E$ is a $1$--equivalence in $\sDGL$.

By the Deligne--Schlessinger--Stasheff theorem proved in \cite{GM}, $\widetilde{f}\otimes E\otimes \m_A$ induces a 
bijection, 
$\F (\tA \otimes E\otimes \m_A) /  \tA^0 \otimes E\otimes \m_A \isom \F (\tA' \otimes E\otimes \m_A)/ \tA'^0 \otimes E\otimes \m_A$,
between orbits in $\F$ of a certain action of degree $0$ components. 
Since both algebras are connected, these degree $0$ pieces are trivial, and our conclusion follows.
\end{proof}

\section{Rational homotopy theory}
\label{sec:rht}

In this section, we start by recalling a couple of basic facts from rational homotopy theory (RHT), harvested from
\cite{BG, H, L, M, Q, S}. We then establish the RHT setup for the proofs of the results from the Introduction.
Except otherwise stated, we work over a characteristic zero field $\K$.

\subsection{Minimal algebras}
\label{ss21}

The {\em free} $\sCGA$ on a positively graded vector space $W^{\bullet}$ is denoted $\bigwedge W^{\bullet}$.
It is the tensor product of the exterior algebra on $W^{odd}$ and the symmetric algebra on $W^{even}$; in
particular it is connected. We denote by $\sMin$ the full subcategory of $\sCDGA$ consisting of free $\sCGA$'s whose
differentials satisfy a certain condition, called {\em minimal}. A minimal  $\sCDGA$ $(\bigwedge W^{\bullet}, d)$
belongs to $\Ob (\sMin_q)$, for $q\ge 1$, if $W^{>q}=0$, and is called {\em $q$--minimal}. Note that
$\sMin_{\infty}=\sMin$. 

The category $\sMin_1$ will be particularly important for our purposes, so we describe it in detail. Let
$\cN^{\bullet}= (\bigwedge^{\bullet} W, d)$ be freely generated in degree $1$. $\cN$ is $1$--minimal if
$W=\cup_{n\ge 1} F^n (\cN)$, where the {\em canonical}, increasing filtration of $W=\cN^1$ is inductively
defined by $F^1=0$ and $F^n =(d_{\mid W})^{-1} (\bigwedge ^2 F^{n-1})$. In particular, 
$ \cN =\lim \limits_{\longrightarrow} (\bigwedge^{\bullet} F^n, d)$ in $\sCDGA$, if $\cN \in \Ob (\sMin_1)$.

Note that $F^2=H^1 \cN$. When $\dim_{\K} H^1 \cN <\infty$, it follows by induction that $F^n$ is finite--dimensional,
for all $n$. In this case, $(\bigwedge^{\bullet} F^n, d)= \CC^{\bullet}(E^n)$, as explained in \ref{ss11}, and $\cN$
gives rise to a Lie tower of finite--dimensional nilpotent Lie algebras and central extensions,
\begin{equation}
\label{eq=ltower}
\cdots \surj E^{n+1} \stackrel{q_n}{\surj} E^n \surj\cdots
\end{equation}
This construction associates to $F^n$ the dual vector space $E^n$, with Lie bracket the negative dual of
$d: F^n \to \bigwedge^2 F^n$, and the above projections $q_n$ are dual to the inclusions $F^n \hookrightarrow F^{n+1}$.

For $E\in \Ob (\sLie)$, the descending {\em lower central series} filtration will be denoted by $\{ E_r \}_{r\ge 1}$, where
$E_1=E$ and $E_{r+1}=[E, E_r]$. (For a group $G$, we may replace the Lie bracket $[\cdot, \cdot]$ by the group
commutator $(\cdot, \cdot)$ and perform the same construction to get the lower central series filtration $\{ G_r \}_{r\ge 1}$.)
With this notation, it is known that $E_n^n=0$ in \eqref{eq=ltower}, for all $n$. 

\subsection{Campbell--Hausdorff groups}
\label{ss22}

We will need later on the following functorial construction, that associates to a finite--dimensional nilpotent Lie algebra $E$
the {\em exponential group} $\eexp (E)=E$, endowed with the Campbell--Hausdorff multiplication from classical local Lie theory:
$x\cdot y=x+y+ \frac{1}{2} [x,y]+\cdots$ Note that in this way $\eexp (E)$ becomes a unipotent linear algebraic group with
Lie algebra $E$, that is uniquely divisible and nilpotent as an abstract group, and $\eexp (E_r)=\eexp (E)_r$, for all $r$.
Moreover, $\Hom _{\sLie}(E, E')= \Hom_{\alggr}(\eexp (E), \eexp (E'))$, and
$\Hom _{\sLie}(E, E')= \Hom_{\gr}(\eexp (E), \eexp (E'))$, when $\K =\Q$.

\begin{example}
\label{ex:lieart}
Starting with a finite--dimensional Lie algebra $E$ and $A\in \Ob (\sArt)$, we may construct another finite--dimensional Lie algebra,
$E\otimes \m_A$, by extension of scalars; see \ref{ss11}. Moreover,  $E\otimes \m_A$ is a nilpotent Lie algebra, since $\m_A$ is a
nilpotent ideal. The groups $\eexp (E\otimes \m_A)$ play a key role in the analysis of analytic germs of representation varieties at the origin.
\end{example}

\subsection{Minimal models}
\label{ss23}

A {\em $q$--minimal model map} is a $q$--equivalence in $\sCDGA$, $\cN \to \A$, with $\cN \in \Ob (\sMin_q)$. A basic result in
rational homotopy theory is that every $h$--connected $\A \in \Ob (\sCDGA)$ has a unique {\em $q$--minimal model}
$\M_q (\A)\in \Ob (\sMin_q)$, for which there is a $q$--equivalence in $\sCDGA$, $\M_q (\A) \to \A$. For $q=\infty$, we shall
omit $q$ from notation.

If $\M_q (\A)= (\bigwedge W^{\bullet}, d)$ and $r\le q$, then $\M_r (\A)= (\bigwedge W^{\le r}, d)$. Conversely, if
$\M_q (\A) \to \A$ is any $q$--minimal model map, it can be extended to an $r$--minimal model map,
$\M_r (\A) \to \A$, for $r\ge q$, where $\M_r (\A)$ is obtained from $\M_q (\A)$ by adding free $\sCGA$ generators and 
extending the differential. 

The key definition below is inspired from D. Sullivan's notion of {\em formality} (corresponding to $\infty$--formality)
introduced in \cite{DGMS}. 

\begin{definition}
\label{def:qtype}
The $\sCDGA$'s $\A$ and $\overline{\A}$ have {\em the same $q$--type} ($1\le q\le \infty$) if there is a zigzag of 
$q$--equivalences in $\sCDGA$ connecting them. Notation: $\A \simeq_q \overline{\A}$; the associated homology
isomorphism will be denoted $H^{\le q} \A \stackrel{\sim}{\leftrightarrow} H^{\le q} \overline{\A}$. The $\sCDGA$
$\A$ is {\em $q$--formal} if  $\A^{\bullet} \simeq_q (H^{\bullet} \A, d=0)$. 
\end{definition}

By minimal model theory for $h$--connected $\sCDGA$'s, $\A \simeq_q \overline{\A}$ if and only if there is a diagram
$\A \leftarrow \cN \rightarrow \overline{\A}$ in $\sCDGA$, where both arrows are $q$--minimal model maps.

\subsection{Untwisted De Rham theory}
\label{ss24}

We briefly review the De Rham algebras of a topological space $X\in \Ob (\sTop)$. Denote by 
$K_{\bullet}$ the simplicial set $\Sing_{\bullet}(X) \in \Ob (\sSS)$. The singular simplices,
$K_n = \{ \sigma: \Delta_n \to X \}$, are endowed with the usual face and degeneracy operators,
$\partial_i : K_{n+1} \to K_n$ and $s_i : K_n \to K_{n+1}$. For $\sigma \in K_n$, set $\mid \sigma \mid = n$.

Let $\Omega^{\bullet}_{\K} (n)$ denote the $\sCDGA$ of differential forms on $\Delta_n$ having $\K$--polynomial
coefficients in the barycentric coordinates. Let $\partial_i^*: \Omega^{\bullet}_{\K} (n+1) \to \Omega^{\bullet}_{\K} (n)$
and $s_i^*: \Omega^{\bullet}_{\K} (n) \to \Omega^{\bullet}_{\K} (n+1)$ be the induced $\sCDGA$ maps.

For $K\in \Ob (\sSS)$, denote by $\Omega^{\bullet}_{\K} (K)\in \Ob (\sCDGA)$ the global sections of the associated
simplicial sheaf in $\sCDGA$, that is,
\begin{equation}
\label{eq=drsheaf}
\Omega^{\bullet}_{\K} (K)= \big \{ \eta_{\sigma} \in \Omega^{\bullet}_{\K} (\mid \sigma \mid), \sigma \in K \, \mid \, 
\partial^*_i  \eta_{\sigma}= \eta_{\partial_i \sigma}, s^*_i  \eta_{\sigma}= \eta_{s_i \sigma}, \forall i \big \} \, .
\end{equation}

Setting $\Omega^{\bullet}_{\K} (X)= \Omega^{\bullet}_{\K} (\Sing (X))$, we have the following De Rham theorem in $\sTop$.

\begin{theorem}
\label{thm:topdr}
There is a natural $\sCGA$ isomorphism, $H^{\bullet} \Omega_{\K} (X) \cong H^{\bullet} (X, \K)$, preserving 
$\Q$--structures. If $X$ is path--connected, $\M_q (\Omega^{\bullet}_{\K} (X))= \M_q (\Omega^{\bullet}_{\Q} (X)) \otimes \K$,
for all $q$.
\end{theorem}

When $\K =\RR$ or $\C$, let $\Omega^{\bullet}_{\diff} (n)$ be the usual $C^{\infty}$ De Rham $\sCDGA$ of $\Delta_n$, containing
$\Omega^{\bullet}_{\K} (n)$ as a sub--$\sCDGA$. Replacing $\Omega^{\bullet}_{\K}$ by $\Omega^{\bullet}_{\diff}$ in
\eqref{eq=drsheaf}, we obtain a natural $\infty$--equivalence in $\sCDGA$,  
$\Omega^{\bullet}_{\K}(K) \hookrightarrow \Omega^{\bullet}_{\diff}(K)$.
Setting $\Omega^{\bullet} (X, \K)= \Omega^{\bullet}_{\diff} (\Sing (X))$, we thus have:
$\Omega^{\bullet}_{\K} (X) \simeq_{\infty} \Omega^{\bullet} (X, \K)$, for $X \in \Ob (\sTop)$. 

Therefore, $\Omega^{\bullet} (X, \K) \simeq_q \A^{\bullet} \eqv \Omega^{\bullet}_{\K} (X) \simeq_q \A^{\bullet}$,
and this implies, when $X$ is path--connected, the existence of two $q$--minimal model maps, 
$\Omega^{\bullet}_{\K} (X) \stackrel{f_q}{\leftarrow} \cN^{\bullet} \stackrel{\overline{f_q}}{\rightarrow} \A^{\bullet}$,
where $f_q$ is defined over $\Q$. 

We will say that the space $X$ is $q$--formal (over $\K$) if the $\sCDGA$ $\Omega^{\bullet}_{\K} (X)$ is $q$--formal.
Equivalently, $\Omega^{\bullet} (X, \K)$ is $q$--formal, over $\K=\RR, \C$. When $X$ is path--connected with finite Betti numbers,
it is known that formality (i.e., $\infty$--formality) is independent of $\K$. 

For a pointed space $(X, \pt)$, note that $\Omega^{\bullet}_{\K} (X), \Omega^{\bullet} (X, \K) \in \Ob (\sACDGA)$, since
$\Omega^{\bullet}_{\K} (\pt)= \Omega^{\bullet} (\pt, \K)=\K$, and $\Omega^{\bullet}_{\K} (X) \hookrightarrow \Omega^{\bullet} (X, \K)$
is an $\sACDGA$--map. 

\subsection{Finiteness conditions}
\label{ss25}

If $X$ is a connected CW--complex, we may and we shall suppose (up to homotopy) that $X^{(0)}= \{ \pt \}$. We denote $\pi_1 (X, \pt)$ by $G$.
Let $c: X \to K(G,1)$ be a classifying map. By Theorem \ref{thm:topdr}, 
$c^*: \Omega^{\bullet}_{\K} (K(G,1)):= \Omega^{\bullet}_{\K}(G) \to \Omega^{\bullet}_{\K}(X)$ is a $1$--equivalence.
Hence, $\Omega^{\bullet}_{\K}(X)$ and $\Omega^{\bullet}_{\K}(G)$ have the same $1$--minimal model, 
$\cN =(\bigwedge^{\bullet}W, d)\otimes_{\Q} \K$. 

Assume that $X^{(1)}$ is finite, that is, the group $G$ is finitely generated. Later on, we shall see that the exponential groups
arising from \eqref{eq=ltower} give, in a particularly  convenient way, the so-called {\em $\Q$--Malcev completion} for the nilpotent 
quotients $\{ G/G_n \}_{n\ge 1}$. We recall that the {\em $\Q$--Malcev completion map} of a nilpotent group $N$,
$\kk : N\to \Mal (N)$, is a group homomorphism, universal with respect to morphisms into uniquely divisible groups \cite{Q}.

When looking at cohomology jump loci from an analytic viewpoint, the following finiteness properties emerge naturally,
for $1\le q\le \infty$.

\begin{definition}
\label{def:fin}
A topological space $X$ is {\em $q$--finite} if  $X$ has the homotopy type of a connected CW--complex with finite $q$--skeleton.
A $\sCDGA$ $\A^{\bullet}$ is {\em $q$--finite} if $\A$ is connected and $\dim_{\K} \A^{\le q} < \infty$. 
For $q=\infty$, our convention is that  $X^{(\infty)}=X$ and $\A^{\le \infty}= \A$. 
\end{definition}

\subsection{The $\sCDGA$ setup}
\label{ss26}

Let $X$ be a pointed path--connected topological space. Assume 
$\Omega^{\bullet} (X, \K) \simeq_q \A^{\bullet}$, with $\A$ connected. Denote by $\M$ the common
$q$--minimal model of the $\sCDGA$'s $\Omega^{\bullet} (X, \K)$ and $\A^{\bullet}$. Then there is a diagram in $\sACDGA$,
$\Omega^{\bullet} (X, \K) \stackrel{f}{\longleftarrow} \M^{\bullet} \stackrel{\overline{f}}{\longrightarrow} \A^{\bullet}$, 
where both $f$ and $\overline{f}$ are $q$--equivalences. Furthermore, $\M$ is defined over $\Q$, and we may assume 
that $H^{\bullet}f$ preserves $\Q$--structures. We will denote by $\cN \hookrightarrow \M$ the $1$--minimal model,
giving rise to the induced  diagram in $\sACDGA$, 
$\Omega^{\bullet} (X, \K) \stackrel{f_1}{\longleftarrow} \cN^{\bullet} \stackrel{\overline{f_1}}{\longrightarrow} \A^{\bullet}$,
where both $f_1$ and $\overline{f_1}$ are $1$--equivalences.

\section{Positive weight}
\label{sec:ex}

We discuss, on examples, a property of $\sCDGA$'s introduced by Body--Sullivan and Morgan, see \cite{BS}: the
existence of positive weight decompositions. This notion is related to a natural extension of formality.

\begin{definition}
\label{def:posw}
A {\em weight decomposition} of a $\Q$--$\sCDGA$ $(\A^{\bullet}, d)$ is a direct sum splitting, 
$\A^i =\oplus_{j\in \Z} \A^i_j$, for $i\ge 0$, such that $\A^{\bullet}_j \cdot \A^{\bullet}_k \subseteq \A^{\bullet}_{j+k}$
and $d(\A^{\bullet}_j) \subseteq \A^{\bullet}_j$. We will say that $\A$ {\em has positive weights} if there is 
a weight decomposition with the property that $\A^1_j =0$, for all $j\le 0$. 
\end{definition}

\begin{example}
\label{ex:fposw}
Let $X$ be a $q$--formal (over $\Q$), $q$--finite space. Then there is a length two zigzag of $q$--minimal model maps, 
$\Omega^{\bullet} (X, \K) \leftarrow \M \rightarrow (H^{\bullet}(X, \K), d=0)$. Furthermore, the associated homology isomorphism
from Definition \ref{def:qtype}, $H^{\le q} \Omega (X, \K) \stackrel{\sim}{\leftrightarrow} H^{\le q}(X, \K)$, clearly 
preserves the natural $\Q$--structures. Obviously, the $\sCDGA$ $(H^{\bullet}(X, \K), d=0)$ is defined over $\Q$, and $q$--finite.

Note that $(H^{\bullet}(X, \Q), d=0)$ has the canonical positive weight decomposition $H^i (X, \Q) =H^i_i (X, \Q)$, for all $i$.
\end{example}

\begin{example}
\label{ex:qproj}
(cf. \cite{M}, see also \cite{FM}) Let $X$ be a quasi--projective manifold. Pick any smooth compactification, $X= \overline{X}\setminus D$,
where $D=\cup_{i=1}^n D_i$ is a union of smooth divisors with normal crossings.

There is an associated $\Q$--$\sCDGA$, $(\A^{\bullet}, d)=(\A^{\bullet}(\overline{X},D), d)$, called the {\em Gysin model} of the compactification, constructed as follows.
As a vector space, $\A^k= \bigoplus_{p+l=k} \A^{p,l}$, where $\A^{p,l}= \bigoplus_{\mid S \mid =l} H^p (\cap_{i\in S} D_i, \Q)(-l)$,
where $S$ runs through the $l$--element subsets of $\{ 1,\dots,  n \}$ and $(-l)$ denotes the Tate twist. The multiplication is induced by cup--product, and 
has the property that $\A^{p,l} \cdot \A^{p',l'} \subseteq \A^{p+p',l+l'}$. The differential, $d: \A^{p,l} \to \A^{p+2, l-1}$, is
defined by using the various Gysin maps coming from intersections of divisors. See \cite{FM} for full details. 

Morgan proved in \cite{M} that $\Omega^{\bullet}_{\Q} (X) \simeq_{\infty}(\A^{\bullet}, d)$, in particular
$\Omega^{\bullet} (X, \K) \simeq_{\infty}(\A^{\bullet}, d) \otimes \K$, and the associated homology isomorphism
preserves $\Q$--structures. By definition the weight of $\A^{p,l}$ is $p+2l$, and we clearly obtain a positive weight decomposition of
$(\A^{\bullet}, d)$, inducing  a splitting of Deligne's weight filtration on $H^{\bullet}(X,\Q)$ (as noted in the Introduction for $H^1(X,\Q)$). 

Note also that both $X$ and $\A$ are $\infty$--finite objects. 

Let $\iota: X \to \overline{X}$ be the inclusion and note that $\A^{p,l}$ (with the Tate twist forgotten) is just the $E_2^{p,l}$-term 
$H^p(  \overline{X}, R^l\iota_*\Q_X)$  in the Leray spectral sequence of the inclusion $\iota$, see also \cite[(3.2.4.1)]{De2},
where a different indexing of the spectral sequence is used.
The fact that the cohomology of  $(\A^{\bullet}, d)$ is the cohomology of $X$ is nothing else but the fact that the above spectral sequence 
degenerates at $E_3$, see Corollary (3.2.13) in \cite{De2}.

If $f: X \to X'$ is a morphism of quasi--projective manifolds and if $\overline{X}$ and $\overline{X'}$ are smooth compactifications as above, 
with normal crossing divisors at infinity $D$ and $D'$, such that $f$ extends to a regular map $\overline{f}: \overline{X} \to \overline{X'}$, 
then there is an induced $\sCDGA$ morphism 
$$f^*:\A^{\bullet}(\overline{X'},D') \to \A^{\bullet}(\overline{X},D)$$
which respects the bigrading. This comes from the  functoriality of the Gysin spectral sequence and the fact that the $d_2$-differentials, 
identified to the Gysin morphisms, are morphisms of Hodge structures of type $(-1,-1)$. See also \S 3.2.11 in \cite{De2}, which explains 
in particular how two Gysin models coming from two different compactifications of $X$ are related.
The functoriality with $\C$-coefficients may be obtained by using logarithmic differential forms, see Proposition 3.9 in \cite{M} 
(where special attention is paid to the functoriality of the products involved).

Assume $X'$ is a (connected) curve and $f$ is non-constant (a particularly important case for applications, cf. \cite{A}). 
Then it is well known that the induced morphism $f^*:H^k(\overline{X'},\Q) \to H^k(\overline{X},\Q)$ is injective for all $k$, see 
for instance Lemma 7.28 in \cite{Vbook}. Using this, it is easy to see that 
$$f^*:\A^{\bullet}(\overline{X'},D') \to \A^{\bullet}(\overline{X},D)$$
is also injective in this case. Indeed, the only injectivity that remains to be checked is related to the points $b_i$, the irreducible components 
of $D'$. Consider the divisor $\overline{f}^{-1}(b_i)=\sum_{j \in A_i}m_jD_j$,
where the $D_j$'s are  the irreducible components of $D$ sitting over $b_i$, and all multiplicities $m_j$ are strictly positive integers.  
Then the morphism
$$f^*:\A^{0,1}(\overline{X'},D') \to \A^{0,1}(\overline{X},D)$$
is given by a direct sum of morphisms of type
$$\Q=H^0(b_i,\Q) \to \oplus_{j\in A_i}H^0(D_j,\Q)=\Q^{|A_i|}, ~~ a \mapsto (m_ja)_{j \in A_i},$$
the sum being taken over all the points $b_i$ in $D'$.

\end{example}

\begin{example}
\label{ex:mfds}
For a $C^{\infty}$ manifold $X$, we denote by $\Omega^{\bullet}_{DR} (X, \K)$ the $\sCDGA$ of usual, global $C^{\infty}$ forms.
Note that $\Omega^{\bullet}_{DR} (X, \K) \simeq_{\infty} \Omega^{\bullet} (X, \K)$, and 
the associated homology isomorphism preserves $\Q$--structures. 
\end{example}

\begin{example}
\label{ex:arr}
(cf. \cite{B}, \cite{OS}, \cite{OT}) Let $p=\prod_{i=1}^n p_i$ be a complex polynomial that splits into distinct degree one factors,
defining an affine {\em arrangement} of complex hyperplanes, $\bA$, with quasi--projective complement, $X_{\bA}= \{ p\ne 0\}$.
Denote by $\A^{\bullet}_{\bA}$ the {\em Orlik--Solomon algebra} of $\bA$. It is a $\sCDGA$ defined over $\Z$, with $d=0$,
whose underlying $\sCGA$ is the quotient of the exterior algebra generated in degree $1$, $\bigwedge^{\bullet}(h_1,\dots, h_n)$,
by an ideal depending only on the combinatorics of $\bA$. Plainly, both $X_{\bA}$ and $\A^{\bullet}_{\bA}\otimes \K$ are $\infty$--finite.

The $\sCGA$ map $\bigwedge^{\bullet}(h_1,\dots, h_n) \to \Omega^{\bullet}_{DR} (X_{\bA}, \C)$ defined by
$h_i \mapsto \frac{1}{2\pi \sqrt{-1}} \frac{dp_i}{p_i}$ induces an $\infty$--equivalence in $\sCDGA$, 
$f: \A^{\bullet}_{\bA}\otimes \C \to  \Omega^{\bullet}_{DR} (X_{\bA}, \C)$, with the property that $H^{\bullet}f$ preserves the
natural $\Z$--structures. In particular, $X_{\bA}$ is formal over $\Q$, and $\A^{\bullet}_{\bA}\otimes \Q$ has
positive weight (equal to degree).
\end{example}

\begin{example}
\label{ex:solv}
Let $\bS$ be a $1$--connected, solvable $\RR$--analytic Lie group, with Lie algebra $\fs$.
For a co-compact discrete subgroup $G\subseteq \bS$, let $X=\bS/G$ be the associated $\infty$--finite {\em solvmanifold}.
Left translation in $\bS$ gives rise to an inclusion in $\sCDGA$ between invariant forms,
\[
f: \CC^{\bullet}(\fs \otimes \K)=  \Omega^{\bullet}_{DR} (\bS, \K)^{\bS} \hookrightarrow  \Omega^{\bullet}_{DR} (\bS, \K)^G=
\Omega^{\bullet}_{DR} (X, \K) \, .
\]
When $\ad_x$ has only real eigenvalues, for every $x\in \fs$, $f$ is an $\infty$--equivalence \cite{Hat}. Consequently,
$\Omega^{\bullet} (X, \K) \simeq_{\infty} \CC^{\bullet}(\fs \otimes \K)$, where the $\sCDGA$ $\CC^{\bullet}(\fs \otimes \K)$
is $\infty$--finite.
\end{example}

\begin{example}
\label{ex:nilm}
Let $G$ be a finitely generated, torsion--free nilpotent group. By a classical result of Malcev \cite{Mal}, there is a finite--dimensional, 
nilpotent $\Q$--Lie algebra $E$, with the property that $G$ embeds as a co-compact discrete subgroup in the unipotent group
$\bS= \eexp (E\otimes \RR)$. Note that the associated {\em nilmanifold} $X=\bS/G$ is a $K(G,1)$. 

Since the condition on $\ad_x$, $x\in E\otimes \RR$, from Example \ref{ex:solv} is satisfied (by nilpotence), we infer that
$\Omega^{\bullet} (X, \K):= \Omega^{\bullet} (G, \K) \simeq_{\infty} \CC^{\bullet}(E \otimes \K)$. Note that $E$ is Quillen's
{\em Malcev $\Q$--Lie algebra} of $G$ \cite{Q}, and $\CC^{\bullet}(E)$ is Sullivan's minimal model $\M_{\infty}(\Omega^{\bullet}_{\Q} (G))$
\cite{S}. Conversely, it follows from \cite{S} and \cite{HMR} that, given any  finite--dimensional, 
nilpotent $\Q$--Lie algebra $E$, there is a finitely generated, torsion--free nilpotent group $G$, with $\infty$--finite $K(G,1)$, 
such that $\Omega^{\bullet} (G, \K) \simeq_{\infty} \CC^{\bullet}(E \otimes \K)$, and $\CC^{\bullet}(E \otimes \K)$ is
clearly $\infty$--finite. 
\end{example}

\begin{example}
\label{ex:nopos}
Let $E\in \Ob (\Q-\sLie)$ be finite--dimensional. By duality, $\CC^{\bullet}(E)$ has positive weights if and only if $E$ has positive weights,
that is, $E= \oplus_{j>0}  E^j$, and $[ E^j, E^{j'}]\subseteq E^{j+j'}$. These conditions impose non-trivial restrictions on the algebraic group
$\Aut_{\sLie}(E\otimes \C) \subseteq \GL (E\otimes \C)$. Indeed, let $\C^{\times} \to \Aut_{\sLie}(E\otimes \C)$ be the algebraic $1$--parameter
subgroup associated to the weight splitting. Letting $t\in \C^{\times}$ tend to $0$, positivity of the weights implies that $0\in \End_{\sLie}(E\otimes \C)$
belongs to the Zariski closure of $\Aut_{\sLie}(E\otimes \C)$. 

For example, let $E$ be the $2$--dimensional, non-abelian solvable $\Q$--Lie algebra with basis $\{ x,y \}$ and bracket $[y,x]=x$. 
The group  $\Aut_{\sLie}(E\otimes \C)$ consists of the matrices 
$\left( \begin{smallmatrix} 
a & c\\ 0& 1 \end{smallmatrix} \right)$, with $a\in \C^{\times}$ and $c\in \C$. Clearly, 
$\CC^{\bullet}(E)$ cannot have positive weights. 
\end{example}

\section{Monodromy and Malcev completion}
\label{sec:mono}

We make the first step in our analysis of analytic germs of representation varieties and cohomology jump loci. We reformulate
D. Sullivan's description of the Malcev completion of the fundamental group in terms of the $1$--minimal model, by
bringing into play flat connections and their monodromy representations. 

\subsection{Monodromy of flat connections}
\label{ss41}

We review the classical construction of the monodromy, following \cite{GM}. Let $M$ be a $C^{\infty}$ manifold, and
$\bG$ a $\K$--analytic Lie group, with Lie algebra $\fg$. 

Every $1$--form $\omega \in \Omega^1_{DR} (M, \K) \otimes \fg$ gives rise to a connection in the trivial principal bundle
$M\times \bG$. For a differentiable path $\sigma: \Delta_1 \to M$, we denote by $T_{\sigma}(\omega) \in \bG$ the
{\em parallel transport} of $1\in \bG$ along $\sigma$. Note that $T_{\sigma}(0)=1$, for all $\sigma$. When 
$\omega \in \F (\Omega_{DR} (M, \K) \otimes \fg)$, the {\em monodromy representation} $\mon (\omega): \pi_1(M, \pt) \to \bG$
sends the homotopy class of a $C^{\infty}$ loop, $[\sigma]$, to $T_{\sigma}(\omega)^{-1}$. This defines a group
homomorphism, where the product $[\sigma] \cdot [\tau]$ equals the juxtaposition $[\sigma \star \tau]$ (first $\sigma$, then $\tau$).

Monodromy is natural, in the following sense. Let $f:M \to M'$ be a differentiable map between pointed manifolds, and 
$\varphi: \bG \to \bG'$ a Lie homomorphism, with tangent map $\varphi_*: \fg \to \fg'$. Assume that 
$\omega \in \Omega^1_{DR} (M, \K) \otimes \fg$ and $\omega' \in \Omega^1_{DR} (M', \K) \otimes \fg'$ are {\em compatible}, i.e.,
\begin{equation}
\label{eq=comp}
(f^* \otimes \fg') (\omega')= ( \Omega^1_{DR} (M, \K) \otimes \varphi_*) (\omega) \, .
\end{equation}
Then $T_{f\circ \sigma}(\omega')= \varphi (T_{\sigma}(\omega))$, for any $C^{\infty}$ path $\sigma$ in $M$, and 
$\varphi \circ \mon (\omega)= \mon (\omega') \circ f_{\sharp}$, when both $\omega$ and $\omega'$ are flat, where
$f_{\sharp}: \pi_1(M, \pt) \to \pi_1(M', \pt)$ is the homomorphism induced on fundamental groups.

\subsection{Artin points of group schemes}
\label{ss42}

(cf. \cite{LM}, \cite{GM}) For a $\K$--linear algebraic group $\BB$ in characteristic $0$ (with Lie algebra $\bb$) and
$\Lambda \in \Ob (\sComm)$, we denote by $\BB (\Lambda)$ the bi-functorial group of $\Lambda$--points of $\BB$.
When $A\in \Ob (\sArt)$, $\BB (A)$ is $\K$--linear and contains the unipotent subgroup $\eexp (\bb \otimes \m_A)$ 
constructed in Example \ref{ex:lieart}:
\begin{equation}
\label{eq=bsemi}
\BB (A) =\eexp (\bb \otimes \m_A) \rtimes_{\Ad \otimes \m_A} \BB \, ,
\end{equation}
where $\Ad: \BB \to \Aut_{\sLie} (\bb)$ is the adjoint representation.

In particular, $T_{\sigma}(\omega) \in \eexp (\bb \otimes \m_A)$, if $\omega \in \Omega^1_{DR}(M, \K) \otimes \bb \otimes \m_A$,
and $\mon (\omega): \pi_1(M, \pt) \to \eexp (\bb \otimes \m_A)$ may be viewed as a representation in $\BB (A)$, when $\omega$
is flat. 

Furthermore, for $\psi \in \Hom_{\loc} (A, A')$ and $\omega'= \psi (\omega)$, $\mon (\omega')= \BB (\psi) \circ \mon(\omega)$, 
where $\BB (\psi): \BB (A) \to \BB (A')$ is the associated group homomorphism between Artin points, by naturality. 

\subsection{Topological monodromy}
\label{ss43}
(see e.g. \cite{GT}, for the case $\BB =\GL (V)$) Let $\BB$ be a linear algebraic group over $\RR$ or $\C$. The above considerations 
may be extended from $C^{\infty}$ manifolds to topological spaces, in the following way. Let $X\in \Ob (\sTop)$ be a path--connected 
pointed space, with fundamental group $G=\pi_1 (X, \pt)$ ( a notation that will be kept from now on).

Given $\omega\in \Omega^1 (X, \K)\otimes \bb$ and $\sigma: \Delta_1 \to X$ a continuous path, we denote by $T_{\sigma}(\omega) \in \BB$
the parallel transport along $\id_{\Delta_1}$, with respect to $\omega_{\sigma} \in \Omega^1_{DR} (\Delta_1, \K)\otimes \bb$:
$T_{\sigma}(\omega) := T_{\id_{\Delta_1}} (\omega_{\sigma})$. When $\omega \in \F (\Omega (X, \K)\otimes \bb)$, we obtain in this way
the {\em topological monodromy representation}, $\mon (\omega)\in \Hom_{\gr} (G, \BB)$: it sends the homotopy class of a loop, $[\sigma]$,
to $T_{\sigma}(\omega)^{-1}$. 

Let $f:X\to X'$ be a continuous pointed map, and $\varphi: \BB \to \BB'$ a morphism of linear algebraic groups. If 
$\omega \in \F (\Omega (X, \K)\otimes \bb)$ and $\omega' \in \F (\Omega (X', \K)\otimes \bb')$ satisfy the compatibility condition \eqref{eq=comp},
it is easy to check that $\varphi \circ \mon (\omega)= \mon (\omega') \circ f_{\sharp}$. 

For $A\in \Ob (\sArt)$ and $\omega \in \F_{\loc} (\Omega (X, \K)\otimes \bb \otimes A)$, 
$\mon (\omega)\in \Hom_{\gr}(G, \eexp (\bb \otimes \m_A)) \subseteq \Hom_{\gr}(G, \BB (A))$. If 
$\psi \in \Hom_{\loc} (A, A')$ and $\omega'= \psi (\omega)$, then $\mon (\omega')= \BB (\psi) \circ \mon(\omega)$.

\subsection{Sullivan's $1$--minimal model map}
\label{ss44}

We follow \cite{S}, see also \cite{CP}. We assume that $X$ is (up to homotopy) a connected CW--complex with $X^{(1)}$ finite
(hence, $G$ is finitely generated) and we work over $\K=\Q$. We recall from \S\ref{ss25} that $\Omega_{\Q}(X)$ and $\Omega_{\Q}(G)$
have the same $1$--minimal model, $\cN =(\bigwedge^{\bullet} W, d)$, with canonical filtration described in \S\ref{ss21}. Sullivan
constructs a $1$--minimal model map {\em adapted to $\pi_1$}, with the following properties.

Denote by $\pr_n: G\to G/G_n$ and $p_n: G/G_{n+1} \to G/G_n$ the canonical projections. There is a tower of commuting diagrams
in $\sCDGA$, 
\begin{equation}
\label{eq=stower}
\xymatrix{
\CC^{\bullet} (E^n)  \ar[rr]^{f^n}
\ar[d]^{q^*_n} 
&& \Omega^{\bullet}_{\Q}(G/G_n)
\ar[d]^{p^*_n}   \\
\CC^{\bullet} (E^{n+1}) \ar[rr]^{f^{n+1}}
&& \Omega^{\bullet}_{\Q}(G/G_{n+1})
}
\end{equation}
where each $f^n$ is a minimal model map. The composition
\[
f_1: \cN = \lim \limits_{\rightarrow} \CC^{\bullet} (E^n)
\stackrel{\lim \limits_{\rightarrow} f^n}{\longrightarrow}
\lim \limits_{\rightarrow} \Omega^{\bullet}_{\Q}(G/G_n)
\stackrel{\lim \limits_{\rightarrow} \pr_n^*}{\longrightarrow}
\Omega^{\bullet}_{\Q}(G) \stackrel{c^*}{\rightarrow} \Omega^{\bullet}_{\Q}(X)
\]
is a $1$--minimal model map (where $c$ denotes a classifying map).

The Lie tower \eqref{eq=ltower} begins with $E^1=0$ and the abelian Lie algebra $E^2=G_{\ab} \otimes \Q$
(where $G_{\ab}:= G/G_2$), hence $\CC^{\bullet} (E^2) = (\bigwedge^{\bullet} H^1(G, \Q), d=0)$. Moreover, $H^1 f^2$ 
is the identity of $H^1(G, \Q)= H^1(G_{\ab}, \Q)$.

Denote by $\gr_{\bullet}(G) =\oplus_{n\ge 1} G_n/G_{n+1}$ the {\em associated graded Lie algebra} of a group $G$, with
Lie bracket induced by the group commutator. By construction, both $\gr_{\bullet}(G)$ and $\gr_{\bullet}(G) \otimes \Q$ are
generated in degree $1$, as Lie algebras. Likewise, let $\gr_{\bullet}(E) =\oplus_{n\ge 1} E_n/E_{n+1}$ be the 
associated graded Lie algebra (generated in degree $1$) of a Lie algebra $E$, with bracket induced from $E$. When $E$ is
finite--dimensional nilpotent, we have a natural identification between $\gr_{\bullet}(\eexp (E))$ and $\gr_{\bullet}(E)$, as
graded Lie algebras, which follows from the discussion in \S\ref{ss22} and standard properties of the Campbell--Hausdorff
multiplication.

With this notation, we also have $\Q$--vector space isomorphisms,
\begin{equation}
\label{eq=griso}
\ker (q_n) \cong \gr_n (G)\otimes \Q \, ,
\end{equation}
for all $n$. It follows that we get central extensions of uniquely divisible nilpotent groups,
\[
1 \to \gr_n (G)\otimes \Q \longrightarrow \eexp (E^{n+1}) \stackrel{q_n}{\longrightarrow} \eexp (E^{n}) \to 1 \, ,
\]
and central extensions of nilpotent groups,
\[
1 \to \gr_n (G)  \longrightarrow G/G_{n+1}  \stackrel{p_n}{\longrightarrow} G/G_n \to 1 \, .
\]

\subsection{Associated flat connections}
\label{ss45}

We recall from \S\ref{ss13} that, for each $n$, we have the distinguished element $\omega_{E^n}\in \F (\CC^{\bullet} (E^n) \otimes E^n)$.
Set $\omega_n := f^n (\omega_{E^n}) \in \F (\Omega^{\bullet}_{\Q}(G/G_n)\otimes E^n) \subseteq 
\F (\Omega^{\bullet} (G/G_n, \K)\otimes E^n \otimes \K)$. 

By \S\ref{ss43}, there is an associated representation, $\kk_n := \mon (\omega_n) : G/G_n \to \eexp (E^n \otimes \K)$. Since $\omega_n$ 
is defined over $\Q$, and monodromy representations can be described in terms of iterated integrals (cf. \cite{C}, \cite{Ha1}), it turns out that actually
$\kk_n \in \Hom_{\gr}( G/G_n, \eexp (E^n))$. 

In addition, it is straightforward to  infer from \eqref{eq=stower} that the flat connections $\omega_{n+1}$ and $\omega_n$ satisfy the compatibility condition
\eqref{eq=comp}, with respect to $p_n: K(G/G_{n+1},1) \to K(G/G_{n},1)$ and $q_n : \eexp (E^{n+1} \otimes \K) \to \eexp (E^{n} \otimes \K)$. 
As remarked in \S\ref{ss43}, this gives a tower of commuting diagrams of group homomorphisms,
\begin{equation}
\label{eq=qtower}
\xymatrix{
G/G_{n+1} \ar[rr]^{\kk_{n+1}}
\ar[d]^{p_n} 
&& \eexp (E^{n+1})
\ar[d]^{q_n}   \\
G/G_n \ar[rr]^{\kk_n}
&& \eexp (E^{n})
}
\end{equation}
By \eqref{eq=griso}, restriction to the kernels of the vertical epimorphisms gives a homomorphism
\[
\kk_{n+1}' : \gr_n (G) \rightarrow \gr_n (G)\otimes \Q \, .
\]

\subsection{Abelian monodromy} 
\label{ss46}

Let us assume in \S\ref{ss43} that $\BB$ is abelian. Hence, the Lie algebra $\bb$ is abelian, and the exponential map 
$\eexp_{\BB} : \bb \to \BB$ is a morphism of $\K$--analytic Lie groups, where $\bb$ has the additive group structure.

In this case, $\F (\Omega (X, \K) \otimes \bb)= Z^1 \Omega (X, \K) \otimes \bb$, by \eqref{eq=mc}, where $Z^1$ denotes 
$1$--cocycles. For  $\omega \in \F (\Omega (X, \K) \otimes \bb)$, denote by $[\omega]\in \Hom_{\gr} (G, \K) \otimes \bb$
its cohomology class. By a standard identification, $[\omega]\in  \Hom_{\gr}(G, \bb)$, so $ \eexp_* [\omega]\in  \Hom_{\gr}(G, \BB)$.

In the abelian case, the differential equation defining the parallel transport can be easily integrated in explicit form, by using the
exponential map of $\BB$. As a result, 
\begin{equation}
\label{eq=abtr}
\mon (\omega) =\eexp_* [\omega]\, , \forall \omega \in \F (\Omega (X, \K) \otimes \bb) \, .
\end{equation}

\begin{example}
\label{ex:abexp}
For $\BB =\eexp (E)$, it follows from \S\ref{ss22} that $\eexp_{\BB}=\id_E$. When the Lie algebra $E$ is abelian, clearly $\eexp (E)$ is
the additive group of $E$. In the particular case from \S\ref{ss44}, when $X=K(G_{\ab}, 1)$ and $E=E^2$, the property $H^1f^2=\id$ and 
\eqref{eq=abtr} together imply that $\kk_2 =\mon (\omega_2) : G_{\ab} \to G_{\ab} \otimes \Q$ is the canonical map. In particular, $\kk_2$
is a $\Q$--Malcev completion map \cite{Q}.

In the rank one case, when $\BB =\C^{\times}$, $\eexp_{\BB}$ sends $z\in \C$ to the usual exponential, $e^z \in \C^{\times}$.
\end{example}

\subsection{Sullivan's theorem} 
\label{ss47}

Let us first analyze, for $n\ge2$, the graded Lie maps from \S\ref{ss44}, 
\begin{equation}
\label{eq=grk}
\gr_{\bullet} (\kk_n) \otimes \Q : \gr_{\bullet} (G/G_n) \otimes \Q \rightarrow  \gr_{\bullet} (E^n)\, .
\end{equation}

\begin{lemma}
\label{lem:filtlcs}
For $r\le n$, the map $E^n \to E^r$ from \eqref{eq=ltower} is the canonical projection, $E^n \surj E^n/E^n_r$. 
\end{lemma}

\begin{proof}
Recall from \S\ref{ss21} that $\CC^{\bullet} (E^n)= (\bigwedge^{\bullet} F^n ,d)$, where $\{ F^n \}_{n\ge 1}$ is 
the canonical filtration of $\cN =(\bigwedge^{\bullet} W ,d)$. Denote by $\{ F^r (F^n) \}_{r\ge 1}$ the canonical filtration of 
$\CC^{\bullet} (E^n)$. A straightforward induction on $r$ shows that $F^r (F^n) = F^r$, for $r\le n$. On the other hand, we 
may infer from duality that $F^r (F^n) =(E^n_r)^{\perp}$, the subspace of $E^{n*}=F^n$ orthogonal to $E^n_r \subseteq E^n$;
see e.g. \cite[Lemma 1.1]{MP}. Therefore, $\ker (E^n \surj E^r)= E^n_r$, again by duality.
\end{proof}

\begin{lemma}
\label{lem:grkiso}
The maps \eqref{eq=grk} are isomorphisms, for $n\ge 2$.
\end{lemma}

\begin{proof}
Clearly, the canonical map $G/G_n \surj G/G_2$ induces an isomorphism on $\gr_1$. Likewise, $E^n \surj E^2$ 
induces an isomorphism on $\gr_1$, by Lemma \ref{lem:filtlcs}. We infer from Example \ref{ex:abexp} that
$\gr_1 (\kk_2)\otimes \Q$ is the identity. Using diagrams \eqref{eq=qtower}, we find that  $\gr_1 (\kk_n)\otimes \Q$
is an isomorphism. Consequently, $\gr_{\bullet} (\kk_n)\otimes \Q$ is an epimorphism of graded Lie algebras.

Plainly, $\gr_r (G/G_n)=0$ for $r\ge n$. Similarly, $\gr_r (E^n)=0$ for $r\ge n$, since $E^n_n=0$. Again by
Lemma \ref{lem:filtlcs}, $\gr_r (E^n) \cong \ker (q_r : E^{r+1} \to E^r)$, if $r<n$. When $r<n$, 
$\gr_r(G/G_n) \otimes \Q \cong \gr_r(G) \otimes \Q \cong \ker (q_r)$; see \eqref{eq=griso}. 
The conclusion follows from a dimension argument.
\end{proof}

\begin{theorem}[\cite{S}]
\label{thm:pi1dr}
Let $G$ be a finitely generated group, with $1$--minimal model map adapted to $\pi_1$ as described in \S \ref{ss44}.
Then the monodromy representation $$\kk_n  \in \Hom_{\gr} (G/G_n , \eexp (E^n))$$ constructed in \S \ref{ss45} is a
$\Q$--Malcev completion map, for all $n\ge 2$.
\end{theorem}

\begin{proof}
For $n=2$, we have seen this in Example \ref{ex:abexp}. Assume by induction that $\kk_n$ 
is a $\Q$--Malcev completion map, in \eqref{eq=qtower}. We deduce from Lemma \ref{lem:grkiso} that 
$\gr_n (\kk_{n+1}) \otimes \Q$ is an isomorphism. This map in turn may be easily identified with $\kk_{n+1}' \otimes \Q$,
where $\kk_{n+1}'$ is the restriction of  $\kk_{n+1}$ to the kernels from \eqref{eq=qtower}; see Lemma \ref{lem:filtlcs}.
We thus see that the map $\kk_{n+1}' : \gr_n (G) \to \gr_n (G) \otimes \Q$ is a $\Q$--Malcev completion map \cite{Q}. By
the 5--Lemma in $\Q$--Malcev completion \cite[Corollary 2.6]{HMR}, applied to diagram \eqref{eq=qtower}, $\kk_{n+1}$ is also
a $\Q$--Malcev completion map.  
\end{proof}

\begin{remark}
\label{rem:adapt}
Recall from \S\ref{ss44} the $\pi_1$--adapted $1$--minimal model map over $\Q$, $f_1: \cN \to \Omega_{\Q} (X)$. For
$\K =\RR$ or $\C$, we may tensor it with $\K$ to obtain a $1$--minimal model map over $\K$,
$f_1 \otimes \K : \cN \otimes \K \to \Omega_{\K} (X) \subseteq \Omega (X, \K)$. We shall simply denote it by
$f_1: \cN \to \Omega (X, \K)$, viewing $(\bigwedge^{\bullet} W, d)$ as a $\Q$--structure of $\cN$. We call this $f_1$
a $1$--minimal model map adapted to $\pi_1$, over $\K$. 

As recalled in \S \ref{ss23}, we may extend $f_1$ to a $q$--minimal model map defined over $\Q$,
$f: \M \to \Omega (X, \K)$. This is the $q$--minimal model map that we are going to use, cf. \S \ref{ss26}.

\end{remark}

\subsection{Unipotent completion}
\label{ss48}

(cf. \cite{Q}, \cite{Ha}) We will need to upgrade the conclusion of Theorem \ref{thm:pi1dr}, over 
an arbitrary characteristic $0$ field $\K$. 
To this end, we start by recalling from \cite[\S3]{Ha} a couple of relevant facts about Quillen's {\em $\K$--unipotent completion} 
of a finitely generated group $G$. 

Quillen constructs a tower of finite-dimensional nilpotent $\Q$--Lie algebras, 
$\{ \fp^{n+1} \stackrel{Q_n}{\longrightarrow} \fp^{n} \}$, and a compatible family of group homomorphisms,
$\{ M_n : G \to \eexp (\fp^n) \}$. Each $M_n$ factors through the canonical projection $\pr_n$ and gives a $\Q$--Malcev completion map,
$M_n : G/G_n \to \eexp (\fp^n)$. This construction has another two key properties, for any characteristic $0$ field $\K$. Firstly, the
representations $M_n (\K) : G \to \eexp (\fp^n \otimes \K)$ have Zariski dense images. Secondly, for any group homomorphism into 
a $\K$--unipotent group, $\varphi \in \Hom_{\gr}(G,U)$, there is $\Phi \in  \Hom_{\alggr}(\eexp (\fp^n \otimes \K),U)$, for some $n$,
such that $\varphi =\Phi \circ M_n (\K)$. 

Consider now the tower of $\Q$--unipotent groups, $\{ q_n : \eexp (E^{n+1}) \to \eexp(E^n) \}$,  together with the compatible family of
group homomorphisms $\{ \kk_n \circ \pr_n \}$ from Theorem \ref{thm:pi1dr}. The universality property of $\Q$--Malcev completion,
together with Zariski density, imply that the $\K$--algebraic towers $\{ Q_n (\K): \eexp (\fp^{n+1} \otimes \K) \to \eexp (\fp^{n} \otimes \K) \}$
and $\{ q_n (\K): \eexp (E^{n+1} \otimes \K) \to \eexp (E^{n} \otimes \K) \}$ are isomorphic, with $\{ M_n (\K) \}$ identified with
$\{ \mon (\omega_n) \circ \pr_n \in \Hom_{\gr} (G, \eexp (E^n \otimes \K)) \}$. 

We infer that Theorem \ref{thm:pi1dr} actually gives the $\K$--unipotent completion of $G$.

\begin{corollary}
\label{cor:kunip}
In Theorem \ref{thm:pi1dr}, the map
\[
\lim \limits_{\longrightarrow} \Hom_{\alggr} (\eexp (E^n\otimes \K), U) \rightarrow \Hom_{\gr} (G, U)\, ,
\]
sending $h_n$ to $h_n \circ \mon (\omega_n)\circ \pr_n$, is a natural bijection, for every $\K$--unipotent group $U$.
\end{corollary}

\begin{remark}
\label{rem:malunip}
When $U= \eexp (F)$, where $F$ is a finite--dimensional nilpotent $\K$--Lie algebra, the left hand side of the above bijection
is  identified with $\lim \limits_{\rightarrow} \Hom_{\sLie}(E^n \otimes \K, F)$ in a natural way; see \S\ref{ss22}. In particular, it
depends only on $F$ and the $1$--minimal model of $X$ over $\K$, since the Lie tower $\{ q_n (\K) \}$ is dual to the canonical
filtration of the $1$--minimal model (cf. \S\ref{ss21}).
\end{remark}

We are now ready to establish the following result, that relates representations with unipotent target and flat connections.
Let $X$ be a connected CW--complex with finite $1$--skeleton and fundamental group $G$. Let $f_1 : \cN \to \Omega_{\K} (X)$ 
be a $\pi_1$--adapted $1$--minimal model map over $\K$ (compare with Remark \ref{rem:adapt}).

\begin{prop}
\label{prop:fbij}
For any finite--dimensional nilpotent $\K$--Lie algebra $F$, the composition below is a bijection
\[
\F (\cN \otimes F) \stackrel{f_1}{\longrightarrow} \F (\Omega_{\K} (X) \otimes F)  \stackrel{\mon}{\longrightarrow}
\Hom_{\gr}(G, \eexp (F)) \, .
\]
\end{prop}

\begin{proof}
Plainly, $\F (\cN \otimes F) =\lim \limits_{\rightarrow} \F (\CC^{\bullet}(E^n \otimes \K)\otimes F)$. By Lemma \ref{lem:cflat}
and duality, we may identify $ \F (\CC^{\bullet}(E^n \otimes \K)\otimes F)$ with 
\[
\Hom_{\sCDGA}(\CC^{\bullet}(F),  \CC^{\bullet}(E^n \otimes \K)) = \Hom_{\sLie} (E^n \otimes \K, F)=
\Hom_{\alggr}(\eexp (E^n \otimes \K), \eexp (F)) \, .
\]
Let $\omega \in  \F (\CC^{\bullet}(E^n \otimes \K)\otimes F)$ correspond in this way to
$\lambda \in \Hom_{\alggr}(\eexp (E^n \otimes \K), \eexp (F))$.

Consider the flat connections 
\[
(\pr_n \circ c)^* (\omega_n) \in   \F (\Omega_{\K} (X) \otimes E^n \otimes \K) \, , 
f^n (\omega) \in \F (\Omega_{\K} (G/G_n) \otimes F) \, . 
\]
It is easy to check that they satisfy the compatibility condition 
\eqref{eq=comp}, with respect to the continuous map $\pr_n \circ c: X \to K(G/G_n ,1)$ and the Lie group morphism $\lambda$.
By naturality, we infer like in \S\ref{ss43} that 
$\lambda \circ \mon (\omega_n) \circ \pr_n = \mon (f^n(\omega)) \circ \pr_n= \mon (f_1 (\omega))$.

We thus see that $ \mon \circ f_1 : \lim \limits_{\longrightarrow} \Hom_{\alggr} (\eexp (E^n\otimes \K), \eexp (F)) \rightarrow
\Hom_{\gr}(G, \eexp (F))$ becomes identified with the bijection from Corollary \ref{cor:kunip}.
\end{proof}

\section{Functors of Artin rings}
\label{sec:tdr}

The next step in understanding the relationship between analytic germs of characteristic varieties inside representation
varieties, on one hand, and resonance varieties inside varieties of flat connections, on the other hand, is at the level of
completions of the corresponding local analytic rings. We are going to formulate and prove the results on completion 
needed in the sequel, in the (equivalent) language of functors of Artin rings. 

\subsection{Coordinate rings of representation varieties}
\label{ss51}

Let $G$ be a finitely generated discrete group and $\BB$ a $\K$--linear algebraic group in characteristic $0$. 
Following \cite{LM}, we will denote by $\cH (G, \BB):= \Hom_{\gr} (G, \BB)$ the affine variety of $G$--representations into $\BB$,
with distinguished point $1$, the trivial representation. Let $P$ be the coordinate ring of $\cH (G, \BB)$, with corresponding
maximal ideal $\m$.

There is a {\em universal} group homomorphism, $u: G\to \BB (P)$, inducing a natural bijection, for $\Lambda \in \Ob (\sComm)$, 
\begin{equation}
\label{eq=ubij}
\Hom_{\alg} (P, \Lambda) \stackrel{\sim}{\longrightarrow} \Hom_{\gr} (G, \BB (\Lambda)) \, ,
\end{equation}
given by $\varphi \mapsto \BB (\varphi) \circ u$.

For a rational representation, $\iota: \BB \to \GL (V)$, denote by $\iota_{\Lambda}: \BB (\Lambda) \to \GL_{\Lambda} (V\otimes \Lambda)$
the associated group homomorphism, where $\GL_{\Lambda}$ stands for $\Lambda$--linear isomorphisms. Clearly,
$\GL (V)(\varphi) \circ \iota_{\Lambda} = \iota_{\Lambda'} \circ \BB (\varphi)$, for $\varphi \in \Hom_{\alg}( \Lambda, \Lambda')$. 

Denote by $\K G\in \Ob (\sAlg)$ the group ring of $G$, viewed as a unital, associative $\K$--algebra. For each 
$\varphi \in \Hom_{\alg}(P, \Lambda)$, note that $V\otimes \Lambda$ is a $\K G$--$\Lambda$--{\em bimodule}
(finitely generated free over $\Lambda$), with $\Lambda$ acting canonically on the right, and $\K G$ acting on the left via the
representation $\iota_{\Lambda} \circ \BB (\varphi) \circ u$.

\subsection{The setup}
\label{ss52}

In our main result on analytic germs of non-abelian jump loci, we will start from now on from the following data. 
We assume $\Omega^{\bullet} (X, \K) \simeq_q \A^{\bullet}$ ($1\le q\le \infty$), where both $X$ and $\A$ are $q$--finite objects,
in the sense of Definition \ref{def:fin}. In \S \ref{ss26}, we suppose that $f_1: \cN \to \Omega (X, \K)$ is the
$\pi_1$--adapted $1$--minimal model map constructed in Remark \ref{rem:adapt}, and we denote by $\overline{f_1}: \cN \to \A$
the resulting $1$--minimal model map.

In particular, let $G$ be a finitely generated group, with $1$--finite classifying space $X=K(G,1)$ and classifying map $c=\id$. Pick a
$\pi_1$--adapted $1$--minimal model map, $f_1: \cN \to \Omega (G, \K)$. Let $\A^{\bullet}$ be a $1$--finite $\sCDGA$. Assume
$\overline{f_1}: \cN \to \A$ is a $1$--minimal model map. 

\subsection{Artin approximation} 
\label{ss53}

(cf. \cite{T}, \cite{GM}, \cite{KM}) We work over $\K=\RR$ or $\C$. Let $G$ and $\A$ be as above. Denote by $\bb$ the Lie algebra
of the $\K$--linear group $\BB$. We want to relate the analytic germ at $1$ from \S\ref{ss51}, $\cH (G, \BB)_{(1)}$, with the
analytic germ at $0$ from \S\ref{ss12}, $\F (\A\otimes \bb)_{(0)}$. 

We denote by $R$ the local analytic ring of $\cH (G, \BB)_{(1)}$, and by $\wR$ its $\m$--adic completion. We have two canonical
maps in $\Hom_{\loc}$, $P\to R$ and $R\to \wR$, the first given by the passage from affine varieties to analytic germs and the
second coming from completion. Similar considerations apply to $\F (\A\otimes \bb)_{(0)}$, providing local $\K$--algebra maps 
$\oP \to \oR$ and $\oR \to \woR$. 

Artin's approximation theorem (see for instance \cite[III.4--5]{T}) implies that the analytic germ $\cH (G, \BB)_{(1)}$ is determined
by its {\em functor of Artin rings}, $A\in \Ob (\sArt) \mapsto \Hom_{\loc} (\wR, A) \equiv \Hom_{\loc}(P, A)$, and likewise for
$\F (\A\otimes \bb)_{(0)}$. 

\begin{remark}
\label{rem:schange}
Let $\psi \in \Hom_{\alg} (\wR, S)$ be onto. Then both $\wR$ and $S$ are formal local rings (in the terminology from \cite{T}),
$\psi$ is  local and induces a tower of surjective local homomorphisms, $\{ P\surj A_n \}$, compatible with the canonical local
epimorphisms $A_{n+1}\surj A_n$, where $A_n := S/\m_S^n \in \Ob (\sArt)$, for all $n$.
\end{remark}

\subsection{Analytic germs of representation varieties and Malcev completion}
\label{ss54}

We are going to prove our first main result on analytic germs coming from twisted homology, via functors of Artin rings.
In the setup from \S\ref{ss53}, we recall from \cite{GM} that the functor of Artin rings of $\cH (G, \BB)_{(1)}$
associates to $A$ the set $\Hom_{\loc}(P, A) \equiv \Hom_{\gr}(G, \eexp (\bb \otimes \m_A))$.
The natural identification map sends $\varphi \in \Hom_{\loc}(P, A)$ to $\BB (\varphi) \circ u$; see \eqref{eq=ubij} and \eqref{eq=bsemi}.

We will use the natural bijection from \S\ref{ss12}, 
$\Hom_{\loc}(\oP, A) \equiv \F_{\loc}(\A \otimes \bb \otimes A)$, that sends $\overline{\varphi} \in \Hom_{\loc}(\oP, A)$ to
$\overline{\varphi}(\overline{\omega})$ (where $\overline{\omega}$ is the universal flat connection \eqref{eq=funiv}), to describe 
the functor of Artin rings of $\F (\A\otimes \bb)_{(0)}$.

Denote by $\overline{f_1}: \F_{\loc}(\cN \otimes \bb \otimes A) \stackrel{\sim}{\longrightarrow} \F_{\loc}(\A \otimes \bb \otimes A)$
the natural bijection provided by Theorem \ref{thm:aflat}, where $\overline{f_1}: \cN \to \A$ is the $1$--minimal model map from 
\S\ref{ss52}.

\begin{prop}
\label{prop:fart}
Let $G$ be a finitely generated group, with $\pi_1$--adapted  $1$--minimal model map $f_1: \cN \to \Omega (G, \K)$.
Let $\overline{f_1} : \cN \to \A$ be another $1$--minimal model map, where $\A$ is $1$--finite. Let $\BB$ be a $\K$--linear
algebraic group, with Lie algebra $\bb$. Then the map
\[
\F_{\loc} (\A \otimes \bb \otimes A) \stackrel{\overline{f}_1^{-1}}{\longrightarrow} \F_{\loc} (\cN \otimes \bb \otimes A)
\stackrel{\mon \circ f_1}{\longrightarrow} \Hom_{\gr} (G, \eexp (\bb \otimes \m_A))
\]
is a natural bijection, for  $A \in \Ob (\sArt)$. In particular, the analytic germs $\cH (G, \BB)_{(1)}$ and $\F (\A, \bb)_{(0)}$
are isomorphic.
\end{prop}

\begin{proof}
Apply Proposition \ref{prop:fbij} to $X=K(G,1)$ and $F=\bb \otimes \m_A$.
\end{proof}

\begin{remark}
\label{rem:compliso}
In the topological setup from the beginning of \S \ref{ss52}, we will use the above natural equivalence between functors 
of Artin rings to identify the completions:  $\wR \stackrel{\sim}{\leftrightarrow} \woR$. (Note that the $1$--minimal model maps for $X$ and 
$G$ differ by $c^*: \Omega(G, \K) \to  \Omega(X, \K)$, but this does not change the map $\mon \circ f_1$, by naturality of monodromy,
since $c_{\sharp}=\id_G$.)
\end{remark}

\begin{example}
\label{ex:abe}
Suppose in Proposition \ref{prop:fart} that $\BB$ is abelian. This implies that the map $\overline{f_1}: \F (\cN , \bb) \to \F (\A , \bb)$ 
equals the linear isomorphism $H^1\overline{f_1} \otimes \bb: H^1\cN \otimes \bb \stackrel{\simeq}{\longrightarrow} H^1\A \otimes \bb$, 
since both $\sCDGA$'s are connected.
Set $e:= \mon \circ f_1 \circ \overline{f}_1^{-1}: \F (\A , \bb) \to \cH (G, \BB)$.

By \eqref{eq=abtr}, the (global) analytic map $e$ induces a local analytic map $e:  \F (\A, \bb)_{(0)} \to \cH (G, \BB)_{(1)}$, i.e., 
a local ring morphism $e: R \to \oR$. It is straightforward to check that the induced map on completions,
$\widehat{e}: \wR \stackrel{\simeq}{\longrightarrow} \woR$, is given in terms of functors of Artin rings by the natural equivalence 
constructed in Proposition \ref{prop:fart}. In particular, $e$ is a local analytic isomorphism. 
\end{example}

\begin{theorem}
\label{thm:gralkm}
Let $G$ be a finitely generated group, with $1$--minimal model $\cN$ over  $\K =\RR$ or $\C$. Let $\BB$ be a
$\K$--linear algebraic group, with Lie algebra $\bb$. Then the analytic germ $\cH (G, \BB)_{(1)}$ depends only on
$\cN$ and $\bb$.
\end{theorem}

\begin{proof}
This is another corollary of  Proposition \ref{prop:fbij}.
We may replace  $\cH (G, \BB)_{(1)}$ by its functor of Artin rings. By uniqueness of $1$--minimal models, there is a
$\pi_1$--adapted $1$--minimal model map $f_1: \cN \to \Omega (G, \K)$. (An arbitrary $1$--minimal model map is
$\pi_1$--adapted only up to algebraic homotopy.)
For $A\in \Ob (\sArt)$, we identify 
$\Hom_{\gr}(G, \eexp (\bb \otimes \m_A))$ with $\F_{\loc} (\cN \otimes \bb \otimes A)$
in a natural way, like in Proposition \ref{prop:fart}. 
\end{proof}

We recall that the $1$--minimal model $\cN$ is the dual form of Quillen's \cite{Q} Malcev Lie algebra of $G$, $\Mal (G)$. 
More precisely, $\Mal (G) =\lim \limits_{\leftarrow} E^n \otimes \K$, as complete Lie algebras. We say that the group $G$ is $q$--formal
(over $\K$) if $K(G,1)$ is a $q$--formal space. It is known that a finitely generated group $G$ is $1$--formal if and only if $\Mal (G)$
has a quadratic presentation, as a complete Lie algebra (see e.g. \cite{ABC}). 

Given these remarks, Theorem \ref{thm:gralkm} may be viewed as a substantial extension of 
\cite[Theorem 17.1]{KM},
 where the
conclusion is obtained only for $1$--formal groups. The proof of the general case from Theorem \ref{thm:gralkm} is simple
(and possibly known to experts): the  $1$--minimal model determines the Malcev Lie algebra, which determines the functor of Artin rings,
hence the analytic germ. The reason why our proof is longer is that we need to relate the $1$--minimal model map
and the Malcev completion of $G$ in terms of monodromy representations, with an eye for our subsequent analysis of germs of
characteristic varieties; see Theorem \ref{thm:twdr} below.

\subsection{Functorial twisted De Rham theorem}
\label{ss55}

(cf. \cite{GT}, \cite{H}, \cite{S}, \cite{W}) Let $X$ be a path--connected pointed space, with fundamental group $G$.
Fix a morphism of $\K$--linear groups (over $\K=\RR$ or $\C$), $\iota: \BB \to \GL (V)$, and denote by $\theta: \bb \to \gl (V)$
the tangent map.

For $A\in \Ob (\sArt)$ and $\omega \in \F_{\loc}(\Omega^{\bullet} (X, \K)\otimes \bb \otimes A)$, consider the Aomoto complex
$(\Omega^{\bullet} (X, \K)\otimes V \otimes A, d_{\omega})$ from \eqref{eq=defaom}. Its homology is endowed with the
canonical graded $A$--module structure. 

As explained in \S\ref{ss43}, we have an associated monodromy representation, 
$\mon (\omega): G \to \eexp (\bb \otimes \m_A) \subseteq \BB (A)$.
Set $L(\omega):= \iota_A \circ \mon (\omega): G \to \GL_A (V\otimes A)$. This defines a {\em local system} of $A$--modules on $X$.
Plainly, $L(\omega)$ is the $\K G$--$A$--bimodule $V\otimes A$ constructed in \S\ref{ss51}, where $\varphi \in \Hom_{\loc} (P, A)$
corresponds to $\mon (\omega)$ under the bijection \eqref{eq=ubij}. We denote by $H^{\bullet}(X, L(\omega))$ the graded
{\em twisted cohomology} $A$--module of $X$, with coefficients $L(\omega)$. 

We will need Sullivan's functorial twisted De Rham theorem for topological spaces.

\begin{theorem}[\cite{GT, S}]
\label{thm:twdr}
For $A \in \Ob (\sArt)$ and $\omega \in  \F_{\loc} (\Omega (X, \K) \otimes \bb \otimes A)$, there is a natural isomorphism of 
graded $A$--modules, $H^{\bullet}(\Omega (X, \K) \otimes V \otimes A, d_{\omega}) \cong H^{\bullet}(X, L(\omega)) $.
\end{theorem}

\begin{proof}
The claims follow from Theorem 4.7 and Remark 4.8 of \cite{GT}. Note that Gomez Tato treats in \cite{GT} the case $\iota =\id_{\GL (V)}$,
$A=\RR$ and $\omega \in  \F (\Omega (X, \RR) \otimes \gl (V))$. Nevertheless, his arguments may be easily adapted to our more general case.
\end{proof}

\section{Universal cochains}
\label{sec:tech}

Our approach to germs of jump loci is based on the construction of two universal cochain complexes for  computing twisted cohomology:
the topological cochains come from the universal representation $u$ from \S\ref{ss51}, and the algebraic counterpart is given by 
the universal flat connection $\overline{\omega}$ from \S\ref{ss12}. 

\subsection{Universal topological cochains}
\label{ss61}

We are going to define this object in the following context. Let $X$ be a connected CW--complex with finite $1$--skeleton, with basepoint
$\pt =X^{(0)}$ and (finitely generated) fundamental group $G$. Let $\iota : \BB \to \GL (V)$ be a rational representation of $\K$--linear
algebraic groups in characteristic $0$.

Denote by $C_{\bullet} =C_{\bullet}(X)$ the cellular $\K$--chain complex of $X$. Lift the cell structure of $X$ to the universal cover $\widetilde{X}$,
and let $\widetilde{C_{\bullet}} =C_{\bullet}(\widetilde{X})$ be the corresponding cellular {\em equivariant} $\K$--chains. Note that
$\widetilde{C_{\bullet}} \cong \Gamma \otimes C_{\bullet}$ is free as a graded left $\Gamma$--module (where $\Gamma =\K G$ acts via 
deck transformations on $\widetilde{X}$), and the differential $\partial_{\bullet}$ is $\Gamma$--linear. 

For $\Lambda \in \Ob (\sComm)$ and a right $\Lambda$--module $M\in \Ob (\sMod-\Lambda)$, we recall that a local system $L$ in $\sMod-\Lambda$
on $X$ with fiber $M$ is identified with a $\Gamma$--$\Lambda$ bimodule structure on $M$. Then
\begin{equation}
\label{eq=celltw}
H^{\bullet}(X, L) = H^{\bullet}(\Hom_{\Gamma}(\widetilde{C_{\bullet}}, M)) \, ,
\end{equation}
as graded $\Lambda$--modules, where $\Hom_{\Gamma}$ stands for $\Gamma$--linear maps; see for instance \cite{W}.

This leads to a first (homotopy--invariant) definition.

\begin{definition}
\label{def:charv}
Let $X$ be a connected CW--complex with finite $1$--skeleton (up to homotopy). For $i,r \ge 0$, the {\em characteristic variety} 
(in degree $i$ and depth $r$, relative to the rational representation $\iota : \BB \to \GL (V)$) is
\[
\V^i_r (X, \iota):= \big \{ \rho \in \cH (G, \BB) \, \mid \, \dim_{\K} H^i (X, {}_{\iota \rho} V) \ge r \big \} \, ,
\]
where $G=\pi_1 (X)$ and  the $\K G$--$\K$--bimodule structure of ${}_{\iota \rho} V$ is given by the representation $\iota \circ \rho$.
For $i\ge 0$ and $\rho \in \cH (G, \BB)$, we define the $i$--th {\em twisted Betti number} by 
$b_i (X, \rho):= \dim_{\K} H^i (X, {}_{\iota \rho} V)$.
\end{definition}

For our second definition, we recall from \S\ref{ss51} the affine $\K$--algebra $P$ (the coordinate ring of $\cH (G, \BB)$) and the
$\K G$--$P$ bimodule $V\otimes P$ corresponding to $\id \in \Hom_{\alg}(P, P)$, on which $G$ acts by the representation
$\iota_P \circ u$. 

\begin{definition} 
\label{def:univtop}
The {\em universal} $P$--cochain complex of the connected CW--complex $X$ with finite $1$--skeleton (with respect to the rational
representation $\iota : \BB \to \GL (V)$) is
\[
C^{\bullet}(X, \iota):= \Hom_{\K G} (\widetilde{C_{\bullet}}(X), V\otimes P) \, ,
\]
where $G=\pi_1(X)$ and $V\otimes P$ is  the $\K G$--$P$ bimodule described above.
\end{definition}

For $\Lambda \in \Ob (\sComm)$ and $\varphi \in \Hom_{\alg}(P, \Lambda)$, denote by $L(\varphi)=V \otimes \Lambda$ 
the $\K G$--$\Lambda$ bimodule (i.e., the local system on $X$ in $\sMod-\Lambda$ with fiber $V\otimes \Lambda$)
constructed in \S\ref{ss51}.

\begin{lemma}
\label{cor:phispec}
Let $\varphi \in \Hom_{\alg}(P, \Lambda)$ be surjective. Then the following hold.
\begin{enumerate}

\item \label{tf2}
There is a natural isomorphism in $\Lambda-\sMod$,
\[
H^{*}(X, L(\varphi))\cong H^{*}(C^{\bullet} (X, \iota) \otimes_P \Lambda) \, .
\]

\item \label{tf3}
When $\varphi \in \Hom_{\alg}(P, \K)$, corresponding to the $\K$--point $\rho \in \cH (G, \BB)$,
$\rho \in \V^i_r (X, \iota)$ if and only if $\dim_{\K}H^i(C^{\bullet} (X, \iota) \otimes_P \K) \ge r$, for all
$i$ and $r$. 
\end{enumerate}
\end{lemma}

\begin{proof}
Part \eqref{tf2}. Set $\Gamma =\K G$, $C_{\bullet}= C_{\bullet}(X)$, $\widetilde{C_{\bullet}}= C_{\bullet}(\widetilde{X})$,
$M=V\otimes P$, $C^{\bullet} =C^{\bullet} (X, \iota)$. By definition, $C^{\bullet} =\Hom_{\Gamma} (\widetilde{C_{\bullet}}, M)$.
Denote by $j:M \to M\otimes_P \Lambda$ the canonical $\Gamma-P$--linear map which sends $m$ to $m\otimes 1$.
Clearly, the correspondence $f\mapsto j\circ f$, where $f\in \Hom_{\Gamma} (\widetilde{C_{\bullet}}, M)$, 
gives rise to a morphism of $\Lambda$--cochain complexes,
\begin{equation}
\label{eq:l86}
\Phi \colon \Hom_{\Gamma} (\widetilde{C_{\bullet}}, M)\otimes_P \Lambda  \longrightarrow
\Hom_{\Gamma} (\widetilde{C_{\bullet}}, M \otimes_P \Lambda) \, .
\end{equation}

For a fixed degree $\bullet=i$, denote by $c_i$ the cardinality of a $\K$--basis of $C_i$. Plainly, $\Phi$ in degree $i$ may be identified 
with the canonical map, $M^{c_i}\otimes_P \Lambda \rightarrow (M\otimes_P \Lambda)^{c_i}$, 
associated to the direct product. This map in turn is an isomorphism, since $P$ is Noetherian and $\Lambda$ is finitely generated 
over $P$, via $\varphi$; see \cite[pp.~31--32]{CE}. A routine check shows that the $\K G-\Lambda$--bimodules
$L(\id_P)\otimes_P \Lambda$ and $L(\varphi)$ are naturally isomorphic, 
since $\GL (V)(\varphi) \circ \iota_P= \iota_{\Lambda}\circ \BB (\varphi)$.
Part \eqref{tf2} follows.

Part \eqref{tf3}. By Definition \ref{def:charv} and \eqref{eq=ubij}, $\rho \in \V^i_r (X, \iota)$ if and only if 
$\dim_{\K}H^i(X, L(\varphi)) \ge r$. Due to Part \eqref{tf2}, this is equivalent to 
$\dim_{\K}H^i(C^{\bullet} (X, \iota) \otimes_P \K) \ge r$.
\end{proof}

\subsection{Universal algebraic cochains}
\label{ss64}

Let $\A^{\bullet} \in \Ob (\sCDGA)$ be connected, with $\A^1$ finite--dimensional. 
Given a finite--dimensional Lie module over a finite--dimensional $\K$--Lie algebra, with
structure map $\theta: \bb \to \gl (V)$, let $\oP$ be the coordinate ring of $\F (\A, \bb)$.
We recall from \eqref{eq=funiv} the universal flat connection, $\overline{\omega} \in \F_{\loc} (\A \otimes \bb \otimes \oP)$.

\begin{definition}
\label{def:univalg}
The {\em universal} $\oP$--cochain complex of $\A$ (with respect to the $\bb$--module $V$) is
the universal Aomoto complex \eqref{eq=univaom},
$C^{\bullet}(\A, \theta):= (\A^{\bullet}\otimes V \otimes \oP, d_{\overline{\omega}})$. 
\end{definition}

\begin{definition}
\label{def:resv}
Let $\A^{\bullet}$ be a $1$--finite $\sCDGA$. For $i,r\ge 0$, the {\em resonance variety} (in degree $i$ and depth $r$, with respect to
the finite--dimensional Lie representation $\theta: \bb \to \gl (V)$) is 
\[
\R^i_r (\A, \theta):= \big \{ \omega \in \F (\A, \bb) \, \mid \, \dim_{\K} H^i (\A \otimes V, d_{\omega}) \ge r \big \} \, ,
\]
where $(\A^{\bullet} \otimes V, d_{\omega})$ is the Aomoto complex \eqref{eq=defaom}. For $i\ge 0$ and $\omega \in \F (\A, \bb)$,
denote by $\beta_i (\A, \omega):=\dim_{\K} H^i (\A \otimes V, d_{\omega})$ the $i$--th {\em Aomoto--Betti number}.
\end{definition}

We have the following analog of Lemma \ref{cor:phispec}. 

\begin{lemma}
\label{cor:respec}
The universal $\oP$--cochain complex has the following properties.
\begin{enumerate}

\item \label{rf2}
For $A\in \Ob (\sArt)$ and $\overline{\varphi} \in \Hom_{\loc}( \oP, A)$, the graded $A$--modules 
$H^{*} (\A^{\bullet} \otimes V \otimes A, d_{\overline{\varphi} (\overline{\omega})})$ and
$H^{*} (C^{\bullet} (\A, \theta) \otimes_{\oP} A)$ are naturally isomorphic, 
where $\overline{\omega}$ is the universal flat
connection \eqref{eq=funiv}.

\item \label{rf3}
For $\overline{\varphi} \in \Hom_{\alg}( \oP, \K)$, corresponding to the $\K$--point $\overline{\rho} \in \F (\A, \bb)$, 
$\overline{\rho} \in \R^i_r (\A, \theta)$ if and only if $\dim_{\K} H^i (C^{\bullet} (\A, \theta) \otimes_{\oP} \K) \ge r$, for all
$i$ and $r$.
\end{enumerate}
\end{lemma}

\begin{proof}
A straightforward direct verification shows that the $A$--cochain complex $C^{\bullet} (\A, \theta) \otimes_{\oP} A$
is naturally isomorphic to the Aomoto complex of $\overline{\varphi}(\overline{\omega}) \in \F_{\loc} (\A \otimes \bb \otimes A)$,
$(\A^{\bullet} \otimes V \otimes A, d_{\overline{\varphi} (\overline{\omega})})$. Part \eqref{rf2} follows.

Part \eqref{rf3}. By Definition \ref{def:resv}, 
$\overline{\rho} \in \R^i_r (\A, \theta) \eqv \dim_{\K} H^i (C^{\bullet} (\A, \theta) \otimes_{\oP} \K) \ge r$, 
since the $\K$--cochain complex 
$(\A^{\bullet}  \otimes V, d_{\overline{\rho}})$ is naturally isomorphic to $C^{\bullet} (\A, \theta) \otimes_{\oP} \K$;
see \S \ref{ss14}.
\end{proof}

\subsection{Inverse limits} 
\label{ss72}

To pursue our analysis, we will also need the $1$--minimal model maps 
$\Omega^{\bullet}(X, \K) \stackrel{f_1}{\longleftarrow} \cN \stackrel{\overline{f_1}}{\longrightarrow} \A^{\bullet}$ from
\S\ref{ss52}, that give the canonical identification $\wR \stackrel{\sim}{\leftrightarrow} \woR$ from Remark \ref{rem:compliso}.
Given a $\K$--algebra surjection $\psi : \wR \surj S$, like in Remark \ref{rem:schange}, we obtain in this way another 
$\K$--algebra surjection, $\overline{\psi} : \woR \surj S$. Note that $S= \lim \limits_{\leftarrow} A_n$, where 
$A_n := S/\m_S^n \in \Ob (\sArt)$. 

As noted in Remark \ref{rem:schange}, $\psi = \{ \psi_n \} \in \lim \limits_{\leftarrow} \Hom_{\loc}(P, A_n)$, where each
$\psi_n \in \Hom_{\loc}(P, A_n)$ is onto. By similar considerations, 
$\overline{\psi} = \{ \overline{\psi}_n \} \in \lim \limits_{\leftarrow} \Hom_{\loc}(\oP, A_n)$, and each
$\overline{\psi}_n$ is onto. 

We know from \S\ref{ss12} that $\overline{\psi}_n (\overline{\omega}) \in \F_{\loc}( \A \otimes \bb \otimes A_n)$, for all $n$, and
$\{ \overline{\psi}_n (\overline{\omega}) \} \in \lim \limits_{\leftarrow} \F_{\loc}( \A \otimes \bb \otimes A_n)$.
Applying Theorem \ref{thm:aflat} to $\overline{f_1}$, we infer that $\overline{\psi}_n (\overline{\omega})= \overline{f_1} (\omega'_n)$,
for all $n$, and $\omega' := \{ \omega'_n \} \in \lim \limits_{\leftarrow} \F_{\loc}( \cN \otimes \bb \otimes A_n)$. Set
$\omega_n := f_1 (\omega'_n)$ and $\omega := \{ \omega_n \} \in \lim \limits_{\leftarrow} \F_{\loc}(\Omega (X, \K) \otimes \bb \otimes A_n)$.  
As remarked in \S\ref{ss43}, $\{ \mon (\omega_n) \} \in \lim \limits_{\leftarrow} \Hom_{\gr} (G, \eexp (\bb \otimes \m_{A_n}))$.

The fact that $\psi$ and $\overline{\psi}$ are canonically identified translates to the following property.

\begin{lemma}
\label{lem:phid}
With notation as above, $\mon (\omega_n)$ is equal to $\BB (\psi_n) \circ u$ in $\Hom_{\gr}(G, \BB (A_n))$, for all $n$,
where $u$ is the universal representation from \S\ref{ss51}.
\end{lemma}

\begin{proof}
As explained in Remark \ref{rem:compliso}, the canonical identification $\wR \stackrel{\sim}{\leftrightarrow} \woR$ corresponds to
the natural identification between the two functors of Artin rings, constructed in Proposition \ref{prop:fart}. Therefore, we know that
$\{ \psi_n \} \in \lim \limits_{\leftarrow} \Hom_{\loc}(P, A_n)$ is identified with 
$\{ \overline{\psi}_n \} \in \lim \limits_{\leftarrow} \Hom_{\loc}(\oP, A_n)$, via the natural bijection from Proposition \ref{prop:fart}. 
In other words, we have the equality 
$\{ \BB (\psi_n) \circ u \}= \{ (\mon \circ f_1 \circ \overline{f}_1^{-1})(\overline{\psi}_n (\overline{\omega})) \}$ in
$\lim \limits_{\leftarrow} \Hom_{\gr}(G, \BB (A_n))$, according to Proposition \ref{prop:fart} and the discussion preceding it.

By construction, $\overline{f}_1^{-1} (\overline{\psi}_n (\overline{\omega}))= \omega'_n$ and $f_1 (\omega'_n) =\omega_n$,
for all $n$, proving our claim.
\end{proof}

We will need one more result on inverse limits.

\begin{lemma}
\label{lem:mlinv}
Let $S\in \Ob (\sAlg)$ and $\m_S \subseteq S$ be an ideal. Assume that $S$ is complete and Hausdorff for the
$\m_S$--adic topology, and $\dim_{\K}(S/\m_S^n)< \infty$, for all $n$.  Let $C^{\bullet}$ be a right
$S$--cochain complex. Suppose that $C^i$ is finitely generated free over $S$ for $i\le q$ and $C^{q+1}$ is
$\m_S$--Hausdorff. Then the natural $S$--linear map 
$$H^{\bullet} (C) \rightarrow \lim \limits_{\leftarrow} H^{\bullet} (C\otimes_S (S/\m_S^n))$$ 
is an isomorphism for $\bullet \le q$ and a monomorphism for $\bullet =q+1$.
\end{lemma}

\begin{proof}
This is a consequence of standard facts about exactness properties of sequential inverse limits (see for instance 
\cite[pp. 126--132]{Sw}). We first claim that the natural map,
$H^{\bullet} (\lim \limits_{\leftarrow} (C/C \cdot \m_S^n)) \rightarrow \lim \limits_{\leftarrow} H^{\bullet}(C/C \cdot \m_S^n)$,
is an isomorphism, for $\bullet \le q+1$. To verify this, denote by $\{ C^{\bullet}_n \}$ the inverse system of $S$--cochain complexes
$\{ C^{\bullet}/C^{\bullet} \cdot \m_S^n \}$. Let $\{ Z^{\bullet}_n \}$, $\{ B^{\bullet}_n \}$ and $\{ H^{\bullet}_n \}$
be the associated inverse systems of cocycles, coboundaries and cohomologies. Clearly, the natural map identifies
$Z^{\bullet} (\lim \limits_{\leftarrow} C_n)$ with $\lim \limits_{\leftarrow} Z^{\bullet}_n$, in such a way that
$B^{\bullet} (\lim \limits_{\leftarrow} C_n) \subseteq \lim \limits_{\leftarrow} B^{\bullet}_n$. Let $d$ be the differential of
$\lim \limits_{\leftarrow} C_n$. 

Since $S/\m_S^n$ is finite--dimensional over $\K$, for all $n$, and  $C^{\le q}$ is finitely generated over $S$, the inverse systems
$\{ Z^{\le q}_n \}_n$  and $\{ B^{\le q+1}_n \}_n$ have the Mittag--Leffler property. Due to this fact, we obtain exact sequences,
\[
0\to \lim \limits_{\leftarrow} Z^{\bullet}_n \longrightarrow \lim \limits_{\leftarrow} C^{\bullet}_n \stackrel{d}{\longrightarrow}
\lim \limits_{\leftarrow} B^{\bullet +1}_n \to 0 \, ,
\]
for $\bullet \le q$, and
\[
0\to \lim \limits_{\leftarrow} B^{\bullet}_n \longrightarrow \lim \limits_{\leftarrow} Z^{\bullet}_n
\longrightarrow \lim \limits_{\leftarrow} H^{\bullet}_n \to 0 \, ,
\]
for $\bullet \le q+1$.
We infer that $B^{\le q+1} (\lim \limits_{\leftarrow} C_n) = \lim \limits_{\leftarrow} B^{\le q+1}_n$, and the natural map,
\[
H^{\le q+1} (\lim \limits_{\leftarrow} C_n) \longrightarrow \lim \limits_{\leftarrow} H^{\le q+1}(C_n) \, ,
\]
is an isomorphism, as asserted.

In order to finish the proof of the lemma, it is therefore enough to check that the natural cochain map given by 
$\m_S$--adic completion, $\kappa: C^{\bullet} \to \lim \limits_{\leftarrow}  C^{\bullet}\otimes_S (S/\m_S^n)$,
is a $q$--equivalence. This in turn follows from the fact that $\kappa$ is an isomorphism in degree $\bullet \le q$,
and a monomorphism for $\bullet =q+1$, due to our hypotheses.
\end{proof}

\begin{example}
\label{ex:l96}
Let $S$ be a $\K$--algebra satisfying the assumptions from Lemma \ref{lem:mlinv} (e.g., a quotient algebra of $\wR$,
like in Remark \ref{rem:schange}). Given $\varphi \in \Hom_{\alg} (P, S)$ (respectively $\overline{\varphi} \in \Hom_{\alg} (\oP, S)$),
set $C^{\bullet}= C^{\bullet}(X, \iota)\otimes_P S$ (respectively $C^{\bullet}= C^{\bullet}(\A, \theta)\otimes_{\oP} S$).
Denote by $c_i$ the cardinality of the $i$--cells of $X$ (respectively of a $\K$--basis of $\A^i$). Set $\ell= \dim_{\K} V$. 
Then $C^i$ is $S$--isomorphic to $P^{\ell c_i} \otimes_P S$ (respectively to the direct sum 
$S^{(\ell c_i)}= \oP^{(\ell c_i)} \otimes_{\oP} S$) for all $i$. When both $X$ and $\A$ are $q$--finite, we infer that $C^{\le q}$
is finitely generated free over $S$, in both cases. In the first case, the natural map $P^{\ell c_i} \otimes_P S \to S^{\ell c_i}$
is an embedding for all $i$, since $P$ is Noetherian (cf. \cite[pp.~31--32]{CE}). We infer that $C^i$ embeds into the direct product
$S^{\ell c_i}$, in both cases. In particular, $C^i$ is $\m_S$--Hausdorff for all $i$. Hence, the $S$--cochain complex  $C^{\bullet}$
meets all the requirements from Lemma \ref{lem:mlinv}.
\end{example}

\subsection{Comparing the universal complexes}
\label{ss73}

We may now put things together to obtain the following conclusion, needed for the proof of our main results.

\begin{prop}
\label{prop:comp}
Let $\psi \in \Hom_{\alg} (\wR, S)$ be onto, identified with  $\overline{\psi} \in \Hom_{\alg} (\woR, S)$ by
the canonical isomorphism  $\wR \stackrel{\sim}{\leftrightarrow} \woR$.
Then there is an isomorphism of graded $S$--modules, 
$H^{\le q} (C (X, \iota) \otimes_P S) \cong H^{\le q} (C (\A, \theta) \otimes_{\oP} S)$, 
where $P\to S$ is induced by $\psi$ and $\oP \to S$ is induced by $\overline{\psi}$. 
\end{prop}

\begin{proof}
By Example \ref{ex:l96}, Lemma \ref{lem:mlinv} provides $S$--isomorphisms, 
\[
H^{\le q}(C (X, \iota) \otimes_P S) \cong \lim \limits_{\leftarrow} H^{\le q}(C (X, \iota) \otimes_P A_n)
\]
and
\[
H^{\le q}(C (\A, \theta) \otimes_{\oP} S) \cong \lim \limits_{\leftarrow} H^{\le q}(C (\A, \theta) \otimes_{\oP} A_n)\, ,
\]
where $P\to A_n$ is $\psi_n$ and $\oP \to A_n$ is $\overline{\psi}_n$ . 

The graded $S$--modules 
$\lim \limits_{\leftarrow} H^{\le q}(\A \otimes V \otimes A_n, d_{\overline{\psi}_n (\overline{\omega})})$ and
$\lim \limits_{\leftarrow} H^{\le q}(C (\A, \theta) \otimes_{\oP} A_n)$ are isomorphic, by Lemma \ref{cor:respec}\eqref{rf2}.

By virtue of Theorem \ref{thm:adr}, we get a graded $S$--isomorphism, 
$\lim \limits_{\leftarrow} H^{\le q}(\M \otimes V \otimes A_n, d_{\omega'_n}) \cong 
\lim \limits_{\leftarrow} H^{\le q}(\A \otimes V \otimes A_n, d_{\overline{\psi}_n (\overline{\omega})})$, since
$\overline{f}(\omega'_n)= \overline{f_1}(\omega'_n)= \overline{\psi}_n (\overline{\omega})$ for all $n$, by construction.
A similar argument gives another graded $S$--isomorphism,  
$\lim \limits_{\leftarrow} H^{\le q}(\M \otimes V \otimes A_n, d_{\omega'_n}) \cong 
\lim \limits_{\leftarrow} H^{\le q}(\Omega (X, \K)  \otimes V \otimes A_n, d_{\omega_n})$.

We may invoke Theorem \ref{thm:twdr} to obtain a graded $S$--isomorphism,  
$$\lim \limits_{\leftarrow} H^{\le q}(\Omega (X, \K)  \otimes V \otimes A_n, d_{\omega_n}) \cong
\lim \limits_{\leftarrow} H^{\le q}(X, L(\omega_n)) \, ,$$ 
where each local system $L(\omega_n)$ is given by the representation
$\iota_{A_n} \circ \mon (\omega_n): G \to \GL_{A_n} (V\otimes A_n)$, cf. \S\ref{ss55}. At this point, we may resort to
Lemma \ref{lem:phid} to identify $L(\omega_n)$ with the local system $L(\psi_n)$, associated to $\psi_n \in \Hom_{\alg}(P, A_n)$
by the construction from \S\ref{ss51}, for all $n$.

Finally, by virtue of Lemma \ref{cor:phispec}\eqref{tf2}, there is a graded $S$--isomorphism,
$$\lim \limits_{\leftarrow} H^{\le q}(X, L(\psi_n)) \cong \lim \limits_{\leftarrow} H^{\le q}(C (X, \iota) \otimes_P A_n)\, ,$$ 
since each $\psi_n \in \Hom_{\alg}(P, A_n)$ is surjective, as noted in \S\ref{ss72}.

These graded $S$--isomorphisms together establish our claim.
\end{proof}

\begin{remark}
\label{rem:finhyp}
Our goal in this paper is to compute topological jump loci near $1\in \cH (G, \BB)$, up to a fixed degree $q\ge 1$,
replacing $X$ by an algebraic representative of its $q$--type, $\A$ (in the sense of Definition \ref{def:qtype}).
In order to make meaningful local statements, we need an affine variety structure on the ambient space $\cH (G, \BB)$,
which leads to our finiteness assumption on $X^{(1)}$. Similar considerations about the algebraic ambient space
$\F (\A, \bb)$ force the finiteness of $\A^1$. In concrete terms, we want to avoid examples like $X= \bigvee_{n\in \N} S^1$
and $\A^{\bullet}= (H^{\bullet}X, d=0)$ (having the same $\infty$--type, by formality), where things go wrong already
in the rank one case (i.e., for $\iota= \id_{\C^{\times}}$).

Similar simple examples show the necessity of our higher finiteness assumptions from Definition \ref{def:fin}. The reason is 
that we need to pass from the Artinian case from Theorem \ref{thm:adr} to the pro--Artinian setup, like in 
Proposition \ref{prop:comp}. This in turn requires the conclusion of Lemma \ref{lem:mlinv} in all degrees up to $q$,
for both universal complexes. Let us examine for any $q\ge 2$ the example when $X= S^1 \vee (\bigvee_{n\in \N} S^q)$
and $\A^{\bullet}= (H^{\bullet}X, d=0)$ (again by formality), in the rank one case. Clearly, $\oP= \C [t]$,
$\widehat{\oP}=\woR= \C [[t]]:= S$, and we may take $\overline{\psi}=\id_S$ in Proposition \ref{prop:comp}. Therefore,
we need  the conclusion of Lemma \ref{lem:mlinv} in degree $q$, for 
$C^{\bullet}= (H^{\bullet}X \otimes S, d_{\overline{\omega}} \otimes_{\oP} S)$. 

It is straightforward to check that $d_{\overline{\omega}}=0$ on $H^+X \otimes \oP$, using \eqref{eq=funiv}. Hence, 
the conclusion of Lemma \ref{lem:mlinv} in degree $q$ is simply that $C^q$ is $(t)$--complete and Hausdorff. But
$C^q=H^qX \otimes S$ cannot be complete, since $H^qX$ is infinite--dimensional. This shows that both $\A$ and $X$
must be supposed to be $q$--finite.
\end{remark}

\section{Proof of the main results}
\label{sec:pfs}

Theorem \ref{thm:km0} was proved in Section \ref{sec:tdr}. We continue with the isomorphism theorem for analytic germs of
cohomology jump loci, Theorem \ref{thm:8main}.

\subsection{Jump loci}
\label{ss100}

We first examine the natural (reduced) affine structure of characteristic and resonance varieties. To this end, we start by recalling
the definition of the jump loci of a $P$--cochain complex $C^{\bullet}$, where $P$ is an affine $\K$--algebra. Given a field
extension, $\K \subseteq \bK$, and $i,r\ge 0$, set
\begin{equation}
\label{eq:specjump}
\V^i_r (C, \bK) = \{ \rho\in \Hom_{\K-\alg}(P, \bK) \; \mid \; \dim_{\bK} H^i(C\otimes_P \bK) \ge r \}\, .
\end{equation}

Under suitable assumptions, these loci are Zariski--closed, in the following sense.

\begin{lemma}
\label{lem:zarspec}
Assume that $C^i$ is finitely generated free over $P$ for $i\le q$, and $C^{q+1}$ is either 
a direct sum $P^{(c)}$ or  a direct product $P^c$, for some set $c$. Then, for $i\le q$ and $r\ge 0$,
there is an ideal $\E^i_r \subseteq P$ such that $\rho \in \V^i_r (C, \bK)$ if and only if $\rho (\E^i_r)=0$. 
\end{lemma}

\begin{proof}
Denote by $\{ d_i: C^i \to C^{i+1} \}_{i\ge 0}$ the differential. By linear algebra, it is enough to consider, for
$i\le q$, the associated elementary $P$--cochain complexes concentrated in degrees $i$ and $i+1$,
$C^i \stackrel{d_i}{\longrightarrow} C^{i+1}$, and to check for these the claim in degree $i$. This in turn is
obvious for $i<q$, given our finiteness assumptions: the corresponding ideals are generated by the appropriate
minors of the $P$--matrix of $d_i$. So, let $P^{(n)} \stackrel{d}{\longrightarrow} C$ be $P$--linear, where $n$
is finite and either $C=P^{(c)}$ or $C=P^c$.

In the first case, we may write $d=j\circ d'$, where  $P^{(n)} \stackrel{d'}{\longrightarrow} P^{(f)}$, with
$f\subseteq c$ finite. Since $j:  P^{(f)} \to P^{(c)}$ is split injective, the complexes associated to $d$ and $d'$ 
have the same jump loci in degree $i=q$, and we are done. In the second case, denote by $d(\rho)$ 
the $\bK$--specialization $d\otimes_P \bK : \bK^{(n)} \to P^c \otimes_P \bK$, and let $j:  P^c \otimes_P \bK \to \bK^c$
be the canonical map. Since $j$ is injective (cf. \cite[pp.~31--32]{CE}), we infer from \eqref{eq:specjump} that
\[
\rho \in \V^q_r (P^{(n)} \stackrel{d}{\longrightarrow} P^c, \bK) \eqv \dim_{\bK} \ker (j\circ d(\rho))\ge r \, .
\]
Since $n$ is finite, this in turn is equivalent to $\dim_{\bK} \ker (\pi_f \circ j\circ d(\rho))\ge r$, for all finite subsets $f\subseteq c$,
where $\pi_f: \bK^c \surj \bK^f$ is the canonical projection. Again, this condition translates to the vanishing of appropriate minors.
\end{proof}

Now, assume $\Omega^{\bullet}(X, \K) \simeq_q \A^{\bullet}$, where both $X$ and $\A$ are $q$--finite,
like in \S \ref{ss52}. As we saw in Example \ref{ex:l96}, both universal complexes, 
$C^{\bullet}(X, \iota)$ and $C^{\bullet}(\A, \theta)$, satisfy the assumptions from Lemma \ref{lem:zarspec}.

\begin{corollary}
\label{lem:redjump}
For $i\le q$, the following hold.
\begin{enumerate}
\item \label{rj1}
For all $r\ge 0$, there is an ideal, $\cI^i_r \subseteq P$, such that 
$\rho \in \V^i_r (C^{\bullet}(X, \iota), \bK) \subseteq \Hom_{\alg}(P, \bK)$ 
if and only if $\rho (\cI^i_r)=0$. In particular, 
$\V^i_r (X, \iota)= V(\cI^i_r ) \subseteq \cH (G, \BB)$.

\item \label{rj2}
For all $r\ge 0$, there is an ideal, $\overline{\cI}^i_r \subseteq \oP$, such that 
$\overline{\rho} \in \V^i_r (C^{\bullet}(\A, \theta), \bK) \subseteq \Hom_{\alg}(\oP, \bK)$ 
if and only if $\overline{\rho} (\overline{\cI}^i_r)=0$. In particular, 
$\R^i_r (\A, \theta)= V(\overline{\cI}^i_r) \subseteq \F (\A, \bb)$.

\item \label{rj3}
When $r>0$, both $\V^i_r (X, \iota)$ and $\R^i_r (\A, \theta)$ are void, excepting finitely many pairs $(i,r)$.
\end{enumerate}
\end{corollary} 

\begin{proof}
Parts \eqref{rj1} (respectively \eqref{rj2}) follow from Lemma \ref{cor:phispec}\eqref{tf3} (respectively  
Lemma \ref{cor:respec}\eqref{rf3}), and Lemma \ref{lem:zarspec}. Part \eqref{rj3}
is a direct consequence of definition \eqref{eq:specjump}, given
the finiteness properties of universal cochains.
\end{proof}

\begin{example}
\label{ex:nps}
Let $X$ be a formal CW--complex, connected and finite. We may take in the above corollary 
$\A^{\bullet}= (H^{\bullet}(X, \K), d=0)$. When $\iota=\id_{\C^{\times}}$, $\F (\A, \bb)=H^1(X, \C)$ and 
$\oP= \Sym (H_1(X, \C))$ is a polynomial algebra. The universal Aomoto complex is 
$C^{\bullet}( \A, \theta)= H^{\bullet}(X, \C) \otimes \oP$, with $\oP$--linear differential $d_{\overline{\omega}}$
sending $\alpha\otimes 1$ to $\sum_i a_i \cup \alpha \otimes a_i^*$, where $\{ a_i \}$ is a basis for $H^1(X, \C)$
and $\{ a_i^* \}$ is the dual basis; see \eqref{eq=funiv} and \eqref{eq=univaom}. When the CW--complex $X$ is minimal
(e.g., for an arrangement complement), this universal Aomoto complex was described in \cite{PS} as the linearization
of the equivariant cochain complex of the universal abelian cover of $X$, and the associated jump loci \eqref{eq:specjump}
were identified with the resonance varieties $\R^i_r (X)$. When the arrangement $\bA$ is central, essential and indecomposable,
Budur found in \cite{Bu} a particularly nice description of the defining ideals $\overline{\cI}^i_r$ from
Corollary  \ref{lem:redjump} \eqref{rj2}, in terms of Fitting ideals, and derived from this several interesting
consequences for $\R^i_r (X_{\bA})$.
\end{example}

Corollary \ref{lem:redjump} takes care of the natural reduced structure of relative jump loci. Now, we recall from \S\ref{ss52} that $X$ and $\A$
are related by the assumption $\Omega^{\bullet}(X, \K) \simeq_q \A^{\bullet}$. In particular, 
$\dim_{\K} H^i (X, \K)= \dim_{\K} H^i \A =: b_i$, for $i\le q$. Set $\ell := \dim_{\K} V$. We start our local analysis of 
jump loci with an easy remark on basepoints. Recall from \S\ref{ss53} the canonical change of rings maps
$P\to R \to \wR$ (respectively $\oP \to \oR \to \woR$), where $R$ (respectively $\oR$) is the local analytic ring of
$\cH (G, \BB)_{(1)}$ (respectively $\F (\A, \bb)_{(0)}$), and $\wR=\widehat{P}$ (respectively $\woR=\widehat{\oP}$) 
denotes completion.

\begin{remark}
\label{rem:basepts}
The equivalences below follow easily from Corollary \ref{lem:redjump}, for $i\le q$ and $r\ge 0$: 
$\wR \cI^i_r$ is a proper ideal of $\wR \eqv 1\in \V^i_r (X, \iota) \eqv \ell b_i \ge r \eqv 0\in \R^i_r (\A, \theta) \eqv \woR \, \overline{\cI}^i_r$
is proper in $\woR$. Otherwise, clearly both analytic germs  $\V^i_r (X, \iota)_{(1)}$ and $\R^i_r (\A, \theta)_{(0)}$ are empty.
\end{remark}

\begin{example}
\label{ex:mnev}
The fact that the jump loci, up to degree $q$, of a cochain complex with finitely generated free components, in degrees $i\le q$,
are Zariski closed does {\em not} follow from general principles. On the contrary, the construction \eqref{eq:specjump} 
exhibits the following `universality' property. Let $Z\subseteq \K$ be an {\em arbitrary} subset. For any $q$, consider the 
$P$--cochain complex $C^{\bullet}$, concentrated in degrees $q$ and $q+1$, defined by $P \stackrel{d}{\longrightarrow} C$,
where $P=\K [t]$. Here $C$ is the quotient of the free $P$--module generated by $w$ and $\{ w_z \}_{z\in Z}$, modulo
the relations $\{ w= (t-z)w_z \}_{z\in Z}$, and the $P$--linear map $d$ sends $1$ to the class of $w$ in $C$. 
A straightforward computation shows that $\V^q_1 (C, \K)= Z$. 
\end{example}

\subsection{Preliminaries on completions}
\label{ss81}

To simplify notation, we fix $i\le q$ and $r\ge 0$ and set $\widehat{\cI}:= \wR \cdot \cI^i_r$,
$\widehat{\overline{\cI}}:= \woR \cdot \overline{\cI}^i_r$, cf. Corollary \ref{lem:redjump}. 
We view both $\widehat{\cI}$ and $\widehat{\overline{\cI}}$ as
ideals in $\wR$, via the canonical identification $\wR \stackrel{\sim}{\leftrightarrow} \woR$ used in \S\ref{ss72} and \S\ref{ss73}.

For an arbitrary prime ideal $\fp \subseteq \wR$, we set $S:=\wR/\fp$, we denote by $\psi: \wR \surj S$ the canonical projection, and by
$S\subseteq \bK$ the inclusion into the field of fractions.

\begin{lemma}
\label{lem:compl1}
We have the equivalence $\sqrt{\widehat{\cI}} \subseteq \fp \eqv \sqrt{\widehat{\overline{\cI}}} \subseteq \fp$.
\end{lemma}

\begin{proof}
Plainly, $\sqrt{\widehat{\cI}} \subseteq \fp$ if and only if $\psi (\widehat{\cI})=0$ in $\bK$, and similarly for $\widehat{\overline{\cI}}$.
By Corollary \ref{lem:redjump}, $\psi (\widehat{\cI})=0$ in $\bK$ if and only if $\dim_{\bK} H^i(C (X, \iota) \otimes_P \bK) \ge r$, and likewise
for $\widehat{\overline{\cI}}$.

We know from Proposition \ref{prop:comp} that $H^i(C (X, \iota) \otimes_P S)$ is $S$--isomorphic to $H^i(C (\A, \theta) \otimes_{\oP} S)$.
Therefore, $H^i(C (X, \iota) \otimes_P \bK)$ and $H^i(C (\A, \theta) \otimes_{\oP} \bK)$ are isomorphic, over $\bK$. Our claim follows.
\end{proof}

\begin{lemma}
\label{lem:compl2}
The ideals $\wR \cdot \sqrt{R\cdot \cI^i_r}$ and $\woR \cdot \sqrt{\oR\cdot \overline{\cI}^i_r}$ are identified via
$\wR \stackrel{\sim}{\leftrightarrow} \woR$, for all $i\le q$ and $r\ge 0$.
\end{lemma}

\begin{proof}
With the notation from Lemma \ref{lem:compl1}, recall from Remark \ref{rem:basepts} that the ideals $\widehat{\cI}$ and
$\widehat{\overline{\cI}}$ are simultaneously proper. Assuming this, let $\widehat{\cI}= \cap_{\alpha =1}^n \fq_{\alpha}$
be a primary decomposition. We infer that $\sqrt{\widehat{\cI}}= \cap_{\alpha =1}^n \sqrt{\fq_{\alpha}}$. By virtue of Lemma
\ref{lem:compl1} for $\fp= \sqrt{\fq_{\alpha}}$, $\sqrt{\widehat{\overline{\cI}}} \subseteq \sqrt{{\widehat{\cI}}}$.
Repeating this argument for $\widehat{\overline{\cI}}$ gives us the identification 
$\sqrt{\wR\cdot \cI^i_r} \equiv \sqrt{\woR\cdot \overline{\cI}^i_r}$, for all $i\le q$ and $r\ge 0$.

On the other hand, $\sqrt{\wR\cdot \cI^i_r}= \wR \cdot \sqrt{R\cdot \cI^i_r}$, and similarly for $\overline{\cI}^i_r$
(see e.g. \cite[p.~37]{T}). This completes the proof.
\end{proof}

\subsection{Proof of Theorem \ref{thm:8main} \eqref{8m1}}
\label{ss82}

We are going to establish our result by making use of M.~Artin's theorem (see e.g. \cite[pp.~58-65]{T}), that allows the
approximation of formal series solutions of analytic equations by convergent solutions. We will need to find a (reduced)
analytic approximation, $e$, of the formal isomorphism $\wR \stackrel{\sim}{\leftrightarrow} \woR$, compatible with {\em all}
germs of jump loci in degrees up to $q$. This is the reason why we have to examine in detail certain steps of the proof given
in \cite{T}.

Set $G=\pi_1 (X)$, as usual. As soon as the local analytic isomorphism 
$e: \cH (G, \BB)_{(1)} \stackrel{\sim}{\longleftrightarrow} \F (\A, \bb)_{(0)}$ is found, we may obviously suppose $r>0$ in
our second claim (on jump loci), since $\V^i_0 (X, \iota)= \cH (G, \BB)$ and $\R^i_0 (\A, \theta)= \F (\A, \bb)$, for all $i$.
Due to Corollary \ref{lem:redjump}, we have to verify the second claim only for finitely many pairs $\alpha:= (i,r)$. At the same time,
we may replace $\V^i_r (X, \iota)$ by $V(\cI^i_r)$, and $\R^i_r (\A, \theta)$ by $V(\overline{\cI}^i_r)$. Moreover, we may
assume that both $\wR\cdot \cI^i_r$ and $\woR\cdot \overline{\cI}^i_r$ are proper ideals, cf. Remark \ref{rem:basepts}.

Denoting $\K \{ x_1,\dots, x_n \}$ by $\OO_n$ and $\K \{ \bar{x}_1,\dots, \bar{x}_{\bar{n}} \}$ by $\OO_{\bar{n}}$,
write $R=\OO_n/J_0$ and $\oR=\OO_{\bar{n}}/\overline{J}_0$, where the ideals $J_0$ and $\overline{J}_0$ are generated by
finitely many convergent series vanishing at $0$. For $\alpha= (i,r)$, consider the proper ideals $\sqrt{R\cdot \cI^i_r} \subseteq R$ 
and $\sqrt{\oR\cdot \overline{\cI}^i_r} \subseteq \oR$, written as  $\sqrt{R\cdot \cI^i_r}= J_{\alpha}/J_0$ and
$\sqrt{\oR\cdot \overline{\cI}^i_r}=  \overline{J}_{\alpha}/\overline{J}_0$, where the finitely generated ideals 
$J_{\alpha}\subseteq \OO_n$ and $\overline{J}_{\alpha} \subseteq \OO_{\bar{n}}$ are contained in the corresponding 
maximal ideals. 

The existence of $e_0 \in \Hom_{\loc}(\widehat{\OO}_{\bar{n}}, \widehat{\OO}_n)$, sending 
$\widehat{\overline{J}}_0$ into $\widehat{J}_0$ and
$\widehat{\overline{J}}_{\alpha}$ into
$\widehat{J}_{\alpha}$ for all $\alpha$, is an Artin--type problem that has a formal solution. We denote by $e_0$ a formal local
homomorphism that induces the canonical identification, 
$e_0:  \woR=\widehat{\OO}_{\bar{n}}/\widehat{\overline{J}}_0 \stackrel{\sim}{\rightarrow} \widehat{\OO}_n/\widehat{J}_0= \wR$,
used in \S\ref{ss72}. Lemma \ref{lem:compl2} implies that $e_0$ is indeed a formal solution of our problem, since 
$e_0 : \woR \stackrel{\sim}{\rightarrow} \wR$ actually identifies  $\widehat{\overline{J}}_{\alpha}$ with $\widehat{J}_{\alpha}$ 
for all $\alpha$. By Artin approximation, we may find $e \in \Hom_{\loc}(\OO_{\bar{n}}, \OO_n)$, whose derivative at the origin
equals the corresponding truncation of $e_0$, sending  $\overline{J}_0$ into $J_0$ and
$\overline{J}_{\alpha}$ into $J_{\alpha}$ for all $\alpha$.

In particular, we obtain a local homomorphism $e: \oR \to R$. To check that $e$ is an isomorphism, it is enough to show that
$\widehat{e}: \woR \to \wR$ is an isomorphism on $\m$--adic completions. By the Hopfian property of the Noetherian ring
$\wR \cong \woR$, it suffices to show that $\widehat{e}$ is onto. Since we are dealing with complete objects, this amounts to the 
surjectivity of the associated graded, $\gr_{\bullet}(\widehat{e})$. Since the associated graded rings are generated in degree
$\bullet =1$, we are left with checking that $\gr_{1}(\widehat{e})=\gr_1 (e)$ is onto. By construction, $\gr_1 (e)=\gr_1 (e_0)$
is an isomorphism, and we are done.

For each $\alpha= (i,r)$, set $I_{\alpha}:= \sqrt{R\cdot \cI^i_r} \subseteq R$ and 
$\overline{I}_{\alpha}:= \sqrt{\oR\cdot \overline{\cI}^i_r} \subseteq \oR$. We know that $e$ induces a local homomorphism between
analytic algebras, $\widetilde{e}: \oR/\overline{I}_{\alpha} \to R/I_{\alpha}$, and the completions, 
$\woR/\widehat{\overline{I}}_{\alpha}$ and  $\wR/\widehat{I}_{\alpha}$, are identified via $e_0$. By the same arguments as before,
$\widetilde{e}$ is an isomorphism if and only if $\gr_1(\widetilde{e})$ is onto. This in turn follows from the corresponding property of 
$\gr_1(e)$. So, $e: \oR \stackrel{\simeq}{\rightarrow} R$ identifies $\sqrt{\oR \cdot \overline{\cI}^i_r}$ with $\sqrt{R \cdot \cI^i_r}$,
for all $\alpha=(i,r)$. 

Consequently, $e: \oR_{\red} \stackrel{\simeq}{\rightarrow} R_{\red}$ gives a reduced local analytic isomorphism,
$$e: \cH (G, \BB)_{(1)} \stackrel{\simeq}{\longrightarrow} \F (\A, \bb)_{(0)}\, ,$$ 
inducing an isomorphism of analytic germs,
$\V^i_r (X, \iota)_{(1)} \stackrel{\simeq}{\longrightarrow} \R^i_r (\A, \theta)_{(0)}$, for all $\alpha =(i,r)$. The proof of
Theorem \ref{thm:8main}\eqref{8m1} is complete. \hfill $\square$

\subsection{Proof of Theorem \ref{thm:8main} \eqref{8m2}}
\label{ss83}

We recall from Example \ref{ex:abe} and \eqref{eq=abtr} that $\eexp_{\BB}$ may be identified with the global analytic map
$e: \F (\A, \bb) \to \cH (G, \BB)$ constructed via monodromy representations. Consider the induced local analytic morphism,
$e: \F (\A, \bb)_{(0)} \to \cH (G, \BB)_{(1)}$, corresponding to the local ring morphism $e:R \to \oR$. 

As pointed out in Example \ref{ex:abe}, $\widehat{e}:\wR \to \woR$ coincides with the canonical identification,
$\wR \stackrel{\sim}{\leftrightarrow} \woR$, used in \S\ref{ss72}. In particular, $e:R \stackrel{\sim}{\rightarrow} \oR$ is
an isomorphism, and $\widehat{e}$ identifies  $\wR \cdot \sqrt{R\cdot \cI^i_r}$ with $\woR \cdot \sqrt{\oR\cdot \overline{\cI}^i_r}$,
for all $i\le q$ and $r\ge 0$, according to Lemma \ref{lem:compl2}.

By standard commutative algebra (see e.g. \cite[Chapter I]{T}), $\wR$ is faithfully flat over $R$ and therefore $e$ identifies 
$\sqrt{R\cdot \cI^i_r}$ with $\sqrt{\oR\cdot \overline{\cI}^i_r}$, for all $i\le q$ and $r\ge 0$. As before, we infer that 
$\eexp_{\BB}: \F (\A, \bb)_{(0)} \stackrel{\simeq}{\longrightarrow} \cH (G, \BB)_{(1)}$ gives a reduced local analytic isomorphism,
inducing isomorphisms of analytic germs, $\R^i_r (\A, \theta)_{(0)} \stackrel{\simeq}{\longrightarrow} \V^i_r (X, \iota)_{(1)}$,
for all $i\le q$ and $r\ge 0$. \hfill $\square$

\subsection{Immediate corollaries}
\label{ss84}

\begin{corollary}
\label{cor:qp}
Let $X$ be a quasi--projective manifold. Then, for every compactification of $X$, the associated Gysin model $\A^{\bullet}$ described in
Example \ref{ex:qproj} has the property that the analytic germs $\cH (\pi_1 (X), \BB)_{(1)}$ and $\F (\A \otimes \K, \bb)_{(0)}$ are
isomorphic, by an isomorphism that identifies the analytic germs $\V^i_r (X, \iota)_{(1)}$ and $\R^i_r (\A \otimes \K, \theta)_{(0)}$, 
for all $i,r\ge 0$.
\end{corollary}

The next corollary of Theorem \ref{thm:8main} follows directly from the definition \ref{def:qtype} of partial formality.

\begin{corollary}
\label{cor:qformal}
Let $X$ be a connected CW--complex with finite $q$--skeleton ($1\le q\le \infty$), up to homotopy. If $X$ is $q$--formal over $\K$,
then the analytic germs $\cH (\pi_1 (X), \BB)_{(1)}$ and $\F ((H^{\bullet} (X, \K), d=0), \bb)_{(0)}$ are isomorphic, in such a way that
the analytic germs $\V^i_r (X, \iota)_{(1)}$ and $\R^i_r ((H^{\bullet} (X, \K), d=0), \theta)_{(0)}$ are identified, for all $i\le q$ and $r\ge 0$. 
\end{corollary}

When $\A^{\bullet}=(H^{\bullet} (X, \K), d=0)$, the definition of the relative resonance varieties $\R^i_r (\A , \theta)$ from 
Definition \ref{def:resv} coincides with the notion used in \cite{DPS}, where the $1$--formal case (over $\C$) was analyzed,
in degree $i=1$. This remark shows that Corollary \ref{cor:qformal} extends Theorem A (one of  the main results from \cite{DPS})
to arbitrary degree $q$ of partial formality, and in arbitrary degree $i\le q$; see also the alternative definition of $1$--formality
recorded at the end of \S\ref{ss54}. 

A simple particular case of representation varieties and cohomology jump loci that was extensively studied in the literature
is the {\em rank one case}, where $\BB= \GL_1 (\C)$ and $\iota =\id_{\C^{\times}}$. In this situation, $\cH (G, \BB)$ is
usually denoted $\T (G)$, and has the particularly simple structure of an affine torus. Nevertheless, the characteristic varieties
$\V^i_r (X, \iota)$ (denoted simply by $\V^i_r (X)$) turn out to be pretty complicated objects, in general. In this case, it is easily checked 
in Definition \ref{def:resv} that $\R^i_r (\A , \theta):=\R^i_r (\A)$ coincides with
\[
\big \{ \omega \in Z^1 \A \, \mid\, \dim_{\K}H^i (\A^{\bullet},  d+ \omega \cdot) \ge r \big \} \, ,
\]
where $\omega \cdot$ denotes left--multiplication by $\omega$ in $\A^{\bullet}$. When $\A^{\bullet}= (H^{\bullet} (X, \C), d=0)$,
$\F (\A, \bb)$ becomes the affine space $H^{1} (X, \C)$, $\R^i_r (\A , \theta):=\R^i_r (X)$ are the usual resonance varieties, intensively studied 
in arrangement theory, and $\eexp_{\BB}: H^{1} (G, \C) \to \T (G)$ is the usual exponential, as explained in Example \ref{ex:abexp}.

When $X=X_{\bA}$ is a complex hyperplane arrangement complement, we may view $X$ either as a quasi--projective manifold, or 
as a $\Q$--formal space, cf. Example \ref{ex:arr}. By Corollary \ref{cor:qformal} (for $q=\infty$), all analytic germs at $1$ of
(non--abelian) characteristic varieties of $X$ may be computed using $\A^{\bullet}= (H^{\bullet} (X, \K), d=0)$, which seems a better choice than 
the more complicated Gysin models. In particular, the analytic germs at $1$ of relative characteristic varieties are combinatorially determined.

\begin{corollary}
\label{cor:solvgerms}
Let $X=\bS/G$ be a solvmanifold, where the Lie algebra $\fs$ of $\bS$ satisfies the condition from Example \ref{ex:solv}.
Then the analytic germs $\cH (G, \BB)_{(1)}$ and $\F (\CC^{\bullet} (\fs \otimes \K), \bb)_{(0)}$
are isomorphic,  and the isomorphism identifies the analytic germ of  $\V^i_r (X, \iota)$ at $1$ with that of
$\R^i_r (\CC^{\bullet} (\fs \otimes \K), \theta)$ at $0$, for all $i,r \ge 0$.
\end{corollary}

\begin{corollary}
\label{cor:nilgerms}
Let $G$ be a finitely generated nilpotent group, with Malcev $\Q$--Lie algebra $E$ \cite{Q} and $\Q$--minimal model 
$\M =\CC^{\bullet} (E)$ \cite{S}. Then the analytic germs $\cH (G, \BB)_{(1)}$ and $\F (\M\otimes \K, \bb)_{(0)}$
are isomorphic by an isomorphism identifying the germs $\V^i_r (K(G, 1), \iota)_{(1)}$ and $\R^i_r (\M\otimes \K, \theta)_{(0)}$,
for all $i,r \ge 0$.
\end{corollary}

\begin{proof}
We are going to reduce the proof to the torsion--free case, examined in Example \ref{ex:nilm}. Consider the exact sequence 
$1\to T \to G \stackrel{p}{\longrightarrow} G_0 \to 1$, where $T$ is the (finite) torsion subgroup of $G$ (see \cite{HMR}), and
$G_0$ is finitely generated nilpotent and torsion--free. Since $T$ is finite, $p$ induces an isomorphism between Malcev $\Q$--Lie algebras, 
hence $\M$ is also the $\Q$--minimal model of $G_0$ (cf. \cite{HMR, Q}). By Theorem \ref{thm:8main}, our statement 
holds when $G$ is replaced by $G_0$. 

Next, we claim that $p$ induces a reduced local analytic isomorphism, 
$p^*: \cH (G_0, \BB)_{(1)} \stackrel{\simeq}{\rightarrow} \cH (G, \BB)_{(1)}$. This may be checked at the level of functors of Artin rings,
by showing that $p^*: \Hom_{\gr} (G_0, \eexp (\bb \otimes \m_A)) \to  \Hom_{\gr} (G, \eexp (\bb \otimes \m_A))$ 
is a bijection, for any $A\in \Ob (\sArt)$. We invoke Corollary \ref{cor:kunip} to see that the last claim is equivalent to 
the bijectivity of $p^*: \Hom_{\alggr} (\eexp (\Mal (G_0)), \eexp (\bb \otimes \m_A)) \to  \Hom_{\alggr} (\eexp (\Mal (G)), \eexp (\bb \otimes \m_A))$;
see also the discussion at the end of \S\ref{ss54}. This in turn is clear, since $p: \Mal (G)  \stackrel{\simeq}{\rightarrow}  \Mal (G_0)$
is an isomorphism.

Finally, we will verify that $p^*: H^{\bullet} (G_0, {}_{\rho_0} V) \to H^{\bullet} (G, {}_{\rho_0 \circ p} V)$ is an isomorphism, for any
$\rho_0 \in \cH (G_0, \GL (V))$. This will clearly imply that $p^*: \cH (G_0, \BB)_{(1)} \stackrel{\simeq}{\rightarrow} \cH (G, \BB)_{(1)}$
identifies the germs $\V^i_r (K(G_0, 1), \iota)_{(1)}$ and $\V^i_r (K(G, 1), \iota)_{(1)}$, for all $i,r\ge 0$, thus finishing the proof of the corollary.

To check our final claim, we consider the Hochschild--Serre spectral sequence of the extension
$1\to T \to G \stackrel{p}{\longrightarrow} G_0 \to 1$ (see \cite{Br}),
\[
E_2^{s,t}= H^s(G_0, H^t(T,V)) \Longrightarrow H^{s+t}(G, {}_{\rho_0 \circ p} V) \, .
\]
In characteristic $0$, $H^+(T, V)=0$, since $T$ is finite. It follows that the spectral sequence degenerates to the isomorphisms
$p^s: H^{s} (G_0, {}_{\rho_0} V) \stackrel{\simeq}{\longrightarrow} H^{s} (G, {}_{\rho_0 \circ p} V)$, $\forall s$, as asserted. 
\end{proof}

Corollaries \ref{cor:solvgerms} and \ref{cor:nilgerms} offer a convenient description of germs at $1$ of arbitrary non--abelian 
jump loci, for two important classes of poly--cyclic groups.

\subsection{Weighted tangent cone}
\label{ss85}

A weight decomposition of a $\sCDGA$ $\A^{\bullet}$ gives rise to a $1$--parameter group homomorphism,
$\C^{\times} \rightarrow \Aut_{\sCDGA} (\A \otimes \C)$, denoted $t \cdot a$ for $t\in \C^{\times}$ and 
$a\in \A^{\bullet} \otimes \C$, defined by $t \cdot a= t^j a$, for $a\in \A^{\bullet}_j \otimes \C$. When
$\dim_{\Q} \A^1 <\infty$, the positivity of the weights in degree $i=1$ implies that the $\C^{\times}$--action 
extends to an algebraic map, $(\C, 0)\to (\gl (\A^{1} \otimes \C), 0)$.

We start the proof of Theorem \ref{thm:tori} with a simple, useful remark. Let $\A^{\bullet}$ be a $1$--finite 
$\C$--$\sCDGA$, endowed with a weight decomposition. Given a finite--dimensional Lie representation, 
$\theta: \bb \to \gl (V)$, consider the $1$--parameter group homomorphism, 
$\C^{\times} \to \Aut_{\sDGL}(\A^{\bullet} \otimes (V \rtimes_{\theta} \bb))$,  
where $t\in \C^{\times}$ acts by $(t \cdot)  \otimes (V \rtimes_{\theta} \bb)$. By direct inspection of Definition \ref{def:resv}, 
we infer that the $\C^{\times}$--action on $\A^1 \otimes \bb$ leaves both $\F( \A, \bb)$ and $\R^i_r (\A, \theta)$ ($i,r\ge 0$)
invariant.

Next, the identification $H^1(X, \C) \stackrel{\sim}{\leftrightarrow} H^1 \A$ may be used to endow $H^1 (G, \C)$ with a
positive weight $\C^{\times}$--action, where $G=\pi_1 (X)$. 

\begin{definition}
\label{def:wtgcone}
The {\em weighted exponential tangent cone} of a Zariski closed subset, $\V \subseteq \T (G)$, is
\[
WETC_1 (\V) := \big \{ z\in H^1 (G, \C) \, \mid \, \eexp (t\cdot z) \in \V \, , \forall t\in \C \big \} \, ,
\]
where $\eexp: H^1 (G, \C) \to \T (G)$ is the usual exponential.
\end{definition}

Set $n=b_1(G)$ and recall that the identification $H^1(G,\C) \stackrel{\sim}{\leftrightarrow} H^1 \A$ preserves
$\Q$--structures. This implies that there exist $(a_{ij})= M\in \GL_n (\Q)$ and $\{ w_i \in \N_+ \}_{1\le i \le n}$
such that, for any $t\in \C$ and $z\in H^1(G,\C) \equiv \C^n$, $t\cdot z= M D(t) M^{-1}(z)$, where $D(t)$ is the
diagonal matrix with entries $\{ t^{w_i} \}_{1\le i \le n}$. The next result extends Lemma 4.3 from \cite{DPS}.

\begin{lemma}
\label{lem:ratl}
For any Zariski closed subset $\V \subseteq \T (G)$, $WETC_1 (\V)$ is a finite union of rational linear 
subspaces of $H^1(G,\C)$. 
\end{lemma}

\begin{proof}
Since clearly $WETC_1 (\V)$ depends only on the analytic germ at $1$ of $\V$, it is enough to prove our claim for
$\V= V(f)$, where $f\in \C [t_1^{\pm 1}, \dots, t_n^{\pm 1}]$ is a non--zero Laurent polynomial with $f(1)=0$.
Moreover, we may replace $WETC_1 (V(f))$ by
\begin{equation}
\label{eq=ef}
\big \{ z\in \C^n \, \mid \, f(\eexp (M D(t)(z)))= 0\, , \forall t\in \C \big \}\, .
\end{equation}

Write $f= \sum_{u\in S} c_u t_1^{u_1} \cdots t_n^{u_n}$, where the support $S\subseteq \Z^n$ is finite and non--empty,
and $c_u \in \C^{\times}$ for all $u\in S$. Fix $z\in \C^n$. For $u\in S$, define $f_u (t)\in \C[t]$ by
$f_u (t)= \sum_{1\le i,j\le n} a_{ij}u_i z_j t^{w_j}$, and note that $f_u(0)=0$. Let $\cP$ be the set of partitions,
$S= S_1 \coprod \dots \coprod S_k$, with the property that $\sum_{u\in S_i} c_u =0$, for $i=1, \dots, k$. 
For $p\in \cP$, define the rational linear subspace $T(p)$ of $\C^n$ by
\[
T(p)= \big \{ z\in \C^n \, \mid \, f_u(t)= f_v(t) \, \text{in}\,  \C [t]\, , \forall u,v\in S_i \, , \forall 1\le i \le k \big \} \, .
\]

We claim that the solutions of \eqref{eq=ef} coincide with $\bigcup_{p\in \cP} T(p)$. Indeed, \eqref{eq=ef} means that
$\sum_{u\in S} c_u \eexp (f_u(t))=0$, for all $t$. Denoting by $p$ the partition of $S$ corresponding to the distinct values,
$\{ f_1,\dots, f_k \}$, of $f_u (t)\in \C [t]$ for $u\in S$, this implies that 
$\sum_{1\le i\le k} (\sum_{u\in S_i} c_u) \eexp (f_i(t))\equiv 0$, and the polynomial $f_i(t)-f_j(t)$ is not constant, for
$1\le i\ne j\le k$. By an old result of E.~Borel \cite{EB}, on linear independence of exponentials of entire functions,
we infer that $p\in \cP$ and $z\in T(p)$. Conversely, if $z\in T(p)$ with $p\in \cP$, then clearly \eqref{eq=ef} holds.
The claim is verified, and the proof of our lemma is complete.
\end{proof}

\subsection{Proof of Theorem \ref{thm:tori}}
\label{ss86}

Recall the identifications $\F (\A, \C) =H^1 \A \stackrel{\sim}{\leftrightarrow} H^1(G, \C)$.
For a given $z\in H^1(G,\C)$ and $i\le q$, $r\ge 0$, $z\in WETC_1 (\V^i_r(X))$ if and only if $t\cdot z\in \R^i_r(\A)$
for $t\in \C$ near $0$, by virtue of Theorem \ref{thm:8main}\eqref{8m2}. By the remark preceding Definition \ref{def:wtgcone}, 
we infer that $\R^i_r(\A)=  WETC_1 (\V^i_r(X))$, and all irreducible components of $\R^i_r(\A)$ are $\C^{\times}$--invariant.
Therefore, Part \eqref{8t1} is a consequence of Lemma \ref{lem:ratl}.
Moreover, $\eexp (\R^i_r(\A))$ is a finite union of connected affine subtori of $\T (G)$. Again by Theorem \ref{thm:8main}, 
$\V^i_r(X)$ coincides with $\eexp (\R^i_r(\A))$, near $1\in \T(G)$. Part \eqref{8t2} follows. \hfill $\square$

\begin{example}
\label{ex:counterc}
The conclusion of Theorem \ref{thm:tori}\eqref{8t2} does not hold without strong assumptions on $X$. Here is a
simple example, extracted from \cite[Example 4.5]{PS1}, where $X$ is a connected finite $2$--dimensional CW complex
with $\T(G)= (\C^{\times})^2$, and $\V^1_1(X)$ has a single (non--translated) component, the curve $\{ t_1+t_2=2\}$,
which contains no positive--dimensional connected affine subtorus.

The same thing happens with  Theorem \ref{thm:tori}\eqref{8t1}. Indeed, let us consider the  $\sCDGA$ 
$\A^{\bullet}= \CC^{\bullet} (E\otimes \C)$, where $E$ is the $2$--dimensional $\Q$--Lie algebra without positive weights
from Example \ref{ex:nopos}. For $\theta=\id_{\C}$, an easy computation shows that 
$\R^1_1( \CC^{\bullet} (E\otimes \C), \theta) = \{ 0,1\} \subseteq \C= \F (\CC^{\bullet} (E\otimes \C), \C)$ is non--linear.

The non--linearity of $\R^1_1( X)$ was pointed out in \cite[Example 10.1]{DPS}, in the case when $X$ is the configuration space
of $n$ distinct points on an elliptic curve with $n\ge 3$. Note that for this family of quasi--projective manifolds linearity holds,
when $\A^{\bullet}= (H^{\bullet}(X, \C), d=0)$ is replaced by a Gysin model. By Theorem \ref{thm:tori}\eqref{8t2},
the germ at $\rho=1$ of $\V^1_1(X)$ is linear, yet $\R^1_1( X)$, the degree $1$ and depth $1$ jump locus of the
Wang \cite{Wa} Aomoto complex $\aom^{\bullet}(X, 1)$, is non--linear.
\end{example}

As explained in Examples \ref{ex:fposw} and \ref{ex:qproj}, Theorem \ref{thm:tori} applies in particular to $q$--formal spaces over $\Q$
and quasi--projective manifolds (where $q=\infty$). Note that the conclusions of Theorem \ref{thm:tori} are valid (for $q=1$)
for $X=K(G,1)$, where the group $G$ is finitely generated and $1$--formal over $\C$, and $\A^{\bullet}= (H^{\bullet} (G, \C), d=0)$.
Indeed, we may define another $1$--finite $\sCDGA$, $\B^{\bullet}$, with $d=0$ and positive weight equal to degree, by
setting $\B^{\le 1}=\A^{\le 1}$, $\B^{>2}=0$ and $\B^2= D\A^2$, the image of the cup--product map $\cup: \wedge^2 \A^1 \to \A^2$.
Clearly, $\F (\A,\C)=\F (\B, \C)$ and $\R^{\le 1}_r (\A)= \R^{\le 1}_r (\B)$, for all $r\ge 0$. Moreover, 
$\Omega^{\bullet} (G,\C) \simeq_1 (\B^{\bullet}, d=0)$, by $1$--formality, and we may achieve that the associated identification,
$H^1(G, \C) \stackrel{\sim}{\leftrightarrow} H^1(G, \C)$, is the identity. In this way, we may see that Theorem \ref{thm:tori}
extends one of the main results in \cite{DPS} (Theorem B from that paper), from $1$--formal groups to $q$--formal spaces
($1\le q\le \infty$).

\subsection{Proof of Corollary \ref{cor:subMHS}}
\label{ss86a}

First note that the mixed Hodge structure on $H^1(X,\Q)$ is split, in the sense that there is a natural direct sum decomposition 
$$H^1(X,\C)=H^{1,0}(X)\oplus H^{0,1}(X)\oplus H^{1,1}(X),$$
see for instance \cite{A} or \cite{Dcompo}. Because of this splitting, it follows that such a mixed Hodge structure can be described 
(as it is the case for pure Hodge structures) by a linear $\C^{\times}$--action on 
$H^1(X,\C)$, given by setting
$$z \cdot u=z^p\overline z ^qu \text {    for  } u \in H^{p,q}(X).$$
To show that any irreducible component $E$ of $\R^i_r (\A, \theta)$ (which we know already to be defined over $\Q$) is a sub--MHS 
in $H^1(X,\Q)$, it is enough to show that $E$ is stable with respect to the above $\C^{\times}$--action, 
exactly as in the proof of Lemma 2 from \S 3  in \cite{V}.

Coming back to the notation from Example \ref{ex:qproj}, let $(\A^{\bullet}, d)$ be the Gysin model of $X$ with respect to some smooth compactification.
Note that each of the summands $\A^{p,l}= \bigoplus_{\mid S \mid =l} H^p (\cap_{i\in S} D_i, \Q)(-l)$ is a pure Hodge structure of weight $p+2l$ 
and as such its complexification carries a natural $\C^{\times}$--action defined exactly as above.

The multiplication is defined by the cup--product which is compatible with the Hodge types. Hence we see that the condition
$\A^{p,l} \cdot \A^{p',l'} \subseteq \A^{p+p',l+l'}$ can be refined to the condition
$$ z\cdot (u \wedge v)= (z\cdot u) \wedge (z\cdot v),$$
for any $u,v \in \A^{\bullet}$ and any $z \in \C^{\times}$, which takes care of the Hodge types $(s,t)$ as well.

The differential, $d: \A^{p,l} \to \A^{p+2, l-1}$, being 
defined by using the various Gysin maps coming from intersections of divisors,  preserves the Hodge types 
(here the Tate twist is essential) and this can be restated as
$$d(z\cdot u)=z \cdot d(u),$$
for any $u \in \A^{\bullet}$ and any $z \in \C^{\times}$.

The above formulas show that $\C^{\times}$ acts on $(\A^{\bullet}, d)$ by $\sCDGA$ automorphisms and
the induced action on $H^1(\A)=H^1(X)$ comes from the Hodge decomposition of $H^1(X)$. Exactly as in 
\S\ref{ss86}, we infer that all irreducible components of  $\R^i_r (\A, \theta)$ are $\C^{\times}$--invariant,
which completes the proof.
\hfill $\square$

\subsection{Proof of Theorem \ref{thm:8gl}}
\label{ss87}

Part \eqref{gl1}. 
Let $E$ be an irreducible component of  the algebraic set $\R^i_r (\A, \theta)_0$. Pick a collection of regular functions on
$\cH (\pi_1(X), \BB)$, $\{ f_i \}$, defining the affine subvariety $ \V^i_r (X, \iota)$. Theorem \ref{thm:8main} \eqref{8m2}
guarantees the existence of an open neighbourhood $U$ of $0$ in the affine space $\F (\A, \bb)$, with the property that
each global analytic function $g_i=f_i \circ \eexp_{\BB}$ vanishes on $U\cap E$. 
It follows that $g_i$ vanishes on $E$, since otherwise its zero locus, being a proper analytic subset in an irreducible variety, would have empty interior.

To prove Part \eqref{gl2}, fix $\omega$ and $i$. Set $m= \min \{ b_i (X, \eexp_{\BB}(t\cdot \omega)) \, \mid \, t\in \K \}$, where 
the $\K$--action on $\F (\A, \bb)$ comes from the positive weight decomposition of $\A^{\bullet}$, as explained in \S\ref{ss85}.

Consider the set $M:= \{ t\in \K \, \mid \, b_i (X, \eexp_{\BB}(t\cdot \omega))=m \}$. By construction,
$\K \setminus M= \{ t\in \K \, \mid \, t\cdot \omega \in  \eexp_{\BB}^{-1}(\V^i_{m+1}(X, \iota)) \}$ is a proper subset of $\K$,
defined by global analytic equations (cf. Corollary \ref{lem:redjump}\eqref{rj1}), hence discrete in $\K$. It follows that, 
for $t\in \K^{\times}$ near $0$, $ b_i (X, \eexp_{\BB}(t\cdot \omega)) =m \le b_i (X, \eexp_{\BB}(\omega))$.

According to Theorem \ref{thm:8main},  $ b_i (X, \eexp_{\BB}(t\cdot \omega)) =\beta_i (\A, t\cdot \omega)$, for 
$t\in \K^{\times}$ close enough to $0$. Finally, $\beta_i (\A, t\cdot \omega)= \beta_i (\A, \omega)$, by $\K^{\times}$--invariance
of resonance varieties (cf. \S\ref{ss85}). In conclusion, $\beta_i (\A, \omega) \le b_i (X, \eexp_{\BB}(\omega))$, as asserted.
\hfill $\square$

\subsection{Proof of Corollary \ref{cor:compactK}}
\label{ssc87}
In the case $\omega \in H^{1,0}(X)$, the missing inequality between the two Betti numbers comes from Corollary 4.2 in \cite{Dcompo},
which also yields  the second claim concerning the spectral sequence.

The equality of the two Betti numbers in the case $\omega \in H^{0,1}(X)$ follows from the fact that 
they are both invariant under complex conjugation.
\hfill $\square$

\subsection{Degenerate Farber--Novikov spectral sequences}
\label{ss88}

We start with a summary of relevant facts about the Farber--Novikov spectral sequence, following \cite[pp.~184--194]{F}.
Let $X$ be a connected, finite CW complex, with fundamental group $G$. Let $\nu :G \surj \Z$ be a group epimorphism.

Denote by $\nu_\C \in H^1(X,\C)$ the associated cohomology class, and let $(H^{\bullet}(X,\C), \nu_\C \cdot)$ be the
corresponding Aomoto complex; see \S\ref{ss84}. Let $\nu^*: \C^{\times} \hookrightarrow \T (G)$ be the algebraic map 
induced on character tori. For a fixed $i\ge 0$, set $m_i = \min \{ b_i (X, \nu^*(t))\, \mid \, t\in \C^{\times} \}$. Set
$M_i:= \{  t\in \C^{\times}\, \mid \, b_i (X, \nu^*(t))=m_i \}$. By the
same line of reasoning as in the proof of Theorem \ref{thm:8gl}\eqref{gl2}, we infer that $M_i$
is the complement of a finite set, $F_i \subseteq \C^{\times}$. Clearly, $F_i =\emptyset$ for $i>\dim X$, hence $F:= \cup_{i\ge 0} F_i$
is finite.

The Farber--Novikov spectral sequence starts at $E^i_2= H^i (H^{\bullet}(X,\C), \nu_\C \cdot)$ and converges to 
$ H^i (X, {}_{\nu^*(t)}\C)$, where $t\in \C^{\times} \setminus F$. 

{\bf Proof of Theorem \ref{thm:8deg}.} By Corollary \ref{cor:qformal}, Theorem \ref{thm:8main}\eqref{8m2} holds for $X$,
in the rank $1$ case, with $q=\infty$ and $\A^{\bullet}= (H^{\bullet}(X,\C), d=0)$. Note also that $\A^{\bullet}$ has
positive weight (equal to degree), and the associated $\C^{\times}$--action on $\F (\A, \bb) \equiv H^1(X,\C)$ is
usual scalar multiplication. 

We have to show that $\beta_i (\A, \nu_\C)= m_i$, for all $i$. Note that $\nu^* (\T (\Z))= \eexp (\C \cdot \nu_C)$. It follows that,
if $t\in \C^{\times}$ is close enough to $0$, then $b_i (X, \eexp (t\nu_\C))=m_i$, for all $i$. Due to Theorem \ref{thm:8main}\eqref{8m2}, 
we may infer that $b_i (X, \eexp (t\nu_\C))= \beta_i (\A, t \nu_\C)$, for all $i$. Finally, $ \beta_i (\A, t \nu_\C)= \beta_i (\A,  \nu_\C)$,
for all $i$, by $\C^{\times}$--invariance of resonance varieties (cf. \S\ref{ss85}).  \hfill $\square$

\medskip

\begin{ack}
The  authors are grateful to the KIAS (Seoul), and respectively to the Max-Planck-Institut f\" ur Mathematik (Bonn),
where this work was completed, for hospitality and the excellent research atmosphere.
\end{ack}

\newcommand{\arxiv}[1]
{\texttt{\href{http://arxiv.org/abs/#1}{arXiv:#1}}}

\bibliographystyle{amsplain}

\begin{thebibliography}{00}

\bibitem{Al} L.~Alaniya, 
{\em Cohomology with local coefficients of some nilmanifolds}, 
Russian Math. Surveys \textbf{54} (1999), no.~5, 1019--1020. 

\bibitem{ABC}  J.~Amor{\'o}s, M.~Burger, K.~Corlette, D.~Kotschick, 
D.~Toledo, {\em Fundamental groups of compact {K}\"ahler 
manifolds}, Math. Surveys Monogr., vol.~44, Amer. Math. Soc., 
Providence, RI, 1996. 

\bibitem{A} D.~Arapura, 
{\em Geometry of cohomology support loci for local systems  {\rm I}}, 
J. Algebraic Geom. \textbf{6} (1997), no.~3, 563--597.  

\bibitem{Be} A.~Beauville, 
{\em Annulation du $H\sp 1$ pour les fibr\' es en droites plats}, in: 
Complex algebraic varieties (Bayreuth, 1990), pp.~1--15,
Lecture Notes in Math., vol.~1507, Springer, Berlin, 1992. 

\bibitem{BS} R.~Body, M.~Mimura, H.~Shiga, D.~Sullivan,
{\em $p$--universal spaces and rational homotopy types},
Comment. Math. Helv. \textbf{73} (1998), 427--442.

\bibitem{EB} E.~Borel,
{\em Sur les z\' eros des fonctions enti\`eres},
Acta Math. \textbf{20} (1897), no.~1, 357--396.

\bibitem{BG} A.~K.~Bousfield, V.~K.~A.~M. Gugenheim,
{\em On PL De Rham theory and rational homotopy type},
Memoirs Amer. Math. Soc., vol.~8, no.~179, Amer. Math. Soc.,
Providence, RI, 1976. 

\bibitem{B} E.~Brieskorn,
{\em Sur les groupes de tresses}, in:
S\'{e}minaire Bourbaki 1971/72,  pp.~21--44, Lecture Notes in Math., 
vol.~317, Springer-Verlag, 1973. 

\bibitem{Br}  K.~S.~Brown,
{\em Cohomology of groups}, Grad. Texts in Math., 
vol.~87, Springer-Verlag, New York-Berlin, 1982.

\bibitem{Bu} N.~Budur,
{\em Complements and higher resonance varieties of hyperplane arrangements},
Math. Res. Lett. \textbf{18} (2011), no.~5, 859--873.

\bibitem{BW} N.~Budur, B.~Wang,
{\em Cohomology jump loci of quasi-projective varieties},
preprint {\tt arXiv:1211.3766}.


\bibitem{CE} H.~Cartan, S.~Eilenberg,
{\em Homological algebra}, Princeton Math. Series, vol.~9,
Princeton Univ. Press, Princeton, 1956.
 
\bibitem{CP} B.~Cenkl, R.~Porter,
{\em Malcev's completion of a group and differential forms},
J. Differential Geom. \textbf{15} (1980), 531--542.

\bibitem{C}  K.-T.~Chen,
{\em Extension of $C\sp{\infty}$ function algebra by integrals
and {M}alcev completion of $\pi\sb{1}$}, Advances in Math.
\textbf{23} (1977), no.~2, 181--210.  

\bibitem{De2}
P. Deligne,  {\em Th\'eorie de Hodge II}, Publ. Math. IHES,\textbf{ 40} (1971), 5--58.

\bibitem{DGMS}  P.~Deligne, P.~Griffiths, J.~Morgan, D.~Sullivan,
{\em Real homotopy theory of {K}\"{a}hler manifolds},
Invent. Math. \textbf{29} (1975), no.~3, 245--274.

\bibitem{Dcompo} A.~Dimca, 
{\em Characteristic varieties and logarithmic differential 1-forms}, 
Compositio Math. \textbf{146} (2010),  129--144.

\bibitem{DP} A.~Dimca, S.~Papadima,
{\em Finite Galois covers, cohomology jump loci, formality properties, and multinets}, 
Annali Scuola Norm. Sup. Pisa \textbf{10}  (2011),  253--268.

\bibitem{DPS} A.~Dimca, S.~Papadima, A.~Suciu,
{\em Topology and geometry of cohomology jump loci}, 
Duke Math. Journal \textbf{148} (2009), no.~3, 405--457.


\bibitem{ESV} H.~Esnault, V.~Schechtman, E.~Viehweg,
{\em Cohomology of local systems of the complement of 
hyperplanes}, Invent. Math. \textbf{109} (1992), no.~3, 
557--561; Erratum, ibid. \textbf{112} (1993), no.~2, 447.

\bibitem{Fa} M.~Falk,
{\em Arrangements and cohomology},
Ann. Combin. \textbf{1} (1997), no.~2, 135--157.  

\bibitem{F1} M.~Farber,
{\em Topology of closed $1$--forms and their critical points},
Topology \textbf{40} (2001), 235--258.

\bibitem{F} M.~Farber, 
{\em Topology of closed one-forms}, 
Math. Surveys Monogr., vol.~108, Amer. Math. Soc., 
Providence, RI, 2004. 

\bibitem{FM} W.~Fulton, R.~MacPherson,
{\em A compactification of configuration spaces},
Annals of Math. \textbf{139} (1994), 183--225.

\bibitem{GM} W.~Goldman, J.~Millson,
{\em The deformation theory of representations of fundamental 
groups of compact {K}\"{a}hler manifolds}, Inst. Hautes \'{E}tudes 
Sci. Publ. Math. \textbf{67} (1988), 43--96.

\bibitem{GT} A.~Gomez Tato,
{\em Th\' eorie de Sullivan pour la cohomologie \` a coefficients locaux},
Trans. Amer. Math. Soc. \textbf{330} (1992), 295--305.

\bibitem{GL} M.~Green, R.~Lazarsfeld,
{\em Deformation theory, generic vanishing theorems, and some conjectures of
Enriques, Catanese and Beauville}, Invent. Math. \textbf{90} (1987), 389--407.

\bibitem{Ha1} R.~Hain,
{\em The geometry of the Mixed Hodge Structure on the fundamental group},
Proc. Symposia Pure Math. \textbf{46} (1987), 247--282.

\bibitem{Ha} R.~Hain, 
{\em Completions of mapping class groups and the cycle $C-C^{-}$},
in: Mapping class groups and moduli spaces of Riemann surfaces, pp.~75--105,
Contemp. Math., vol.~150, Amer. Math. Soc., Providence, RI, 1993. 

\bibitem{HM} R.~Hain, M.~Matsumoto,
{\em Weighted completion of Galois groups and Galois actions on the fundamental group of
$\PP^1 \setminus \{ 0,1,\infty \}$}, Compositio Math. \textbf{139} (2003), 119--167.

\bibitem{H} S.~Halperin,
{\em Lectures on minimal models},
M\' em. Soc. Math. France, vol.~230, 1983.

\bibitem{Hat} A.~Hattori,
{\em Spectral sequence in the De Rham cohomology of fibre bundles},
J. Fac. Sci. Univ. Tokyo \textbf{8} (1960), 289--331.

\bibitem{HMR} P.~Hilton, G.~Mislin, J.~Roitberg,
{\em Localization of nilpotent groups and spaces},
North-Holland Math. Studies, vol.~15, North-Holland,
Amsterdam, 1975.

\bibitem{HS} P.~J.~Hilton, U.~Stammbach,
{\em A course in homological algebra},
Grad. Texts in Math., vol.~4, Springer-Verlag, New York, 1971.

\bibitem{KM} M.~Kapovich, J.~Millson, 
{\em On representation varieties of {A}rtin groups, projective 
arrangements and the fundamental groups of smooth complex 
algebraic varieties}, Inst. Hautes \'{E}tudes Sci. Publ. Math.
\textbf{88} (1998), 5--95. 

\bibitem{Kas} H.~Kasuya,
{\em De Rham and Dolbeault cohomology of solvmanifolds with
local systems}, preprint {\tt arXiv:1207.3988}.

\bibitem{KP} T.~Kohno, A.~Pajitnov,
{\em Twisted Novikov homology and jump loci in formal and
hyperformal spaces}, preprint {\tt arXiv:1302.6785}.

\bibitem{L}  M.~Lazard,
{\em Sur les groupes nilpotents et les anneaux de {L}ie},
Ann. Sci. \'{E}cole Norm. Sup. \textbf{71} (1954), 101--190. 

\bibitem{LY} A.~Libgober, S.~Yuzvinsky,
{\em Cohomology of {O}rlik--{S}olomon algebras and local 
systems},  Compositio Math. \textbf{21} (2000), no.~3, 337--361.  

\bibitem{LM} A.~Lubotzky, A.~R.~Magid,
{\em Varieties of representations of finitely generated groups},
Memoirs Amer. Math. Soc., vol.~58, no.~336, Amer. Math. Soc.,
Providence, RI, 1985.  

\bibitem{Mac}  A.~Macinic,
{\em Cohomology rings and formality properties of nilpotent groups},
J. Pure Appl. Algebra \textbf{214} (2010), 1818--1826.

\bibitem{MaP} A.~Macinic, S.~Papadima,
{\em Characteristic varieties of nilpotent groups and 
applications}, in: Proceedings of the Sixth Congress 
of Romanian Mathematicians (Bucharest, 2007), 
pp.~57--64, vol.~1, Romanian Academy, Bucharest, 2009.

\bibitem{Mal} A.~L.~Malcev,
{\em Nilpotent groups without torsion},
Izvest. Akad. Nauk SSSR, ser. Math. \textbf{13} (1949), 201--212.

\bibitem{MP} M.~Markl, S.~Papadima,
{\em Homotopy {L}ie algebras and fundamental groups via 
deformation theory}, Annales Inst. Fourier \textbf{42} (1992), 
no.~4, 905--935.  

\bibitem{Mil} D.~V.~Millionschikov,
{\em Cohomology of solvmanifolds with local coefficients and problems of the Morse-Novikov theory},
Russian Math. Surveys \textbf{57} (2002), 813--814.

\bibitem{M}  J.~W.~Morgan,
{\em The algebraic topology of smooth algebraic varieties},
Inst. Hautes \'{E}tudes Sci. Publ. Math. \textbf{48} (1978), 137--204.

\bibitem{OS} P.~Orlik, L.~Solomon,
{\em Combinatorics and topology of complements of
hyperplanes}, Invent. Math. \textbf{56} (1980), 167--189.

\bibitem{OT} P.~Orlik, H.~Terao,
{\em Arrangements of hyperplanes}, Grundlehren Math. Wiss., vol.~300,
Springer-Verlag, Berlin, 1992.

\bibitem{PS} S.~Papadima, A.~Suciu,
{\em The spectral sequence of an equivariant chain 
complex and homology with local coefficients}, 
Trans. Amer. Math. Soc. \textbf{362} (2010), 
no.~5, 2685--2721.


\bibitem{PS1} S.~Papadima, A.~Suciu,
{\em Bieri--{N}eumann--{S}trebel--{R}enz invariants and 
homology jumping loci}, Proc. London Math. Soc. \textbf{100} 
(2010), no.~3, 795--834.

\bibitem{P} P.~F.~Pickel,
{\em Rational cohomology of nilpotent groups and Lie algebras},
Communications in Alg. \textbf{6} (1978), 409--419.

\bibitem{Q}  D.~Quillen,
{\em Rational homotopy theory}, Annals of Math. 
\textbf{90} (1969), 205--295.

\bibitem{STV} V. ~Schechtman, H. ~Terao, A. ~Varchenko, 
{\em Local systems over complements of hyperplanes and the
Kac-Kazhdan condition for singular vectors},
J. Pure Appl. Algebra, \textbf{100} (1995), 93--102. 

\bibitem{S}  D.~Sullivan,
{\em Infinitesimal computations in topology}, Inst. Hautes 
\'{E}tudes Sci. Publ. Math. \textbf{47} (1977), 269--331.

\bibitem{Sw} R.~M.~Switzer,
{\em Algebraic Topology--homotopy and homology},
Grundlehren Math. Wiss., vol.~212,
Springer-Verlag, Berlin, 1975.

\bibitem{T} J.~C.~Tougeron,
{\em Id\'{e}aux de fonctions diff\'{e}rentiables}, Ergebnisse 
der Math., vol.~71, Springer-Verlag, Berlin-New York, 1972. 

\bibitem{Vbook}  C.~Voisin, 
{\em Hodge Theory and Complex Algebraic Geometry I}, Cambridge Studies in Advanced Math.
\textbf{76}, CUP, 2002.

\bibitem{V}  C.~Voisin, 
{\em On the homotopy types of compact K\"ahler and complex projective manifolds}, 
Inventiones Math. \textbf{157} (2004), 329--343.

\bibitem{Wa} B.~Wang,
{\em Cohomology jump loci of compact K\"ahler manifolds},
preprint {\tt arXiv:1303.6937}.


\bibitem{W}  G.~W.~Whitehead,
{\em Elements of homotopy theory}, Grad. Texts in Math., 
vol.~61, Springer-Verlag, New York, 1978.  


\end{thebibliography}

\end{document}